\documentclass[a4paper,10pt]{article}
\usepackage{graphicx}
\usepackage{amsmath}
\usepackage{amsfonts}
\usepackage{amssymb}
\usepackage{amsthm}
\usepackage{subfig}

\newtheorem{theorem}{Theorem}[section]
\newtheorem{lemma}{Lemma}[section]
\newtheorem{assumption}{Assumption}[section]
\newtheorem{proposition}{Proposition}[section]

\newtheorem{corollary}{Corollary}[section]
\newtheorem{algorithm}{Algorithm}[section]

\newcommand{\Label}[1]{\label{#1}{{\mbox{$\;$\fbox{\tiny\tt #1}$\;$}}}}
\renewcommand{\Label}[1]{\label{#1}}

\newcommand{\lc}{\mathrel{\raise2pt\hbox{${\mathop<\limits_{\raise1pt\hbox{\mbox{$\sim$}}}}$}}}

\begin{document}
\title{Convergence Rate  and Quasi-Optimal Complexity of Adaptive Finite Element
Computations for Multiple Eigenvalues
\thanks{This work was partially
supported by the Funds for Creative Research Groups of China under Grant 11021101, the National Basic Research Program of China under Grant 2011CB309703,
 the National Science Foundation of China under  Grants 11101416 and 91330202,
and the National Center for Mathematics and Interdisciplinary Sciences, Chinese Academy of Sciences.}}
\author{
Xiaoying Dai\thanks{LSEC, Institute of Computational Mathematics and Scientific/Engineering Computing, Academy of Mathematics and Systems Science,
Chinese Academy of Sciences, Beijing 100190, China ({daixy@lsec.cc.ac.cn}).}
 \and  Lianhua He\thanks{LSEC, Institute of
Computational Mathematics and Scientific/Engineering Computing, Academy of Mathematics and Systems Science, Chinese Academy of Sciences, Beijing 100190,
China ({helh@lsec.cc.ac.cn}).}
 \and Aihui Zhou\thanks{LSEC, Institute of
Computational Mathematics and Scientific/Engineering Computing, Academy of Mathematics and Systems Science, Chinese Academy of Sciences, Beijing 100190,
China ({azhou@lsec.cc.ac.cn}).}  }

\date{}
\maketitle

\date{}

\maketitle
\begin{abstract}
In this paper, we study an adaptive finite element method for multiple eigenvalue problems of a class of second order elliptic equations. By using some eigenspace approximation technology and its crucial
property which is also presented in this paper,
we extend the results in \cite{dai-xu-zhou08} to multiple eigenvalue problems,
we obtain  both  convergence rate and quasi-optimal complexity of the adaptive finite element eigenvalue approximation.

 \end{abstract}



\section{Introduction}\setcounter{equation}{0}


Adaptive finite element computation is  efficient in solving partial differential equations and
has been successfully used in scientific and engineering computing. Its numerical analysis has been also derived
much attention from the mathematical community.
Since Babu{\v s}ka and Vogelius \cite{babuska-vogelius84} gave an analysis of an adaptive finite element method (AFEM) for linear
symmetric elliptic problems in one dimension, there has been much investigation on the convergence and complexity of AFEMs in
literature (see, e.g., \cite{binev-dehmen-devore-04,cascon-kreuzer-nochetto-siebert08,dai-xu-zhou08,dorfler96,
garau-morin-zuppa09,stevenson06,stevenson08} and the references
cited therein).
In the context of the finite element approximations of eigenvalue problems, in particular, we note that there are  a number of  works concerning a posteriori error estimates  \cite{becker-r01,duran-03,h-rannacher01,larson00,mao-shen-zhou06,verfurth96}, AFEM
convergence \cite{dai-xu-zhou08,garau-morin11,garau-morin-zuppa09,giani-graham09,he-zhou11} and complexity \cite{dai-xu-zhou08,garau-morin11,he-zhou11}. Except for the convergence
analysis in \cite{garau-morin-zuppa09}, to our best knowledge, there is no any work about convergence
rate and complexity of AFEM for multiple eigenvalue problems. The purpose of this paper is to fill in the gap.

We understand that multiple eigenvalue problems are topic in science  and engineering, such as Hartree-Fock equation  and Kohn-Sham equation used to model ground state electronic structures of molecular systems in quantum chemistry and materials science, in which hundreds of thousands of eigenvalues and their corresponding  eigenfunctions are desired, and among these eigenvalues, most are multiple \cite{dai-zhou11, bris03,  martin04,saad10,shen-zhou06}. While the central computation in solving either Hartree-Fock equation or Kohn-Sham equation is the repeated solution of linear Schr{\" o}dinger type equation, of which adaptive finite element analysis and computation are significant. Hence, we want to study the convergence rate  and complexity of AFEMs for multiple eigenvalue problems and  focus on the following elliptic eigenvalue problems: find $\lambda\in \mathbb{R}$ and $u\in H^1_0(\Omega)$ such
that
\begin{eqnarray*}\Label{elliptic-problem}
\left\{\begin{array}{rl}
-\nabla\cdot (A\nabla u) +cu &= \lambda u \quad \mbox{in} \quad \Omega,\\
\|u\|_{0,\Omega} &= 1\quad\mbox{on}~~\partial\Omega,\end{array}\right.
\end{eqnarray*}
where
$A, c$ are coefficients stated precisely in Section \ref{preliminary}.

We see that the analysis technologies for the convergence rate and complexity of AFEM in literature are only valid for simple eigenvalues and their corresponding eigenfunctions, it can not be applied  directly to multiple eigenvalue
cases. The difficulty lies in that in context of  multiple eigenvalue cases, it is not practicable to figure out the discreted eigenfunctions obtained over different  meshes so as to approximate
the same exact eigenfunctions. As a result, the standard technology of measuring  the error of every eigenfunction does not work well any more,  which results in
the difficulty when analyzing the reduction for error of the approximate eigenfunction over two consecutive meshes.
Instead, we employ the gap between the eigenfunction space and its approximation, which seems natural but requires  some delicate technical tools in analysis. To carry out the analysis of the eigenspaces and their approximations,  in this paper, we introduce a system of some source problems associated with the multiple  eigenvalue problem, for which we also need to generalize the existing results of adaptive finite element approximations of scale problems to a setting of vector version.
By using the similar perturbation argument in
\cite{dai-xu-zhou08,he-zhou11} (see Theorem \ref {thm-eigen-boundary} and Lemma \ref{lemma-bound-eigen}) together
 with eigenfunction space approximation technology and its crucial property (see Lemma \ref{eta-etah} and
 Lemma \ref{mark-eta-etah}) that is also shown in this paper,
  we obtain the convergence rate and quasi-complexity of AFEM for multiple eigenvalue problems.

Now let us give somewhat more detailed but informal description  the main results in this paper. We propose and analyze an adaptive finite element algorithm for multiple eigenvalue problem, Algorithm \ref{algorithm-AFEM-eigen},
which is based on the residual type a posteriori error estimators also designed in  this paper, and prove that, for instance
\begin{itemize}
\item  Under some mild assumption,  the gap, $\delta_{H_0^1(\Omega)}(M(\lambda), M_h(\lambda))$, between the continue eigenspace $M(\lambda)$ and its finite element approximation $M_{h}(\lambda)$ has the following a posteriori estimates (see Theorem \ref{thm-error-estimator-space})
\begin{eqnarray*}
 \delta_{H_0^1(\Omega)}(M(\lambda), M_h(\lambda))   &\lc& \eta_h(U_h, \Omega)\nonumber\\
    \eta^2_h(U_h, \Omega)-
    osc^2_h (U_h, \Omega) &\lc &
  \delta_{H_0^1(\Omega)}^2(M(\lambda), M_h(\lambda)).
 \end{eqnarray*}
\item  Under some reasonable assumptions,  the adaptive finite element approximation eigenspaces will converge to the exact eigenspace with some convergence rate, as shown as follows (see Theorem \ref{thm-convergence-rate-eigenspace})
\begin{eqnarray*}\label{main-result-conc}
\delta^2_{H_0^1(\Omega)}(M(\lambda), M_{h_k}(\lambda)) \lc \alpha^{2k},
  \end{eqnarray*}
where $\alpha \in (0, 1)$ is some constant. Furthermore, if the marked sets are of minimal cardinality, thenthe adaptive finite element approximation eigenspaces have a quasi-optimal complexity as follows (see Theorem \ref{thm-optimal-complexity-eigenspace})
\begin{eqnarray*}
  \delta^2_{H_0^1(\Omega)}(M(\lambda), M_{h_n}(\lambda)) \lc (\#\mathcal{T}_{h_n}
  -\#\mathcal{T}_{h_0})^{-2s}.
  \end{eqnarray*}
\end{itemize}

The paper is organized as follows. In the next section, we shall describe some basic notation and review the existing
results of finite element approximations for a class of linear
second order elliptic source and eigenvalue problems, which will be used in our analysis. In Section \ref{adaptive-algorithm},  we construct the a posteriori error estimators for finite element eigenvalue problems from the relationship between the elliptic eigenvalue approximation and the associated boundary value approximation
 and then design adaptive finite element algorithm for the elliptic eigenvalue problems.
We analyze the
convergence rate quasi-optimal complexity of the adaptive finite element eigenvalue computations in Sections \ref{convergence-sec} and \ref{complexity}, respectively. We present several numerical examples in Section \ref{numerical-experiments} to support  our theory. Finally, we remark how our main results can be expected for computing the first $N$ eigenvalues with both simple and multiple eigenvalues are included, the Steklov eigenvalue problems,  and inexact
numerical solutions.

\section{Preliminaries} \Label{preliminary}\setcounter{equation}{0}
Let $\Omega\subset \mathbb{R}^d(d\ge 1)$ be a polytopic bounded domain. We shall use the standard notation for Sobolev spaces $W^{s,p}(\Omega)$ and their
associated norms and seminorms, see, e.g., \cite{ada,cls}. For $p=2$, we denote $H^s(\Omega)=W^{s,2}(\Omega)$ and $H^1_0(\Omega)=\{v\in H^1(\Omega):
 v\mid_{\partial\Omega}=0\}$, where $v\mid_{\partial\Omega}=0$ is understood in
the sense of trace, $\|\cdot\|_{s,\Omega}= \|\cdot\|_{s,2,\Omega}$, and $(\cdot,\cdot)$ is the standard $L^2$ inner product.
 Throughout this paper, we shall
use $C$ to denote a generic positive constant which may stand for different values at its different occurrences.  For convenience, the symbol $\lc$ will
be used in this paper. The notation that $A\lc B$
 means that $A\le C_1 B$, and the notation $A \cong B$ means $C_2 A \le B \le C_3 A$, where $C_1$, $C_2$, $C_3$ are  some constants  that are
 independent of mesh parameters. All the constants involved are
independent of mesh sizes.

Let $\{\mathcal{T}_h\}$ be a shape regular family of nested conforming meshes over $\Omega$: there exists a constant $\gamma^{\ast}$ such that
\begin{eqnarray*}
\frac{h_T}{\rho_T} \leq \gamma^{\ast} ~~~ \forall T \in \bigcup_h\mathcal{T}_h,
\end{eqnarray*}
where, for each $T\in \mathcal{T}_h$, $h_T$ is the diameter of $T$, and $\rho_T$ is the diameter of the biggest ball contained in $T, h=\max\{h_T: T\in
\mathcal{T}_h\}$. Let $\mathcal{E}_h$ denote the set of interior faces (edges or sides) of $\mathcal{T}_h$.

Let $S^{h,k}(\Omega)$ be a space of continuous functions on $\Omega$ such that for $v\in S^{h,k}(\Omega)$, $v$ restricted to each $T$ is a polynomial of
degree not greater than $k$, namely,
\begin{eqnarray*}
S^{h,k}(\Omega)=\{v\in C(\bar{\Omega}): ~v|_{T}\in P^k_{T} \quad\forall T \in \mathcal{T}_h\},
\end{eqnarray*}
where $P^k_{T}$ is the space of polynomials of degree not greater than a positive integer $k$. Set $S^{h,k}_0(\Omega)=S^{h,k}(\Omega)\cap H^1_0(\Omega)$.
We shall denote $S^{h,k}_0(\Omega)$ by $S^{h}_0(\Omega)$ for simplification of notation afterwards.

\subsection{A linear elliptic boundary value problem}\Label{sc:linear boundary}
In this subsection, we shall present some basic properties of a second order elliptic boundary value problem for vector version and its finite element approximations, which is just the simple extension of the existed results  for scalar version.
These properties will be used in our analysis in the following sections.

Consider the homogeneous boundary value problem:
\begin{equation}\Label{problem}
\left\{\begin{array}{rl}
L u_i &= f_i \,\,\, \mbox{in} \quad \Omega, ~~i=1, \cdots, N,\\[0.2cm]
u_i &= 0\,\,\,\mbox{on}~~\partial\Omega, \end{array}\right.
\end{equation}
where $N$ is a positive integer, $L$ is a linear second order elliptic operator:
\[ Lu =-\nabla \cdot(A\nabla u)+cu\] with $A: \Omega\to \mathbb{R}^{d\times d}$ being piecewise Lipschitz
over initial triangulation  and symmetric positive definite with smallest eigenvalue uniformly bounded away from 0, and $0\le c\in L^{\infty}(\Omega)$.

The weak form of ~(\ref{problem}) reads: find $U\equiv(u_1,\cdots,u_N)\in  (H^1_0(\Omega))^N$ such that
\begin{eqnarray}\Label{variation}
  a(u_i, v_i)=   b(f_i, v_i) \qquad\forall v_i\in H^1_0(\Omega), i=1,\cdots, N,
\end{eqnarray}
where
\begin{eqnarray*}
 a(u, v)=(A\nabla u,\nabla v)+(cu,v)\quad\mbox{and}\quad
 b(u,v)= (u,v).
\end{eqnarray*}
We observe that $a(\cdot,\cdot)$ is a bounded bilinear form over $H^1_0(\Omega)$:
\begin{eqnarray*}
|a(w, v)| \leq C_a \|w\|_{1, \Omega} \|v\|_{1, \Omega} ~~\forall w, v\in H^1_0(\Omega)
\end{eqnarray*}
and for energy norm $\|\cdot\|_{a,\Omega}$, which is defined by $\|w\|_{a,\Omega}=\sqrt{a(w,w)}$, there hold
\begin{eqnarray*}
 c_a \|w\|_{1, \Omega} \leq \|w\|_{a,\Omega} \leq C_a \|w\|_{1, \Omega}
 ~~\forall w\in H^1_0(\Omega),
\end{eqnarray*}
where $c_a$ and $C_a$ are positive constants. We understand that \eqref{variation} is uniquely solvable for any $f_i\in H^{-1}(\Omega)(i=1,\cdots,N)$.

For  $U = (u_1, \cdots, u_N)  \in (H^1_0(\Omega))^N$,   we
shall denote by
\begin{eqnarray*}\label{def-norm-abuse}
\|U\|_{a,\Omega}=\left(\sum_{i=1}^{N}\|u_i\|_{a,\Omega}^2\right)^{1/2}, ~~\mbox{and} ~~ \|U\|_{1,\Omega}=\left(\sum_{i=1}^{N}\|u_i\|_{1,\Omega}^2\right)^{1/2}.
\end{eqnarray*}

For $L^2(\Omega)$ with $\|\cdot\|_{b,\Omega}=\sqrt{b(\cdot,\cdot)}$, we see that there is a unique compact operator  $K: L^2(\Omega) \rightarrow H^1_0(\Omega)$ satisfying
\begin{eqnarray}\label{def-k}
 a(Kw, v) = b(w, v)~~~~\forall w\in L^2(\Omega), v \in H^1_0(\Omega).
\end{eqnarray}

Define the Galerkin-projection $R_h: H^1_0(\Omega) \to V_h\equiv S^h_0(\Omega)$ by
\begin{eqnarray}\Label{Gprojection}
a(u-R_h u, v) =0 \quad\forall u\in H^1_0(\Omega) ~\forall  v\in V_h,
\end{eqnarray}
and apparently
\begin{equation*}\Label{stable}
      \|R_hu\|_{1,\Omega}\lc \|u\|_{1,\Omega} \quad\forall u\in H^1_0(\Omega).
\end{equation*}
If we define the operator $K_h: L^2(\Omega) \rightarrow S^h_0(\Omega)$ as follows:
\begin{eqnarray}\label{def-kh}
 a(K_h w, v) = b(w, v)~~~~\forall w\in L^2(\Omega), v \in S^h_0(\Omega),
\end{eqnarray}
then
\begin{eqnarray*}
K_h = R_h K.
\end{eqnarray*}

The following results can be found in \cite{babuska-osborn89,xuzhou00}.
\begin{proposition}\Label{prop2.1}
Let
\begin{eqnarray*}
\rho_{_{\Omega}}(h)=\sup_{ f\in L^2(\Omega),\|f\|_{b,\Omega}=1} \inf_{v\in V_h}\|Kf-v\|_{1,\Omega}.
\end{eqnarray*}
Then $\rho_{_{\Omega}}(h)\to 0$ as $h\to 0$ and
\begin{eqnarray*}\Label{L2rho}
\|u-R_hu\|_{b,\Omega}\lc\rho_{_\Omega}(h)\|u-R_hu\|_{1,\Omega} \quad\forall u\in H^1_0(\Omega).
\end{eqnarray*}
\end{proposition}

 A standard finite element scheme for (\ref{variation}) reads: find
 $U_h = (u_{1, h}, \cdots, u_{N,h}) \in (V_h)^{N}$ satisfying
 \begin{eqnarray}\label{dis-fem}
    a(u_{i,h}, v_i) =  b(f_i, v_i)    ~~ \forall  v_i \in V_h, i=1,\cdots,N.
\end{eqnarray}

 Let $\mathbb{T}$ denote the class of all conforming refinements by bisection of $\mathcal{T}_{h_0}$. For
$\mathcal{T}_h \in \mathbb{T}$, define the element residual $\tilde{\mathcal{R}}_T(u_{i,h})$ and the jump residual $\tilde{J}_E(u_{i,h})$  by
 \begin{eqnarray*}\label{residual}
  \tilde{\mathcal{R}}_T(u_{i,h}) &=& f_i-L u_{i,h} = f_i +\nabla\cdot(A\nabla u_{i,h})-c u_{i,h}\quad \mbox{in}~ T\in
  \mathcal{T}_h,\\
  \tilde{J}_E(u_{i,h}) &=& -A \nabla u_{i,h}^{+}\cdot \nu^{+} - A \nabla u_{i,h}^{-}\cdot
  \nu^{-} = [[A\nabla u_{i,h}]]_E \cdot \nu_E  ~~~~ \mbox{on}~ E\in \mathcal{E}_h,
\end{eqnarray*}
where $E$ is the  common side of elements $T^+$ and $T^-$ with unit outward normals $\nu^+$ and $\nu^-$, respectively, and $\nu_E=\nu^-$. Let $\omega_T$
be the union of elements sharing a side with $T$ and $\omega_E$ be the union of elements which shares the side $E$, that is, $\omega_E = T^+ \cap T^- .$

For $T\in \mathcal{T}_h$, we define the local error indicator
  $\tilde{\eta}_h(u_{i,h}, T)$ by
  \begin{eqnarray*}\label{error-indicator}
   \tilde{\eta}^2_h(u_{i,h}, T) =h_T^2\|\tilde{\mathcal{R}}_T(u_{i,h})\|_{0,T}^2
   + \sum_{E\in \mathcal{E}_h,E\subset\partial T
   } h_E \|\tilde{J}_E(u_{i,h})\|_{0,E}^2
  \end{eqnarray*}
and the oscillation $\widetilde{osc}_h(u_{i,h},T)$ by
\begin{eqnarray*}\label{local-oscillation}
\widetilde{osc}^2_h(u_{i,h},T) = h_T^2\|\tilde{\mathcal{R}}_{T}(u_{i,h})-\overline{\tilde{\mathcal{R}}_{T}(u_{i,h})}\|_{0,T}^2
   + \sum_{E\in \mathcal{E}_h,E\subset\partial T
   } h_E \|\tilde{J}_E(u_{i,h})-\overline{\tilde{J}_E(u_{i,h})}\|_{0,E}^2,
\end{eqnarray*}
where $h_E$ is the diameter of $E$, $\overline{w}$ is the $L^2$-projection of $w\in L^2(\Omega)$ to polynomials of some degree  on $T$ or $E$.

We define the error
 estimator $\tilde{\eta}_h(u_{i,h}, \Omega)$ and the oscillation  $\widetilde{osc}_h(u_{i,h},\Omega)$ by
 \begin{eqnarray*}\label{error-estimator}
  \tilde{\eta}^2_h(u_{i,h}, \Omega) =  \sum_{T\in \mathcal{T}_h, T \subset \Omega}
  \tilde{\eta}^2_h(u_{i,h}, T) \quad \textnormal{and} \quad
\widetilde{osc}^2_h(u_{i,h}, \Omega) = \sum_{T\in \mathcal{T}_h, T \subset \Omega}
  \widetilde{osc}^2_h(u_{i,h}, T).
  \end{eqnarray*}
For any $U_h=(u_{1,h},\cdots,u_{N,h})\in (V_h)^N$, we set
\begin{eqnarray*}
\tilde{\eta}^2_h(U_h, T)=\sum_{i=1}^N\tilde{\eta}^2_h(u_{i,h}, T) \quad \textnormal{and} \quad \widetilde{osc}^2_h(U_h,
T)=\sum_{i=1}^N\widetilde{osc}^2_h(u_{i,h}, T), ~~ \forall T\in \mathcal{T}_h,
\end{eqnarray*}
and
\begin{eqnarray*}
\tilde{\eta}^2_h(U_h, \Omega)=\sum_{T\in \mathcal{T}_h} \tilde{\eta}^2_h(U_h, T) \quad \textnormal{and} \quad \widetilde{osc}^2_h(U_h,
\Omega)=\sum_{T\in \mathcal{T}_h} \widetilde{osc}^2_h(U_h, \Omega).
\end{eqnarray*}

In our analysis we need the following result\cite{cascon-kreuzer-nochetto-siebert08}.
\begin{lemma}\label{lemma-osc-L}
  There exists a constant $C_{\ast}$ which depends on $A$, regularity constant $\gamma^{\ast}$, and coefficient c,
   such that
  \begin{eqnarray}\label{lemma-osc-L-neq}
    \widetilde{osc}_h(V,T)\leq \widetilde{osc}_h(W,T)+C_{\ast}\|V-W\|_{1,\omega_T},
\quad\forall V,W\in (V_h)^N, ~\forall T \in \mathcal{T}_h.
  \end{eqnarray}
\end{lemma}

We have the standard  a posteriori error estimates for the finite element approximation of boundary value problems (\ref{problem}) as follows (c.f., e.g.,
\cite{mekchay-nochetto05,morin-nochetto-siebert-02,verfurth96})
\begin{eqnarray}\label{boundary-upper}
  \|u_i - u_{i,h} \|_{a,\Omega} \leq \tilde{C}_1 \tilde{\eta}_h (u_{i,h}, \Omega),
\end{eqnarray}
\begin{eqnarray}\label{boundary-lower}
~~~~~~ \tilde{C}^2_2 \tilde{\eta}^2_h (u_{i,h}, \Omega)- \tilde{C}^2_3 \widetilde{osc}^2_h(u_{i,h}, \Omega) \le \|u_i-u_{i,h}\|_{a,\Omega}^2,
\end{eqnarray}
where $\tilde{C}_1, \tilde{C}_2$ and $ \tilde{C}_3$ are positive constants depending
 on the shape regularity of the mesh $\mathcal {T}_h$.

The adaptive algorithm with {\bf D\"{o}rfler marking strategy} for solving (\ref{dis-fem}) can be stated as follows (c.f. \cite{cascon-kreuzer-nochetto-siebert08}):
\vskip 0.1cm
\begin{algorithm}\label{algorithm-AFEM-bvp}~
 Choose a parameter $0 < \theta <1.$
\begin{enumerate}
\item Pick a given mesh $\mathcal{T}_0$, and let $k=0$.
\item Solve the system (\ref{dis-fem}) on $\mathcal{T}_k$ to get the discrete solution $U_{h_k}\equiv(u_{1,h_k},\cdots, u_{N,h_k})$.
\item Compute local error indictors $\tilde{\eta}_{h_k}(U_{h_k},\tau)$ for all $\tau\in \mathcal{T}_{h_k}$.
\item Construct $\mathcal{M}_{h_k} \subset \mathcal{T}_{h_k}$ by {\bf D\"{o}rfler marking strategy} with parameter
 $\theta$.
\item Refine $\mathcal{T}_{h_k}$ to get a new conforming mesh $\mathcal{T}_{h_{k+1}}$ by Procedure {\bf REFINE}.
\item Let $k=k+1$ and go to 2.
\end{enumerate}
\end{algorithm}
\vskip 0.1cm

{\bf D\"{o}rfler marking strategy}
 in  Algorithm \ref{algorithm-AFEM-bvp} was introduced in \cite{dorfler96,morin-nochetto-siebert-02} when $N=1$. It is used to enforce error reduction  and can be defined as
follows.
\begin{center}
\begin{tabular}{|p{120mm}|}\hline
\begin{center}
{\bf D\"{o}rfler marking strategy}
 \end{center}\\
 Given a parameter $0<\theta < 1$.
\begin{enumerate}
 \item Construct a   subset $\mathcal{M}_{h_k}$ of
 $\mathcal{T}_{h_k}$ by selecting some elements
 in $\mathcal{T}_{h_k}$ such that
\begin{eqnarray*}
 \sum_{T\in \mathcal{M}_{h_k}} \tilde{\eta}_{h_k}^2(U_{h_k}, T)  \geq \theta
  \tilde{\eta}_{h_k}^2(U_{h_k},\Omega).
 \end{eqnarray*}
 \item Mark all the elements in $\mathcal{M}_{h_k}$.
 \end{enumerate}
\\
 \hline
\end{tabular}
\end{center}

As pointed out in \cite{cascon-kreuzer-nochetto-siebert08}, the Procedure {\bf REFINE} here is
 some iterative or recursive bisection (see, e.g., \cite{maubach-95}) of elements with the minimal refinement condition that marked elements are bisected at least once. Given a fixed number $b\geq 1$, for any $\mathcal{T} \in \mathbb{T}$, and a subset $\mathcal{M} \in \{\mathcal{T}_h\}$ of marked elements, $\mathcal{T}_{\ast} = \mbox{\bf REFINE}(\mathcal{T}, \mathcal{M})$ outputs a conforming mesh $\mathcal{T}_{\ast} \in \{\mathcal{T}_h\}$, where at least all elements of $\mathcal{M}$ are bisected $b$ times.
 Define
 \begin{eqnarray}\label{def-R}
 \mathcal{R}_{\mathcal{T}\rightarrow \mathcal{T}_{\ast}} := \mathcal{T} \setminus (\mathcal{T}_{\ast} \cap \mathcal{T}).
 \end{eqnarray}
  $\mathcal{R}_{\mathcal{T}\to \mathcal{T}_{\ast}}$ is the set of refined elements from mesh $\mathcal{T}$ to $\mathcal{T}_{\ast}$.
Obviously, we have that $\mathcal{M} \subset \mathcal{R}_{\mathcal{T}\rightarrow \mathcal{T}_{\ast}}$.

By some primary operation, we can easily extend the corresponding results of  the case  $N=1$ in \cite{cascon-kreuzer-nochetto-siebert08} to vector version as follows, which will be used in our following analysis.

 \begin{theorem}\label{convergence-boundary}
 Let $\{U_{h_k}\}_{k\in \mathbb{N}_0}$ be a sequence of finite element
 solutions of boundary problems produced by
  {\bf Algorithm \ref{algorithm-AFEM-bvp}}.
Then there exist constants ${\tilde\gamma}>0$ and  $\xi\in (0,1)$, depending only on the shape regularity of meshes, the data, and the parameters used by
{\bf Algorithm \ref{algorithm-AFEM-bvp}}, such that for any two consecutive iterates $k$ and $k+1$ we have
\begin{eqnarray*}
\|U-U_{h_{k+1}}\|^2_{a,\Omega} + \tilde{\gamma} \tilde{\eta}^2_{h_{k+1}}(U_{h_{k+1}}, \Omega)
 \leq  \xi^2 \Big( \|U-U_{h_k}\|^2_{a,\Omega} +
\tilde{\gamma} \tilde{\eta}^2_{h_{k}}(U_{h_{k}}, \Omega)\Big).
\end{eqnarray*}
Indeed, the constant ${\tilde\gamma}$ has the following form
\begin{eqnarray}\label{gamma-boundary}
\tilde{\gamma} = \frac{1}{(1 + \delta^{-1}) C_{\ast}^2}
\end{eqnarray}
with some constant $\delta\in (0,1)$.
 \end{theorem}

\begin{lemma}\label{complexity-refine}
 Assume that $\mathcal{T}_{h_0}$
verifies condition (b) of Section 4 in \cite{stevenson08}. Let $\{\mathcal{T}_{h_k}\}_{k\geq 0}$ be any sequence of refinements of $\mathcal{T}_{h_0}$
where $\mathcal{T}_{h_{k+1}}$ is generated from $\mathcal{T}_{h_k}$ by $\mathcal{T}_{h_{k+1}}={\bf REFINE}(\mathcal{T}_{h_k},\mathcal{M}_{h_k})$ with a
subset $\mathcal{M}_{h_k}\subset \mathcal{T}_{h_k}$. Then
\begin{eqnarray*}\label{complexity-estimate}
\#\mathcal{T}_{h_k}-\#\mathcal{T}_{h_0} \lc \sum_{j=0}^{k-1}\#\mathcal{M}_{h_j} \quad \forall k\geq 1
\end{eqnarray*}
is valid,  where the hidden constant depends on $\mathcal{T}_{h_0}$ and b.
Here and hereafter $\#\mathcal{T}$ means the number of elements in $\mathcal{T}$.
\end{lemma}

 \begin{lemma}\label{localized-upper-bound}
Let $u_{h_k,l} \in V_{h_k}$ and $ u_{h_{k+1},l} \in V_{h_{k+1}} (l=1,\cdots,N)$ be discrete solutions of (\ref{dis-fem}) over a conforming mesh
$\mathcal{T}_{h_k}$ and its  refinement $\mathcal{T}_{h_{k+1}}$ with marked element $\mathcal{M}_{h_{k}}$. Let $\mathcal{R} :=R_{\mathcal{T}_{h_k}\to
\mathcal{T}_{h_{k+1}}}$ be the set of refined elements, then the following localized upper bound is valid
\begin{eqnarray*}\label{localized-upper-bound-conc}
 \|U_{h_k} - U_{h_{k+1}} \|^2_{a, \Omega} \leq \tilde{C}_1^2
 \sum_{T \in \mathcal{R}}
    \tilde{\eta}^2_{h_k}(U_{h_k}, T),
\end{eqnarray*}
where $U_{h_k}\equiv (u_{h_k,1}\cdots,u_{h_k,q})$ and $U_{h_{k+1}}\equiv (u_{h_{k+1},1}\cdots,u_{h_{k+1},q})$.
\end{lemma}

 \begin{proposition}\label{complexity-boundary-optimal-marking}
Let $u_{h_k,l} \in V_{h_k}$ and $ u_{h_{k+1},l} \in V_{h_{k+1}} (l=1,\cdots,N)$ be discrete solutions of (\ref{dis-fem}) over a conforming mesh
$\mathcal{T}_{h_k}$ and its  refinement $\mathcal{T}_{h_{k+1}}$ with marked element $\mathcal{M}_{h_{k}}$. Suppose that they satisfy the energy decrease
property
\begin{eqnarray*}\label{optimal-mariking-cond}
\|U-U_{h_{k+1}}\|_{a, \Omega}^2+ \tilde{\gamma}_0 \widetilde{osc}^2_{h_{k+1}}(U_{h_{k+1}}, \Omega) \leq \tilde{\xi}_0^2 \big(\|U-U_{h_k}\|_{a, \Omega}^2
+\tilde{\gamma}_0 \widetilde{osc}^2_{h_{k}}(U_{h_{k}},\Omega)\big)
\end{eqnarray*}
with $\tilde{\gamma}_0>0$ being a constant and $\tilde{\xi}_0^2\in (0,\frac{1}{2})$.
 Then the set $\mathcal{R} :=R_{\mathcal{T}_{h_k}\to
\mathcal{T}_{h_{k+1}}}$  satisfies the D$\ddot{o}$rfler property
\begin{eqnarray*}\label{optimal-mariking-conc}
\sum_{T \in \mathcal{R}
  }  \tilde{\eta}^2_{h_k}(U_{h_k}, T) \geq \tilde{\theta} \sum_{T \in \mathcal{T}_{h_k}}
   \tilde{\eta}^2_{h_k}(U_{h_k}, T)
\end{eqnarray*}
with $\tilde{\theta} = \frac{\tilde{C}_2^2(1-2\tilde{\xi}^2_0)}{\tilde{C}_0 ( \tilde{C}_1^2 + (1 + 2 C_{\ast}^2 \tilde{C}_1^2) \tilde{\gamma}_0)}$, where
$\tilde{C}_0 = \max(1, \frac{\tilde{C}_3^2}{\tilde{\gamma}_0})$.
\end{proposition}

\subsection{A linear  eigenvalue problem}\Label{sc:linear eigenvalue}

A number $\lambda$ is called an eigenvalue of the form $a(\cdot,\cdot)$ relative to the form $b(\cdot,\cdot)$ if there is a nonzero function $0\neq u\in
H^1_0(\Omega)$, called an associated eigenfunction, satisfying
\begin{eqnarray}\Label{eigen}
a(u,v)=\lambda b(u,v)~~~~\forall v\in H^1_0(\Omega).
\end{eqnarray}

We see that (\ref{eigen}) has a countable sequence of real eigenvalues
$$
0 < \lambda_1\ <\lambda_2\le\lambda_3\le\cdots
$$
and corresponding eigenfunctions
$$
u_1, u_2, u_3,\cdots,
$$
which can be assumed to satisfy
$$
b(u_i, u_j)=\delta_{ij}, ~i,j=1,2,\cdots
$$
In the sequence $\{\lambda_j\}$, the $\lambda_j$'s are repeated according to geometric multiplicity.

The following property of eigenvalue and eigenfunction approximation is useful (see \cite{babuska-osborn89, babuska-osborn91}).
\begin{proposition}\Label{prop3.2}
Let $(\lambda, u)$ be an eigenpair of (\ref{eigen}). For any $w\in H^1_0(\Omega)\setminus\{0\}$,
\begin{eqnarray*}
\frac{a(w,w)}{b(w,w)}-\lambda=\frac{a(w-u, w-u)}{b(w,w)}- \lambda\frac{b(w-u,w-u)}{b(w,w)}.
\end{eqnarray*}
\end{proposition}

A standard finite element scheme for (\ref{eigen}) is: find a pair  $(\lambda_h, u_h)$, where $\lambda_h$ is a number and $0\not=u_h\in V_h$,
satisfying
\begin{eqnarray}\Label{fe-eigen}
a(u_h,v)=\lambda_h b(u_h,v)~~~~\forall v\in V_h.
\end{eqnarray}
Let us order the eigenvalues of  (\ref{fe-eigen}) as follows
$$
0<\lambda_{1,h} < \lambda_{2,h}\le \cdots \le\lambda_{n_h,h}, ~~n_h=\mbox{dim} ~V_h,
$$
and assume the corresponding eigenfunctions
$$
u_{1,h}, u_{2,h},\cdots, u_{n_h,h}
$$
satisfy
$$
b(u_{i,h}, u_{j,h})=\delta_{ij}, ~i,j=1,2,\cdots,n_h .
$$

As a consequence of the minimum-maximum principle (see \cite{babuska-osborn91} or \cite{chatelin83}) and Proposition \ref{prop3.2}, we have
\begin{eqnarray}\Label{eigen-err}
\lambda_i\le\lambda_{i,h}\le \lambda_i+C_i\|u_i-u_{i,h}\|^2_{a}, ~~i=1,2,\cdots, n_h.
\end{eqnarray}

Let $\lambda$ be any eigenvalue of (\ref{eigen}) with multiplicity $q$ and   $M(\lambda)$ denote the space of eigenfunctions corresponding to $\lambda$,
that is
\begin{eqnarray*}
M(\lambda)=\{w\in H^1_0(\Omega): w ~\mbox{is an eigenvector of (\ref{eigen}) corresponding to } \lambda\}.
\end{eqnarray*}
Without loss of generality, we assume the index of the eigenvalue $\lambda$ are $k_0 + 1, \cdots, k_0 + q$, that is, $\lambda_{k_0 } < \lambda = \lambda_{k_0 + 1} = \cdots = \lambda_{k_0 + q} < \lambda_{k_0 + q + 1}$. Let $\lambda_{h, l}$ be the $(k_0 + l)$-th eigenvalue of the corresponding discrete problem  (\ref{fe-eigen}), $u_{h, l}$ be the eigenfunction corresponding to $\lambda_{h, l}$,  for $(l=1, \cdots, q)$.
We see that $\lambda$ will be approximated from above by  the Galerkin approximate eigenvalues:
\begin{eqnarray*}
 \lambda \leq \lambda_{h,1} \leq \cdots \leq \lambda_{h, q}.
\end{eqnarray*}
Set
$$
\delta_h(\lambda)=\sup_{w\in M(\lambda), \|w\|_{b,\Omega}=1}\inf_{v\in V_h}\|w-v\|_{a,\Omega},
$$
and
 $M_h(\lambda) =  \mbox{span} \{u_{h, 1}, \cdots, u_{h, q}\}$.

From the definition of operators $K$ and $K_h$, we see that $K$ has eigenvalues
$$
 \mu_1 = \lambda_1^{-1} \geq \mu_2 = \lambda_2^{-1} \ge \mu_3 = \lambda_3^{-1}\geq\cdots \searrow 0
$$
associated with   eigenfunctions
$$
u_1, u_2, u_3,\cdots,
$$
and $K_h$ has eigenvalues
$$
 \mu_{1,h} = \lambda_{1, h}^{-1} \geq \mu_{2, h} = \lambda_{2, h}^{-1} \geq\cdots \ge \mu_{n_h, h} = \lambda_{n_h, h}^{-1}
$$
associated with   eigenfunctions
$$
u_{1, h}, u_{2,h}, \cdots, u_{n_h,h}.
$$

Let $\Gamma$ be a circle in the complex plane centered at $\mu = \lambda^{-1}$ and enclosing no any other eigenvalues of $K$. Then for $h$ sufficiently small, except $\mu_{1,h} = \lambda_{1, h}^{-1}, \mu_{2,h} = \lambda_{2, h}^{-1}, \cdots, \mu_{q,h} = \lambda_{q, h}^{-1}$, there is no any other eigenvalues of $K_h$ contained in $\Gamma$.  Define the spectral projection associated with $K$ and $\mu$ as follows:
\begin{eqnarray}
E = E(\lambda) &=& \frac{1}{2 \pi i} \int_{\Gamma}(z - K)^{-1} dz, \label{operator-E}\\
E_h = E_h(\lambda) &=& \frac{1}{2 \pi i} \int_{\Gamma}(z - K_h)^{-1} dz.\label{operator-Eh}
\end{eqnarray}
It has been proved that $E_h(\lambda): M(\lambda)  \rightarrow  M_h(\lambda)$ is one to one and onto if $h$ is sufficiently small~\cite{babuska-osborn89,babuska-osborn91}.


 The following results are  classical and can be found in literature
(see, e.g., \cite{babuska-osborn89,babuska-osborn91,chatelin83}).

\begin{proposition}\Label{prop3.1}
 Let $\lambda$ be any eigenvalue of (\ref{eigen})
with multiplicity $q$ and   $ u_{h, 1}, \cdots, u_{h, q}$ with $\|u_{h, l}\|_{b,\Omega} = 1$($l = 1, \cdots, q$) be the Galerkin eigenfunctions
corresponding to $\lambda_{h, 1}, \cdots, \lambda_{h, q}$, respectively. There hold
\begin{eqnarray}\label{eigen-f-err}
\|u-E_{h} u\|_{b,\Omega}&\lc& \rho_{_\Omega}(h)\|u- E_{h} u\|_{a,\Omega}\quad \forall u\in M(\lambda), \\\label{eigen-f-err-2}
 \|u_{h, l}-E u_{h, l}\|_{b,\Omega}&\lc&
\rho_{_\Omega}(h)\|u_{h, l}- E u_{h, l}\|_{a,\Omega},\\
\| u_{h, l} - E  u_{h, l}\|_{a,\Omega} &\lc&   \delta_h(\lambda), ~~~\lambda_{h, l} - \lambda \lc \delta_h(\lambda)^2, \label{eigen-f-err-3}
\end{eqnarray}
where $E$ and $E_h$ are orthogonal  projection defined in (\ref{operator-E}) and (\ref{operator-Eh}), respectively.
\end{proposition}

We can easily obtain the following two corollaries, which will be used in our following analysis.
\begin{corollary}\label{coro2.1}
 For any $u\in M(\lambda)$ with $\|u\|_b = 1$, we have
 \begin{eqnarray}\label{Ehu-bnorm}
1 - C \rho_{\Omega}(h) \delta_h(\lambda) \leq \|E_h u \|_b^2 \leq 1,
\end{eqnarray}
where $C$ is some constant not depending on $h$.
\end{corollary}
\begin{proof}
On the one hand, since $E_h$ is an orthogonal projection, we get
\begin{eqnarray}\label{coro-ineq1}
\|E_h u\|_b \leq \|E_h\| \|u \|_b = 1.
\end{eqnarray}
On the other hand, we have
 \begin{eqnarray}\label{coro-ineq2}
 b(E_h u, E_h u) &=& b(E_h u - u + u, E_h u - u + u) \nonumber\\ 
 &=& 1 + b(E_h u -u, E_h u) + b(E_h u - u, u)  = 1 + b(E_h u - u, u),
\end{eqnarray}
where the fact that $E_h $ is an orthogonal projection is used in the last equation.
We can easily obtain from the Proposition \ref{prop3.1} that
\begin{eqnarray}\label{coro-ineq3}
|b(E_h u - u, u)| \leq C \rho_{\Omega}(h) \delta_h(\lambda),
\end{eqnarray}
here, $C$ is some constant not depending on $h$.
Combining (\ref{coro-ineq1}), (\ref{coro-ineq2}), and (\ref{coro-ineq3}), we obtain the conclusion.
\end{proof}

\begin{corollary}\label{coro2.2}
 For any $u_i, u_j \in M(\lambda)$ with $b(u_i, u_j)  = \delta_{ij} ~(i,j=1,2\cdots,q)$, there holds
 \begin{eqnarray}\label{Ehu-bnorm}
 b(E_h u_i, E_h u_j) = \delta_{i j} + \mathcal{O}(\rho_{\Omega}(h) \delta_h(\lambda)).
\end{eqnarray}
\end{corollary}
\begin{proof}
Note that
 \begin{eqnarray}\label{coro-ineq2}
 b(E_h u_i, E_h u_j) &=& b(E_h u_i - u_i + u_i, E_h u_j - u_j + u_j)\nonumber\\
 &=& \delta_{ij} + b(E_h u_i -u_i, E_h u_j) + b(u_i, E_h u_j - u_j)  = \delta_{ij} + b(u_i, E_h u_j - u_j).
\end{eqnarray}
We obtain from  proposition \ref{prop3.1} that
\begin{eqnarray}\label{coro-ineq3}
|b(u_i, E_h u_j - u_j)| \lc \rho_{\Omega}(h) \delta_h(\lambda),
\end{eqnarray}
which finishes the proof.
\end{proof}

The following results will be used in our analysis~\cite{babuska-osborn89}.
\begin{lemma}\label{Eh-Rh}
There is a constant $C$ independent of $h$, such that for any $u\in M(\lambda)$
\begin{eqnarray}\label{lemma-Eh-Rh-conc}
1\leq \frac{\|u - E_h  u\|_{a, \Omega}}{\|u - R_h u\|_{a, \Omega}} \leq 1 + C \nu(h),
\end{eqnarray}
where $\nu(h)$ is defined as follows:
\begin{eqnarray}\label{def-nu}
\nu(h)=\sup_{ f\in H_0^1(\Omega),\|f\|_{a,\Omega}=1} \inf_{v\in V_h}\|Kf-v\|_{a,\Omega}.
\end{eqnarray}
\end{lemma}
%

For two spaces $X$ and $Y$ of $H_0^1(\Omega)$, we denote
\begin{eqnarray}\label{def-D}
  d_{H_0^1(\Omega)}(X, Y) = \sup_{u \in X, \|u\|_{b} = 1}\inf_{v\in Y}\|u - v\|_{a, \Omega},
  \end{eqnarray}
  and the gap between $X$ and
$Y$ as follows:
  \begin{eqnarray}\label{def-gap}
  \delta_{H_0^1(\Omega)}(X, Y) = \max\{d_{H_0^1(\Omega)}(X, Y), d_{H_0^1(\Omega)}(Y, X)\}.
  \end{eqnarray}

  For $d_{H_0^1(\Omega)}(X, Y)$ defined above, we have \cite{babuska-osborn91}
  \begin{lemma}\label{thm-dxy-dyx}
  If $\dim X = \dim Y <\infty$, then $d_{H_0^1(\Omega)}(Y, X) \leq d_{H_0^1(\Omega)}(X, Y)[1 - d_{H_0^1(\Omega)}(X, Y)]^{-1}$.
  \end{lemma}
%

\section{Adaptive finite element method}\label{adaptive-algorithm}
\setcounter{equation}{0}


Here and hereafter we consider the approximation for some eigenvalue $\lambda$ of (\ref{eigen}) with multiplicity $q$ and its corresponding eigenfunction space  $ M(\lambda)$. Let $(\lambda_{h, l}, u_{h, l}) (l=1, \cdots, q)$ be
 the $q$ eigenpair of (\ref{fe-eigen}) which satisfy (\ref{eigen-f-err}), (\ref{eigen-f-err-2}), and \eqref{eigen-f-err-3}.


%

Note that (\ref{eigen}) and (\ref{fe-eigen}) can be rewritten as
\begin{eqnarray*}\label{operator-K-1}
  u=\lambda K u,~~~ u_{h}=\lambda_{h} R_h K u_{h},
\end{eqnarray*}
where $K$ and $R_h$ are the operators  defined by (\ref{def-k}) and (\ref{Gprojection}), respectively.

For any $u \in M(\lambda)$ with $\|u\|_b = 1$, since $E_h u \in M_h(\lambda)$, we have that  there exist some constants $\{\alpha_{h, l}(u)\}_{l=1}^{q}$
 such that $E_h u = \sum_{l=1}^{q} \alpha_{h, l}(u) u_{h, l}$. In further, Corollary \ref{coro2.1} implies that  $\sum_{l=1}^{q}  \alpha^2_{h, l}(u) \leq 1$ . Define $\lambda^h = \frac{a(E_h u, E_h u)}{b(E_h u, E_h u)}$, we have
 \begin{eqnarray}\label{eigenval-approx}
  \lambda^h  = \frac{1}{\sum_{l=1}^{q} \alpha_{h, l}^2(u)} \sum_{l=1}^{q} \alpha_{h, l}^2(u) \lambda_{h, l},
 \end{eqnarray}
which together with
Proposition \ref{prop3.2}  leads to
\begin{eqnarray}
|\lambda-\lambda^{h}|
 &\le&  \frac{\|u-E_h u\|_{a,\Omega}^2}{\|E_h u\|_{b, \Omega}^2} =  \frac{\|u-E_h u\|_{a,\Omega}^2}{ \sum_{l=1}^{q} \alpha_{h, l}^2(u)}.
 \end{eqnarray}
%

 Define $w^h =  \sum_{l=1}^{q} \alpha_{h,l}(u) \lambda_{h, l} K u_{h, l}$, and  we see that
\begin{eqnarray}\label{u-w}
 E_h u = R_h w^h.
\end{eqnarray}

\begin{theorem}\label{thm-eigen-boundary} Given $u \in M(\lambda)$ with $\|u\|_b = 1$, ant let $r(h)= \rho_{_\Omega}(h)+ \delta_h(\lambda).$ Then
\begin{eqnarray}\label{eigen-boundary-neq}
 \|u-E_h u\|_{a,\Omega}= \|w^h - R_h w^h\|_{a,\Omega} +\mathcal
 {O}(r(h))\|u-E_h u\|_{a,\Omega}.
\end{eqnarray}
\end{theorem}
\begin{proof}
 We obtain from the definition of $w^h$ that
\begin{eqnarray*}
 u-w^h&=&\lambda K u - \sum_{l=1}^{q} \alpha_{h, l}(u) \lambda_{h, l}  K u_{h, l} \nonumber\\
 &=&\lambda K (u- E_h u) + \lambda K(\sum_{l=1}^{q} \alpha_{h, l}(u) u_{h, l})
 - \sum_{l=1}^{q} \alpha_{h, l}(u) \lambda_{h, l} K u_{h, l}\nonumber\\
 &=& \lambda K (u- E_h u) + \sum_{l=1}^{q} \alpha_{h, l}(u)(\lambda -  \lambda_{h, l}) K u_{h, l}. \nonumber\\
\end{eqnarray*}
Since
\begin{eqnarray*}
\lambda^h - \lambda  = \frac{1}{\sum_{l=1}^{q} \alpha_{h, l}^2(u)}\sum_{l=1}^{q} \alpha_{h, l}^2(u) (\lambda_{h, l} - \lambda),
\end{eqnarray*}
we have
\begin{eqnarray*}
\| \sum_{l=1}^{q} \alpha_{h, l}(u)(\lambda -  \lambda_{h, l}) K u_{h, l} \|_{a,\Omega} &\lc & \sum_{l=1}^{q} |\alpha_{h, l}(u)(\lambda -  \lambda_{h, l})| \nonumber\\
& \leq & \big(\sum_{l=1}^{q} \alpha^2_{h, l}(u)(\lambda_{h, l} -  \lambda)\big)^{1/2} \big(\sum_{l=1}^{q} (\lambda_{h, l} -  \lambda) \big)^{1/2}   \nonumber\\
& = &   \big(\sum_{l=1}^{q} \alpha^2_{h, l}(u)\big)^{1/2}   (\lambda^h - \lambda)^{1/2}\big(\sum_{l=1}^{q} (\lambda_{h, l} -  \lambda) \big)^{1/2}  ,  \nonumber\\
\end{eqnarray*}
where H\"{o}lder inequality and  (\ref{eigenval-approx}) are used in the second inequality and the last equation, respectively.
Then, from (\ref{eigen-f-err-3}), we get
\begin{eqnarray*}
\| \sum_{l=1}^{q} \alpha_{h, l}(u)(\lambda -  \lambda_{h, l}) K u_{h, l} \|_{a,\Omega}&\leq & \big(\sum_{l=1}^{q} (\lambda_{h, l} -  \lambda) \big)^{1/2}   \|u - E_h u\|_a \nonumber\\
&\lc &\delta_h(\lambda) \|u - E_h u\|_a,
\end{eqnarray*}
which together with  the fact $\|K(u-E_h u)\|_{a,\Omega}\lc\| u - E_h u\|_{b, \Omega}$ and (\ref{eigen-f-err})  leads to
\begin{eqnarray}\label{thm1-neq1}
 \| u - w^h\|_{a,\Omega}
 &\leq & \tilde{C} ( \rho_{\Omega}(h) + \delta_h(\lambda) ) \|u - E_h u\|_{a,\Omega},
 \end{eqnarray}
 with $\tilde{C}$ some constant not depending on $h$.
Note that (\ref{u-w}) implies
\begin{eqnarray*}
 u - E_h u = w^h-R_h w^h + u-w^h.
 \end{eqnarray*}
Hence we obtain (\ref{eigen-boundary-neq}) from (\ref{thm1-neq1}). This completes the proof.
 \end{proof}

%

\subsection{A posteriori error estimators}
Following the element residual $\tilde{\mathcal{R}}_T(u_{i,h})$ and the jump residual $\tilde{J}_E(u_{i,h})$ for  (\ref{dis-fem}), we now define an
element residual $\mathcal{R}_T(E_h u)$ and a jump residual $J_E(E_h u)$ for (\ref{fe-eigen})  as follows:
 \begin{eqnarray*}
  \mathcal{R}_T(E_h u)
  &=&\lambda^{h} E_h u  + \nabla\cdot(A\nabla E_h u)-c E_h u~~ \mbox{in}~ T\in
  \mathcal{T}_h,\label{residual-eigen-multi}\\
  J_E(E_h u) &=& -A \nabla (E_h u)^{+}\cdot \nu^{+} - A \nabla (E_h u)^{-}\cdot
  \nu^{-}\nonumber\\
   &=& [[A\nabla E_h u]]_E \cdot \nu_E  ~~~~ \mbox{on}~ E\in
  \mathcal{E}_h, \label{jump-eigen-multi}
\end{eqnarray*}
where $E$ , $\nu^+$ and $\nu^-$ are defined as those of section \ref{sc:linear boundary}.

 For $T\in \mathcal{T}_h$, we define the local error indicator
  $\eta_h(E_h u, T)$ by
  \begin{eqnarray*}\label{error-local-eigen-multi}
   \eta^2_h(E_h u,T) = h_T^2\|\mathcal{R}_T(E_h u)\|_{0,T}^2 + \sum_{E\in \mathcal{E}_h,E\subset\partial T
   } h_E \|J_E(E_h u)\|_{0,E}^2
  \end{eqnarray*}
and the oscillation $osc_h(E_h u,T)$ by
\begin{eqnarray*}\label{osc-local-eigen-multi}
osc^2_h(E_h u,T) = h_T^2\|\mathcal{R}_{T}(E_h u)-\overline{\mathcal{R}_{T}(E_h u)}\|_{0,T}^2
   + \sum_{E\in \mathcal{E}_h,E\subset\partial T
   } h_E \|J_E(E_h u)-\overline{J_E(E_h u)}\|_{0,E}^2.
\end{eqnarray*}

We define the error
 estimator
$\eta_h(E_h u, \Omega)$ and the oscillation $osc_h(E_h u,\Omega)$ by
 \begin{eqnarray*}\label{error-estimator-eigen-multi}
  \eta^2_h(E_h u, \Omega) = \sum_{T\in \mathcal{T}_h, T \subset \Omega}
  \eta^2_h(E_h u, T)\quad \textnormal{and} \quad
  osc^2_h(E_h u, \Omega) =  \sum_{T\in \mathcal{T}_h, T \subset \Omega}
  osc^2_h(E_h u, T).
  \end{eqnarray*}

For $U_h=(u_{h,1},\cdots,u_{h,q})\in (V_h)^q$,  we let
\begin{eqnarray*}
\eta^2_h(U_h, T)=\sum_{l=1}^q\eta^2_h(u_{h, l}, T) \quad \textnormal{and} \quad osc^2_h(U_h, T)=\sum_{l=1}^q osc^2_h(u_{h, l}, T), ~\forall T\in\mathcal{T}_h,
\end{eqnarray*}
and
\begin{eqnarray*}
\eta^2_h(U_h, \Omega)=\sum_{T\in\mathcal{T}_h} \eta^2_h(U_h, T) \quad \textnormal{and} \quad osc^2_h(U_h, \Omega)=\sum_{T\in\mathcal{T}_h} osc^2_h(U_h, T).
\end{eqnarray*}
For any $E_h U=(E_h u_1,\cdots,E_h u_q)\in (V_h)^q$, we set
\begin{eqnarray*}
\eta^2_h(E_h U, T)=\sum_{l=1}^q\eta^2_h(E_h u_{l}, T) \quad \textnormal{and} \quad osc^2_h(E_h U, T)=\sum_{l=1}^q osc^2_h(E_h u_l, T),
\end{eqnarray*}
and
\begin{eqnarray*}
\eta^2_h(E_h U, \Omega)=\sum_{T\in\mathcal{T}_h}\eta^2_h(E_h U, T) \quad \textnormal{and} \quad osc^2_h(E_h U, \Omega)=\sum_{T\in\mathcal{T}_h} osc^2_h(E_h U, T).
\end{eqnarray*}

We shall now present the following property of eigenspace approximation that will  play a crucial role in our analysis.
\begin{lemma}\label{eta-etah} Let $h \in (0, h_0)$ and $h_0 \ll 1$.
 Then, for any  orthonormal basis $\{u_l\}_{l=1}^{q}$ of $M(\lambda)$, there hold
\begin{eqnarray}\label{lemma-eta-etah-conc1}
  \eta_h^2(U_h, T) \lc \eta_h^2(E_h U, T) \lc \eta_h^2(U_h, T), ~\forall T\in \mathcal{T}_h,
\end{eqnarray}
and
\begin{eqnarray}\label{lemma-eta-etah-conc2}
  osc_h^2(U_h, T) \lc osc_h^2(E_h U, T) \lc osc_h^2(U_h, T), ~\forall T\in \mathcal{T}_h,
\end{eqnarray}
where $E_h U = (E_h u_1, \cdots, E_h u_q)$, $U_h = (u_{h,1}, \cdots, u_{h, q})$.
\end{lemma}
\begin{proof}
First, we prove (\ref{lemma-eta-etah-conc1}).
We denote $E_h u_l$ as $v_{h, l}$. On  one hand, Since $M_{h}(\lambda) = \mbox{span} \{u_{h,1}, \cdots, u_{h, q}\}$,  we have that
 there exists $q$ constants $\beta_{h,j}^{l}(j = 1, \cdots, q)$ such that
\begin{eqnarray*}
v_{h,l} = \sum_{j=1}^q \beta_{h,j}^{l} u_{h,j}, ~~l = 1, \cdots, q.
\end{eqnarray*}
We see from Corollary \ref{coro2.1}  that $\|v_{h, l}\|_b \leq 1$,  namely,
 \begin{eqnarray*}
 \sum_{j=1}^{q} \big(\beta_{h, j}^{l}\big)^2  \leq 1.
 \end{eqnarray*}
We may apply the above fact and  analyze as follows
\begin{eqnarray*}
\eta_h^2(E_h U, T)&=&  \sum_{l=1}^{q} \eta_h^2(\sum_{j=1}^q \beta_{h,j}^{l} u_{h,j}, T) \nonumber\\
&\leq&  \sum_{l=1}^{q} \sum_{j=1}^q \big(\beta_{h,j}^{l}\big)^2 \sum_{j=1}^{q} \eta_h^2(u_{h,j}, T) \nonumber\\
&\leq&  \sum_{l=1}^{q} \sum_{j=1}^q   \eta_h^2(u_{h,j}, T)  = q \sum_{j=1}^q   \eta_h^2(u_{h,j}, T).
\end{eqnarray*}
Consequently,
\begin{eqnarray}\label{lemma-eta-etah-ineq1}
\eta_h^2(E_h U, T)\leq  q  \eta_h^2(U_h, T).
\end{eqnarray}
Similarly, we can get
\begin{eqnarray}\label{lemma-eta-etah-ineq1-2}
osc_h^2(E_h U, T)\leq  q  osc_h^2(U_h, T).
\end{eqnarray}

On the other hand, since the operator $E_h: M(\lambda) \rightarrow M_h(\lambda)$ is one-to-one and onto~(\cite{babuska-osborn89, babuska-osborn91}), we have that  $\{v_{h,l}\}_{l=1}^{q}$ is  basis of $M_h(\lambda)$, namely,
$M_{h}(\lambda) = \mbox{span} \{u_{h,1}, \cdots, u_{h, q}\} =  \mbox{span} \{v_{h,1}, \cdots, v_{h, q}\}$.  So  there exist $q$ constants $\hat{\beta}_{H,j}^{l}(j = 1, \cdots,
q)$ 
such that
\begin{eqnarray*}
u_{h,l} = \sum_{j=1}^q \hat{\beta}_{h,j}^{l} v_{h,j}, ~~l=1, \cdots, q.
\end{eqnarray*}
Note that
\begin{eqnarray}\label{lemma-eta-etah-eq1}
1 = b(u_{h,l}, u_{h, l}) &=& b(\sum_{j=1}^q \hat{\beta}_{h,j}^{l} v_{h,j}, \sum_{j=1}^q \hat{\beta}_{h,j}^{l} v_{h,j}) \nonumber\\
& = & \sum_{j=1}^q (\hat{\beta}_{h,j}^{l})^2b(v_{h, j}, v_{h, j}) + \sum_{i \neq j, i, j=1}^q \hat{\beta}_{h,i}^{l} \hat{\beta}_{h,j}^{l} b(v_{h, i}, v_{h, j}).
\end{eqnarray}
 We obtain from Corollary \ref{coro2.2}  that
\begin{eqnarray*}
b(v_{h, i}, v_{h, j}) = \delta_{ij} + \mathcal{O}(\rho_{\Omega}(h)  \delta_h(\lambda)),
\end{eqnarray*}
which implies that there exists some constant C not
depending on $h$ such that
\begin{eqnarray}\label{lemma-eta-etah-ineq2}
 (1 - C \rho_{\Omega}(h)  \delta_h(\lambda)) \sum_{j=1}^q (\hat{\beta}_{h,j}^{l})^2 \leq \sum_{j=1}^q (\hat{\beta}_{h,j}^{l})^2b(v_{h, j}, v_{h, j})  \leq (1 + C \rho_{\Omega}(h)  \delta_h(\lambda)) \sum_{j=1}^q (\hat{\beta}_{h,j}^{l})^2,
\end{eqnarray}
and
\begin{eqnarray}\label{lemma-eta-etah-ineq3}
 |\sum_{i \neq j, i, j=1}^q \hat{\beta}_{h,i}^{l} \hat{\beta}_{h,j}^{l} b(v_{h, i}, v_{h, j})| &\leq & \sum_{i \neq j, i, j=1}^q |\hat{\beta}_{h,i}^{l} \hat{\beta}_{h,j}^{l} |\rho_{\Omega}(h)  \delta_h(\lambda) \nonumber\\
 &\leq & C \sum_{i=1}^q (\hat{\beta}_{h,i}^{l})^2 \rho_{\Omega}(h)  \delta_h(\lambda).
 \end{eqnarray}
 Combining (\ref{lemma-eta-etah-eq1}), (\ref{lemma-eta-etah-ineq2}), and (\ref{lemma-eta-etah-ineq3}), we get
 \begin{eqnarray*}
 \frac{1}{1 + C \rho_{\Omega}(h) \delta_h(\lambda) } \leq \sum_{i=1}^q (\hat{\beta}_{h,i}^{l})^2 \leq \frac{1}{1 - C \rho_{\Omega}(h)  \delta_h(\lambda)}.
 \end{eqnarray*}
Therefore,
\begin{eqnarray*}
  \eta_h^2(U_h, T) &=&  \sum_{l=1}^{q}   \eta_h^2(\sum_{j=1}^q \hat{\beta}_{h,j}^{l} v_{h,j}, T) \nonumber\\
&\leq&   \sum_{l=1}^{q} \Big( \sum_{j=1}^q \big(\hat{\beta}_{h,j}^{l}\big)^2\Big) \Big(  \sum_{j=1}^q   \eta_h^2(v_{h,j}, T)\Big) \nonumber\\
&\leq& \frac{1}{1 - C \rho_{\Omega}(h)  \delta_h(\lambda)} \sum_{l=1}^{q}  \sum_{j=1}^q     \eta_h^2(v_{h,j}, T)\nonumber\\
 &\leq& \frac{q}{1 - C \rho_{\Omega}(h)  \delta_h(\lambda)}  \sum_{j=1}^q    \eta_h^2(v_{h,j}, T),
\end{eqnarray*}
that is,
\begin{eqnarray}\label{lemma-eta-etah-ineq4}
 \eta_h^2(U_h, T)\leq  \frac{q}{1 - C \rho_{\Omega}(h) \delta_h(\lambda)}     \eta_h^2(E_h U, T).
\end{eqnarray}
By using the same arguments, we have
\begin{eqnarray}\label{lemma-eta-etah-ineq4-2}
 osc_h^2(U_h, T)\leq  \frac{q}{1 - C \rho_{\Omega}(h) \delta_h(\lambda)}     osc_h^2(E_h U, T).
\end{eqnarray}

Since $h \in (0, h_0)$ and $h_0 \ll 1$, combining (\ref{lemma-eta-etah-ineq1}) and (\ref{lemma-eta-etah-ineq4}),  we arrive at (\ref{lemma-eta-etah-conc1}), and (\ref{lemma-eta-etah-conc2}) can be obtained  from (\ref{lemma-eta-etah-ineq1-2}) and (\ref{lemma-eta-etah-ineq4-2}). This completes
 the proof.
\end{proof}

Given $h_0\in (0,1)$, define $$\tilde{r}(h_0)=\sup_{h\in (0,h_0)}r(h).$$
\begin{theorem}\label{thm-error-estimator}
There exist constants $C_1, C_2$ and $C_3$, which only depend on the shape regularity constant $\gamma^{\ast}$, coercivity constant $c_a$ and continuity
constant $C_a$ of the bilinear form, such that
 \begin{eqnarray}\label{upper-bound}
  \|u-E_h u\|_{a,\Omega} \leq C_1 \eta_h(E_h u, \Omega)
 \end{eqnarray}
 and
 \begin{eqnarray}\label{lower-bound}
~~~~~~C^2_2 \eta^2_h(E_h u, \Omega)-
  C^2_3 osc^2_h (E_h u, \Omega) \le
 \|u-E_h u\|_{a,\Omega}^2
 \end{eqnarray}
 provided $h_0 \ll 1$.
Consequently,
 \begin{eqnarray*}\label{upper-bound-eigenvalue}
  |\lambda - \lambda^{h}| \lc  \eta^2_h(E_h u, \Omega)
 \end{eqnarray*}
 and
 \begin{eqnarray*}\label{lower-bound-eigenvalue}
~~~~~~ \eta^2_h(E_h u, \Omega)-
   osc^2_h (E_h u, \Omega) \lc
  |\lambda - \lambda^{h}|.
 \end{eqnarray*}
\end{theorem}
\begin{proof} Recall that $L w^h = \sum_{l=1}^{q} \alpha_{h,l}(u) \lambda_{h, l} u_{h, l}$.
We obtain from (\ref{boundary-upper}) and
 (\ref{boundary-lower}) that
\begin{eqnarray}\label{auxiliary-boundary-problem-upper}
  \|w^h - R_h w^h \|_{a,\Omega} \leq \tilde{C}_1 \eta_h (E_h u, \Omega)
\end{eqnarray}
and
\begin{eqnarray}\label{auxiliary-boundary-problem-lower-1}
&&\tilde{C}^2_2 \eta^2_h (E_h u, \Omega)- \tilde{C}^2_3 osc^2_h (E_h u, \Omega)
 \le \|w^h-R_h
 w^h\|_{a,\Omega}^2.
\end{eqnarray}
Combining (\ref{u-w}), (\ref{eigen-boundary-neq}), (\ref{auxiliary-boundary-problem-upper}) with (\ref{auxiliary-boundary-problem-lower-1}), we complete the proof. In
particular, we may choose the constants $C_1$, $C_2$ and $C_3$ satisfying
\begin{eqnarray}\label{coef-eigen-bound}
  C_1 = \tilde{C}_1 (1+ \tilde{C} \tilde{r}(h_0)), ~~~ C_2 = \tilde{C}_2 (1 - \tilde{C}
  \tilde{r}(h_0)), ~~~ C_3 = \tilde{C}_3(1 - \tilde{C} \tilde{r}(h_0)).
\end{eqnarray}
\end{proof}

%

From Theorem \ref{thm-error-estimator}, we can get the following a posteriori estimates for the gap $\delta_{H_0^1(\Omega)}(M(\lambda), M_h(\lambda))$ as follows.
\begin{theorem}\label{thm-error-estimator-space}
Let  $\lambda \in \mathbb{R}$  be some eigenvalue of (\ref{eigen}) with multiplicity $q$ and the corresponding eigenspace being $M(\lambda) = span\{u_1, \cdots, u_q\}$, $M_h(\lambda) = span\{ u_{h, 1}, \cdots, u_{h, q}\}$ be its   finite element approximation. Suppose $h \in (0, h_0)$ and $h_0 \ll 1$, then there
exist constants $C_1, C_2$ and $C_3$, which only depend on the shape regularity constant $\gamma^{\ast}$, coercivity constant $c_a$ and continuity
constant $C_a$ of the bilinear form, such that
 \begin{eqnarray}\label{upper-bound-space}
   \delta_{H_0^1(\Omega)}(M(\lambda), M_h(\lambda))   \lc \eta_h(U_h, \Omega)
 \end{eqnarray}
 and
 \begin{eqnarray}\label{lower-bound-space}
 \eta^2_h(U_h, \Omega)-
      osc^2_h (U_h, \Omega) \lc
  \delta_{H_0^1(\Omega)}^2(M(\lambda), M_h(\lambda)),
 \end{eqnarray}
 where $U_h = (u_{h, 1}, \cdots, u_{h, q})$.
\end{theorem}
\begin{proof}
Assume $\{u_1, \cdots, u_q\}$ be any orthonormal basis of $M(\lambda)$. Set  $U = (u_1, \cdots, u_q)$.
From Lemma \ref{Eh-Rh} and Theorem \ref{thm-error-estimator}, we have that for any $u \in M(\lambda)$ with $\|u\|_b = 1$, there hold
\begin{eqnarray}\label{upper-bound-space}
  \|u-R_h u\|_{a,\Omega} \lc  \eta_h(E_h u, \Omega)
 \end{eqnarray}
 and
 \begin{eqnarray}\label{lower-bound-space}
  \eta^2_h(E_h u, \Omega)-
    osc^2_h (E_h u, \Omega) \lc
 \|u-R_h u\|_{a,\Omega}^2.
 \end{eqnarray}
  Therefore
 \begin{eqnarray*}
   \|U-R_h U\|_{a,\Omega} \lc \eta_h(E_h U, \Omega)
 \end{eqnarray*}
 and
 \begin{eqnarray*}
 \eta^2_h(E_h U, \Omega)-
    osc^2_h (E_h U, \Omega) \lc
 \|U-R_h U\|_{a,\Omega}^2.
 \end{eqnarray*}
 From Lemma \ref{eta-etah}, we have
 \begin{eqnarray}\label{upper-bound-Rh}
   \|U-R_h U\|_{a,\Omega} \lc \eta_h(U_h, \Omega)
 \end{eqnarray}
 and
 \begin{eqnarray}\label{lower-bound-Rh}
  \eta^2_h(U_h, \Omega)-
  C_3 osc^2_h (U_h, \Omega) \lc
 \|U-R_h U\|_{a,\Omega}^2.
 \end{eqnarray}

 Since $\{u_1, \cdots, u_q\}$ is any orthonormal basis of $M(\lambda)$, therefore,
   \begin{eqnarray}\label{dist-eq}
 \sup_{u\in M(\lambda), \|u\|_b =1}\inf_{v \in M_h(\lambda)} \| u - v\|_{a, \Omega} &=& \sup_{u\in M(\lambda), \|u\|_b =1} \|u-R_h u\|_{a,\Omega} \nonumber\\
 &\cong& \max_{l=1, \cdots, q} \|u_l - R_h u_l\|_{a, \Omega} \cong \|U - R_h U\|_{a, \Omega}.
 \end{eqnarray}
 Combining  (\ref{upper-bound-Rh}),  (\ref{lower-bound-Rh}), and (\ref{dist-eq}), we get
 \begin{eqnarray*}
   \sup_{u\in M(\lambda), \|u\|_b =1}\inf_{v \in M_h(\lambda)} \| u - v\|_{a, \Omega} \lc \eta_h(U_h, \Omega)
 \end{eqnarray*}
 and
 \begin{eqnarray*}
  \eta^2_h(U_h, \Omega)-
    osc^2_h (U_h, \Omega) \lc \sup_{u\in M(\lambda), \|u\|_b =1}\inf_{v \in M_h(\lambda)} \| u - v\|_{a, \Omega}.
 \end{eqnarray*}
 namely,
 \begin{eqnarray}\label{upper-bound-dxy}
   d_{H_0^1(\Omega)}(M(\lambda), M_h(\lambda)   \lc \eta_h(U_h, \Omega)
 \end{eqnarray}
 and
 \begin{eqnarray}\label{lower-bound-dxy}
     \eta^2_h(U_h, \Omega)-
   osc^2_h (U_h, \Omega) \lc  d_{H_0^1(\Omega)}(M(\lambda), M_h(\lambda).
 \end{eqnarray}

 Since $\dim M(\lambda) = \dim M_h(\lambda) = q$,  we obtain from Lemma \ref{thm-dxy-dyx} that
  \begin{eqnarray}\label{dxy-dyx-neq}
 d_{H_0^1(\Omega)}(M(\lambda), M_h(\lambda)  \lc d_{H_0^1(\Omega)}(M_h(\lambda), M(\lambda)  \lc d_{H_0^1(\Omega)}(M(\lambda), M_h(\lambda).
 \end{eqnarray}
 Therefore, combining (\ref{upper-bound-dxy}), (\ref{lower-bound-dxy}), and (\ref{dxy-dyx-neq}), we have
 \begin{eqnarray}\label{upper-bound-dyx}
   d_{H_0^1(\Omega)}(M_h(\lambda), M(\lambda) \lc \eta_h(U_h, \Omega)
 \end{eqnarray}
 and
 \begin{eqnarray}\label{lower-bound-dyx}
   \eta^2_h(U_h, \Omega)-
    osc^2_h (U_h, \Omega) \lc   d_{H_0^1(\Omega)}(M_h(\lambda), M(\lambda).
 \end{eqnarray}

 The definition of $\delta_{H_0^1(\Omega)}(M(\lambda), M_h(\lambda)$, together with (\ref{upper-bound-dxy}), (\ref{lower-bound-dxy}), (\ref{upper-bound-dyx}), and (\ref{lower-bound-dyx}) means
\begin{eqnarray*}
   \delta_{H_0^1(\Omega)}(M(\lambda), M_h(\lambda) \lc \eta_h(U_h, \Omega)
 \end{eqnarray*}
 and
 \begin{eqnarray*}
    \eta^2_h(U_h, \Omega)-
    osc^2_h (U_h, \Omega) \lc \delta_{H_0^1(\Omega)}(M(\lambda), M_h(\lambda).
 \end{eqnarray*}
 This completes the proof.
\end{proof}
\subsection{Adaptive algorithm}
Recall that the adaptive procedure consists of loops of the form
\begin{center}
\begin{tabular}{|p{80mm}|}\hline
\vspace{2mm}
 $
~~~~~ \mbox{\bf Solve}\rightarrow
 \mbox{\bf Estimate}\rightarrow\mbox{\bf Mark}\rightarrow\mbox{\bf Refine}
$ \vspace{2mm}\\
 \hline
\end{tabular}
\end{center}
 We assume that the solutions of the finite dimensional problems
can be solved to any accuracy efficiently.\footnote{ In fact, we have ignored two important practical issues: the inexact solution of the resulting algebraic
system and the numerical integration.
We remark the discussion about  the inexact solution in Section \ref{concluding-remark}.} The a posteriori error estimators are an essential part of the {\bf
 Estimate} step. In the following discussion, we use  $\eta_h(U_{h}, \Omega)$
defined above as the a posteriori error estimator.

\vskip 0.1cm
\begin{algorithm}\label{algorithm-AFEM-eigen}~
 Choose a parameter $0 < \theta <1.$
\begin{enumerate}
\item Pick a given mesh $\mathcal{T}_{h_0}$, and let $k=0$.
\item Solve the system (\ref{fe-eigen}) on $\mathcal{T}_{h_k}$ to get the discrete solution  $(\lambda_{h_k, l}, u_{h_k, l})(l=1, \cdots, q)$.
\item Compute local error indictors $\eta_{h_k}(u_{h_{k,l}}, T)(l=1, \cdots, q)$. for all $T \in \mathcal{T}_{h_k}$.
\item Construct $\mathcal{M}_{h_k} \subset \mathcal{T}_{h_k}$ by {\bf D\"{o}rfler marking strategy} and parameter
 $\theta$.
\item Refine $\mathcal{T}_{h_k}$ to get a new conforming mesh $\mathcal{T}_{h_{k+1}}$  by Procedure {\bf REFINE}.
\item Let $k=k+1$ and go to 2.
\end{enumerate}
\end{algorithm}
\vskip 0.1cm

The {\bf D\"{o}rfler marking strategy} in the algorithm above is similar to those for the boundary value problems, only with $\tilde{\eta}$ being replaced by $\eta$, which are  stated as follows.

\begin{center}
\begin{tabular}{|p{120mm}|}\hline
\begin{center}
{\bf D\"{o}rfler marking strategy}
 \end{center}\\
 Given a parameter $0<\theta < 1$.
\begin{enumerate}
 \item Construct a   subset $\mathcal{M}_{h_k}$ of
 $\mathcal{T}_{h_k}$ by selecting some elements
 in $\mathcal{T}_{h_k}$ such that
\begin{eqnarray}\label{dorfler-prop}
 \sum_{T\in \mathcal{M}_{h_k}} \eta_{h_k}^2(U_{h_k}, T)  \geq \theta
  \eta_{h_k}^2(U_{h_k},\Omega).
 \end{eqnarray}
 \item Mark all the elements in $\mathcal{M}_{h_k}$.
 \end{enumerate}
\\
 \hline
\end{tabular}
\end{center}

%
%

We shall now present the following property of eigenspace approximation that will  play a crucial role in our analysis.
\begin{lemma}\label{mark-eta-etah} Let $H \in (0, h_0)$ and $h_0 \ll 1$.
Given constant $\theta\in(0,1)$. If
\begin{eqnarray}\label{mark-eta-etah-cond}
\sum_{T\in \mathcal{M}_H} \eta_H^2(U_H, T)  \geq \theta \eta_H^2(U_H, \Omega),
\end{eqnarray}
then there exists some constant $\theta' \in (0, 1)$, such that for any orthonormal basis $\{u_l\}_{l=1}^{q}$ of $M(\lambda)$, there holds
\begin{eqnarray}\label{mark-eta-etah-conc}
\sum_{T\in \mathcal{M}_H} \eta_H^2(E_H U, T) \geq \theta' \eta_H^2(E_H U, \Omega),
\end{eqnarray}
where $E_H U=(E_H u_1,\cdots,E_H u_q)$.
\end{lemma}
\begin{proof}
From the proof of Lemma \ref{eta-etah}, for all $T \in \mathcal{T}_H$ we have (see also (\ref{lemma-eta-etah-ineq1}) and (\ref{lemma-eta-etah-ineq4})).
\begin{eqnarray}
\eta_H^2(E_H U, T)\leq  q  \eta_H^2(U_H, T),
\end{eqnarray}
and
\begin{eqnarray}
 \eta_H^2(U_H, T)\leq  \frac{q}{1 - C \rho_{\Omega}(H) \delta_H(\lambda)}     \eta_H^2(E_H U, T),
\end{eqnarray}
which are nothing but (\ref{lemma-eta-etah-ineq1}) and (\ref{lemma-eta-etah-ineq4}) with $h$ being replaced by $H$,
$C$ is some constant independent of $H$.
Therefore
\begin{eqnarray}\label{mark-eta-etah-ineq}
 \frac{q}{1 - C \rho_{\Omega}(H) \delta_H(\lambda)}    \sum_{T\in \mathcal{M}_H} \eta_H^2(E_H U, T)
  \geq \frac{\theta}{q}\eta_H^2(E_H U, \Omega).
\end{eqnarray}
Since $H \in (0, h_0)$ and $h_0 \ll 1$, we have that there exists some $0<\sigma \ll 1$, such that $ C \rho_{\Omega}(H) \delta_H(\lambda) \leq \sigma$. Hence, (\ref{mark-eta-etah-ineq}) is nothing but  (\ref{mark-eta-etah-conc}) with $\theta' = (1 - \sigma) \frac{\theta}{ q^2 } $.  Here and hereafter, we choose $\sigma = 0.01$, for instance, then $\theta' = 0.99  \frac{\theta }{q^2 }$.
\end{proof}

Similarly, we  have
\begin{lemma}\label{mark-etah-eta} Let $H \in (0, h_0)$ and $h_0 \ll 1$.
Let  $\{u_l\}_{l=1}^{q}$ be an orthonormal basis of $M(\lambda)$.  Given constant $\theta \in(0,1)$. If
\begin{eqnarray}\label{mark-etah-eta-cond}
\sum_{T\in \mathcal{M}_H} \eta_H^2(E_H U, T) \geq \theta \eta_H^2(E_H U, \Omega),
\end{eqnarray}
where $E_H U=(E_H u_1,\cdots,E_H u_q)$,
then there exists some constant $\theta' \in (0, 1)$, such that
\begin{eqnarray}\label{mark-etah-eta-conc}
\sum_{T\in \mathcal{M}_H} \eta_H^2(U_H, T)  \geq \theta' \eta_H^2(U_H, \Omega).
\end{eqnarray}
\end{lemma}

\section{Convergence rate}\Label{convergence-sec}\setcounter{equation}{0}
%

Following  Theorem \ref{thm-eigen-boundary}, by using the similar arguments in \cite{dai-xu-zhou08}, we  can establish some relationship between the two level approximations, which will be used in our analysis
for  convergence rate.
\begin{lemma}\label{lemma-bound-eigen}
Let $h, H \in (0, h_0)$,  $\{u_l\}_{l=1}^{q}$ be any orthonormal basis of $M(\lambda)$, $U = (u_1, u_2, \cdots, u_q)$,   $\lambda^{H, l} = a(E_H u_l, E_H u_l)$, $w^{H, l}=\sum_{i=1}^{q} \alpha_{H,i}(u_l) \lambda_{H,
i} K u_{H,i}$ and  $W^H\equiv (w^{H,1}\cdots, w^{H,q})$. Then
 \begin{eqnarray}\label{lemma-bound-eigen-conc-1}
\|U - E_h U\|_{a, \Omega}= \|W^{H} - R_h W^{H} \|_{a, \Omega}
  + \mathcal{O} (\tilde{r}(h_0)) \left( \|U - E_h U
 \|_{a, \Omega} + \|U - E_H U\|_{a, \Omega}\right),
\end{eqnarray}
\begin{eqnarray}\label{lemma-bound-eigen-conc-2}
osc_h(E_h U, \Omega)= \widetilde{osc}_h(R_h W^{H}, \Omega)  + \mathcal{O} (\tilde{r}(h_0)) \left( \|U - E_h U\|_{a, \Omega} + \|U - E_H U\|_{a,
\Omega}\right),
\end{eqnarray}
and
\begin{eqnarray}\label{lemma-bound-eigen-conc-3}
\eta_{h}(E_h U, \Omega) =  \tilde{\eta}_{h}(R_{h} W^{H}, \Omega)
 + \mathcal{O}
 (\tilde{r}(h_0))\left( \|U - E_h U\|_{a, \Omega} + \|U -
E_H U\|_{a, \Omega}\right).
\end{eqnarray}
\end{lemma}

\begin{proof}
It is sufficient to prove that for any $u\in M(\lambda)$ with $\|u\|_{b} = 1$, $w^{h}=\sum_{l=1}^{q} \alpha_{h, l}(u) \lambda_{h, l} K u_{h, l}$, and $w^{H}=\sum_{l=1}^{q} \alpha_{H,l}(u) \lambda_{H, l} K u_{H, l}$, the following equalities hold,
 \begin{eqnarray}\label{bound-eigen-conc-1}
~~ \|u - E_h u\|_{a, \Omega}&=& \|w^H - R_h w^H \|_{a, \Omega}
 + \mathcal{O} (\tilde{r}(h_0))\left( \|u - E_h u
 \|_{a, \Omega}+  \|u - E_H u\|_{a, \Omega}\right),
\end{eqnarray}
\begin{eqnarray}\label{bound-eigen-conc-2}
~~~~osc_h(E_h u, \Omega)&=&\widetilde{osc}_h(R_h w^H, \Omega)
 + \mathcal{O} (\tilde{r}(h_0))\left( \|u - E_h u\|_{a, \Omega} +
\|u - E_H u\|_{a, \Omega}\right),
\end{eqnarray}
and
\begin{eqnarray}\label{bound-eigen-conc-3}
 ~~~\eta_{h}(E_h u, \Omega) = \tilde{\eta}_{h}(R_h w^H, \Omega) + \mathcal{O}
 (\tilde{r}(h_0))\left( \|u - E_h u\|_{a, \Omega} + \|u -
E_H u\|_{a, \Omega}\right).
\end{eqnarray}

First, we prove (\ref{bound-eigen-conc-1}).  We see that
 \begin{eqnarray*} \|R_h (w^h - w^H)+ w^H -
u\|_{a,\Omega} &\lc& \| w^h - w^H\|_{a,\Omega}
+ \|u - w^H\|_{a,\Omega}\\
&\lc& \|u - w^H\|_{a,\Omega} + \|u - w^h\|_{a,\Omega},
\end{eqnarray*}
which together with (\ref{thm1-neq1}) leads to
\begin{eqnarray}\label{lem-bound-eigen1}
\|R_h (w^h - w^H) + w^H - u\|_{a,\Omega} \lc \tilde{r}(h_0)(\|u - E_H u\|_{a,\Omega} + \|u - E_h u\|_{a,\Omega}).
\end{eqnarray}
Note  that (\ref{u-w}) implies
 \begin{eqnarray}\label{lem-bound-eigen2}
u-E_h u=w^H - R_h w^H+R_h(w^H -w^h)+u-w^H,
 \end{eqnarray}
 we  get (\ref{bound-eigen-conc-1}) from
 (\ref{lem-bound-eigen1}).

Next, we prove (\ref{bound-eigen-conc-2}). We obtain from  Lemma
\ref{lemma-osc-L} that
 \begin{eqnarray}\label{lem-bound-eigen3}
\widetilde{osc}_h(R_h(w^h-w^H),\Omega)\lc  \|R_h (w^H - w^h)\|_{a, \Omega},
\end{eqnarray}
which together with (\ref{thm1-neq1}) and (\ref{lem-bound-eigen1}) yields
\begin{eqnarray}\label{lem-bound-eigen4}
 \widetilde{osc}_h(R_h (w^H - w^h), \Omega) \lc
\tilde{r}(h_0)(\|u - E_H u\|_{a,\Omega} + \|u - E_h u\|_{a,\Omega}).
  \end{eqnarray}
Due to $E_h u=R_h w^H+R_h(w^h-w^H)$, from
\begin{eqnarray*}\label{temp6}
\widetilde{osc}_h(R_h w^h,\Omega)=\widetilde{osc}_h(R_h w^H+R_h (w^h-w^H),\Omega),
\end{eqnarray*}
(\ref{lem-bound-eigen4}),
and the definition of oscillation,
we then arrive at (\ref{bound-eigen-conc-2}).

We finally  prove (\ref{bound-eigen-conc-3}).  By (\ref{boundary-lower}), we have
\begin{eqnarray*}\label{lem-bound-eigen5}
\tilde{\eta}_h(R_h( w^h - w^H), \Omega) &\lc& \| (w^h - w^H) -R_h(w^h - w^H)\|_{a, \Omega}+ \widetilde{osc}_h(R_h(w^h - w^H), \Omega)
\nonumber\\
&\lc & \|u - w^h\|_{a, \Omega} +  \|u - w^H\|_{a, \Omega} + \|R_h(w^h - w^H)\|_{a, \Omega},
\end{eqnarray*}
where  (\ref{lem-bound-eigen3}) is used in the last inequality. Using (\ref{thm1-neq1}) and (\ref{lem-bound-eigen1}), we obtain
\begin{eqnarray}\label{lem-bound-eigen6}
\tilde{\eta}_h(R_h( w^h - w^H), \Omega) &\lc& \tilde{r}(h_0)(\|u - E_H u\|_{a,\Omega} + \|u - E_h u\|_{a,\Omega}).
\end{eqnarray}
From (\ref{lem-bound-eigen6}) and the fact that
\begin{eqnarray*}
\tilde{\eta}_h(R_h w^h, \Omega) = \tilde{\eta}_h(R_h w^H + R_h( w^h - w^H), \Omega),
\end{eqnarray*}
we get
\begin{eqnarray*}
\tilde{\eta}_h(R_h w^h, \Omega) = \tilde{\eta}_h(R_h w^H, \Omega)
 + \mathcal{O} (\tilde{r}(h_0))\left( \|u - E_h u\|_{a, \Omega} + \|u -
E_H u\|_{a, \Omega}\right),
\end{eqnarray*}
which is nothing but (\ref{bound-eigen-conc-3}) since $\tilde{\eta}_h(R_h w^h, \Omega) = \eta_h(E_h u,\Omega)$.
\end{proof}

We are now in the position  to present and analyze the error reduction result.
\begin{theorem}\label{error-reduction}
Let $\lambda \in \mathbb{R}$  be some eigenpair of (\ref{eigen}) with multiplicity $q$, $\{u_l\}_{l=1}^{q}$ be any orthonormal basis of $M(\lambda)$, and
$\{(\lambda_{h_k, l}, u_{h_k, l}), l=1,\cdots,q\}_{k\in \mathbb{N}_0}$ be a sequence of finite element
 solutions
 produced by {\bf Algorithm \ref{algorithm-AFEM-eigen}}.
 Then there exist constants $\gamma>0$ and $\alpha \in (0,1)$,
  depending only on the shape regularity of meshes, $ C_a $ and $c_a$, the parameter $\theta$ used by
 {\bf Algorithm \ref{algorithm-AFEM-eigen}},
   such that for any two consecutive iterates $k$ and
 $k+1$, we have
  \begin{eqnarray}\label{error-reduction-neq}
 \|U-E_{h_{k+1}} U\|_{a,\Omega}^2+\gamma \eta^2_{h_{k+1}}(E_{h_{k+1}} U, \Omega)
   \leq \alpha^2  \left(\|U-E_{h_k} U\|_{a,\Omega}^2
   +\gamma \eta^2_{h_{k}}(E_{h_{k}} U, \Omega)\right)
  \end{eqnarray}
 provided $h_0\ll 1$.
  Therefore, {\bf Algorithm \ref{algorithm-AFEM-eigen}} converges with a linear rate $\alpha$,
 namely, the $n$-th iterate solution $(\lambda^{h_n, l}, E_{h_n}u_{l}) (l=1, \cdots, q)$ of
 {\bf Algorithm $C$} satisfies
\begin{eqnarray}\label{convergence-neq-1}
\|U-E_{h_n}U\|_{a,\Omega}^2+\gamma \eta^2_{h_{n}}(E_{h_{n}} U, \Omega)  &\leq C_0 \alpha^{2n}
  \end{eqnarray}
and
\begin{eqnarray}\label{convergence-eigenvalue}
  \lambda^{h_n, l} - \lambda \lc \alpha^{2n},
\end{eqnarray}
  where $C_0= \|U-E_{h_0} U\|_{a,\Omega}^2+\gamma \eta^2_{h_{0}}(E_{h_{0}} U, \Omega).$
\end{theorem}

\begin{proof}
For convenience, we use  $(\lambda_{h,l}, u_{h,l})$, $(\lambda_{H,l}, u_{H, l})$ to denote $(\lambda_{h_{k+1}, l}, u_{h_{k+1}, l})$ and
$(\lambda_{h_k,l}, u_{h_k,l})$, respectively. We see that it is sufficient to prove
\begin{eqnarray*}\label{error-reduction-neq-2}
\|U-E_h U\|_{a,\Omega}^2 + \gamma \eta^2_{h}(E_{h}U, \Omega)\leq \alpha^2 \big(\|U-E_H U\|_{a,\Omega}^2 +\gamma \eta^2_{H}(E_H U, \Omega)\big).
\end{eqnarray*}

We derive from Lemma \ref{eta-etah} that {\bf D\"{o}rfler  marking strategy}
implies that there exists a constant $\theta' \in (0,1)$ such that
\begin{eqnarray}\label{eta-etah2}
\sum_{T\in \mathcal{M}_H} \eta_H^2(E_H U, T) \geq \theta' \eta_H^2(E_H U, \Omega).
\end{eqnarray}
Recall that $w^{H, l} = \sum_{i=1}^{q} \alpha_{H,i}(u_l) \lambda_{H, i} K u_{H, i}$, we get from (\ref{eta-etah2}) that  for  $W^H\equiv(w^{H,
1},\cdots,w^{H,q})$, {\bf  D\"{o}rfler marking strategy} is satisfied with $\theta = \theta'$.  So, we conclude from Theorem \ref{convergence-boundary}   that there
exist constants $\tilde{\gamma}>0$ and $\xi\in (0,1)$ satisfying
\begin{eqnarray}\label{error-reduction-neq-4}
\|W^{H}-R_h W^{H}\|_{a,\Omega}^2 +\tilde{\gamma} \tilde{\eta}^2_h(R_h W^{H}, \Omega) \leq  \xi^2\big(\|W^{H}- E_H U \|_{a,\Omega}^2+ \tilde{\gamma}
\eta^2_H(E_H U, \Omega)\big),
\end{eqnarray}
where  the fact $E_H U = R_H W^{H}$ and $  \tilde{\eta}^2_h(R_H W^{H}, \Omega) =\eta^2_h(E_H U, \Omega)$ are used.

 From (\ref{thm1-neq1}), we get that there exists constant $\hat{C}_1>0$ such that
 \begin{eqnarray}\label{convergence-neq2}
\Big(1+\hat{C}_1 \tilde{r} (h_0)\Big) \|U-E_H U\|_{1,\Omega}^2 + \tilde{\gamma}  \eta^2_{H}(E_H U, \Omega) \ge \|W^{H}- E_H U\|_{1,\Omega}^2
+ \tilde{\gamma}  \eta^2_H(E_H U, \Omega).  \quad
\end{eqnarray}
By Lemma \ref{lemma-bound-eigen} and the Young's inequality,  we have that for any $\delta_1 \in (0, 1)$, there exists constant $ \hat{C}_2 >0$ such that
\begin{eqnarray}\label{convergence-neq3}
\|U-E_h U\|_{1,\Omega}^2 +  \tilde{\gamma} \eta^2_{h}(E_h U,\Omega) \le (1+\delta_1) \|W^{H}-R_h W^{H}\|_{1,\Omega}^2
+  (1+\delta_1) \tilde{\gamma} \tilde{\eta}^2_h(R_h W^{H}, \Omega) \nonumber \\
+\hat{C}_2(1 + \delta_1^{-1}) \tilde{r}^2(h_0) \big(\|U-E_h U\|_{1,\Omega}^2 + \|U-E_H U\|_{1,\Omega}^2\big).\quad
\end{eqnarray}
Here, we choose $\delta_1$ satisfying $(1+\delta_1) \xi <1$.

Combining~\eqref{error-reduction-neq-4}, \eqref{convergence-neq2} with \eqref{convergence-neq3}, we get that
\begin{eqnarray*}
&&\Big(1 - \hat{C}_2(1 + \delta_1^{-1}) \tilde{r}^2(h_0)\Big)\|U-E_h U\|_{1,\Omega}^2+ \tilde{\gamma} \eta^2_{h}(E_h U, \Omega)\nonumber\\
&\leq& \Big((1+\delta_1) \xi^2 +  (1+\delta_1) \xi^2 \hat{C}_1 \tilde{r}(h_0)  + \hat{C}_2(1+\delta_1^{-1})\tilde{r}^2(h_0)\Big) \|U-E_H U\|_{1,\Omega}^2 \nonumber\\
&& +(1+\delta_1)) \xi^2 \tilde{\gamma} \eta^2_{H}(E_H U, \Omega). \quad
\end{eqnarray*}
Since $h_0\ll 1$ implies ${\tilde{r}}(h_0)\ll 1$, there holds
\begin{eqnarray*}
&&\|U-E_h U\|_{1,\Omega}^2+\frac{\tilde{\gamma}}{1-\hat{C}_3\delta_1^{-1} \tilde{r}^2(h_0)}\eta^2_{h}(E_h U, \Omega) \\
&\leq& \frac{(1+\delta_1)  \xi^2+\hat{C}_3 \tilde{r}(h_0)}{1-\hat{C}_3\delta_1^{-1}\tilde{r}^2(h_0)}
\left(\|U-E_H U\|_{1,\Omega}^2 +\frac{ \xi^2\tilde{\gamma}}{(1+\delta_1)   \xi^2
+\hat{C}_3 \tilde{r}(h_0)}\eta^2_{H}(E_H U, \Omega)\right),
\end{eqnarray*}
with $\hat{C}_3$ some constant depending on $\hat{C}_1$ and $\hat{C}_2$ when $h_0 \ll 1$.
Besides, we see that the constant $\alpha$ defined by
\begin{eqnarray*}
\alpha = \left(\frac{ (1+ \delta_1)  \xi^2 +  \hat{C}_3
\tilde{r}(h_0)}{1 - \hat{C}_3 \delta_1^{-1} \tilde{r}^2(h_0)}\right)^{1/2}
\end{eqnarray*}
satisfies $\alpha \in (0,1)$ when $h_0\ll 1$.

Finally, we arrive at (\ref{error-reduction-neq}) by using the fact that
\begin{eqnarray*}
\frac{ \xi^2 \tilde{\gamma}}{ (1+ \delta_1)
\xi^2 + \hat{C}_3\tilde{r}(h_0) }<\gamma,
\end{eqnarray*}
where
\begin{eqnarray}\label{gamma}
\gamma = \frac{\tilde{\gamma}}{1 - \hat{C}_3 \delta_1^{-1}\tilde{r}^2(h_0)}.
\end{eqnarray}
This completes the proof.
\end{proof}

Similar to Theorem \ref{thm-error-estimator-space}, we can also get the convergence rate for the gap between
$M(\lambda)$ and its finite elements approximation $M_{h_k}(\lambda)$.
\begin{theorem}\label{thm-convergence-rate-eigenspace}
Let  $\lambda \in \mathbb{R}$  be some eigenvalue of (\ref{eigen}) with multiplicity $q$ and the corresponding eigenspace being $M(\lambda) = span\{u_1, \cdots, u_q\}$,
$\{(\lambda_{h_k, l}, u_{h_k, l}), l=1,\cdots,q\}_{k\in \mathbb{N}_0}$ be a sequence of finite element
 solutions
 produced by {\bf Algorithm \ref{algorithm-AFEM-eigen}}. Set $M_{h_k}(\lambda) = span\{u_{h_k, 1}, \cdots, u_{h_{k}, q}\}$.  If $h_0\ll 1$,
 then there exists constant  $\alpha \in (0,1)$,
  depending only on the shape regularity of meshes, $ C_a $ and $c_a$, the parameter $\theta$ used by
 {\bf Algorithm \ref{algorithm-AFEM-eigen}}, such that
 {\bf Algorithm \ref{algorithm-AFEM-eigen}} satisfies
\begin{eqnarray}\label{thm-4.2-conc}
\delta^2_{H_0^1(\Omega)}(M(\lambda), M_{h_k}(\lambda)) \lc \alpha^{2k}.
  \end{eqnarray}
\end{theorem}
\begin{proof}
Let $\{u_1, \cdots, u_q\}$ be an orthonormal basis of $M(\lambda)$.
For any $u\in M(\lambda)$ with $\|u\|_b$ = 1, there are $\{\alpha_1,\alpha_2,\cdots,
\alpha_q\}\subset R$ such that $u = \sum_{l=1}^q \alpha_l u_l$, and $\sum_{l=1}^{q} \alpha_l^2 = 1$. Therefore,
we may estimate as follows
 \begin{eqnarray}\label{thm-4.2-neq-1}
 &&\| u - E_{h_k} u\|^2_{a, \Omega} = \|\sum_{l=1}^q \alpha_l(u_l - E_{h, k} u_l)\|^2_{a, \Omega} \nonumber\\
 &\leq & \sum_{l=1}^{q} |\alpha_l|^2 \sum_{l=1}^{q} \|u_l-E_{h_k}u_l\|_{a, \Omega}  = \sum_{l=1}^{q}  \|u_l-E_{h_k}u_l\|^2_{a, \Omega} \nonumber\\
 &\leq &\sum_{l=1}^{q} \|u_l - E_{h_k} u_l\|_{a,\Omega}^2+\gamma \eta^2_{h_{k}}(E_{h_{k}} u_l,  \Omega) \lc \alpha^{2k},
\end{eqnarray}
where (\ref{convergence-neq-1}) is used in the last inequality above and $\alpha \in (0, 1)$ is the one in Theorem \ref{error-reduction}.
Therefore, from the definition of $d^2_{H_0^1(\Omega)}(M(\lambda), M_{h_k}(\lambda))$, we get
\begin{eqnarray}\label{thm-4.2-neq-2}
d^2_{H_0^1(\Omega)}(M(\lambda), M_{h_k}(\lambda)) \lc \alpha^{2k}
\end{eqnarray}

In further, we get
\begin{eqnarray}\label{thm-4.2-neq-3}
d^2_{H_0^1(\Omega)}(M_{h_k}(\lambda), M(\lambda)) \lc \alpha^{2k}
  \end{eqnarray}
  from  Lemma \ref{thm-dxy-dyx} and (\ref{thm-4.2-neq-2}).

  Combining (\ref{thm-4.2-neq-1}), (\ref{thm-4.2-neq-2}), and the definition of $\delta_{H_0^1(\Omega)}(M_{h_k}(\lambda), M(\lambda))$,
  we obtain (\ref{thm-4.2-conc}).

 %
\end{proof}

\section{Complexity}\Label{complexity}\setcounter{equation}{0}


As \cite{cascon-kreuzer-nochetto-siebert08, dai-xu-zhou08}, to analyze the complexity of {\bf Algorithm \ref{algorithm-AFEM-eigen}},
we first introduce a function approximation class as follows
\begin{eqnarray*}
\mathcal{A}_{\gamma}^s=\{v \in H: |v|_{s,\gamma} < \infty \},
\end{eqnarray*}
where $\gamma>0$ is some constant,
\begin{eqnarray*}
|v|_{s,\gamma} = \sup_{\varepsilon
>0}\varepsilon \inf_{\{\mathcal{T}\subset \mathcal{T}_{h_0}:
\inf(\|v-v_{\mathcal{T}}\|_{1,\Omega}^2 + (\gamma +1) osc^2_{\mathcal{T}}(v_{\mathcal{T}}, \mathcal{T}))^{1/2} \leq \varepsilon\}} \big(\#\mathcal{T} -
\# \mathcal{T}_{h_0}\big)^s
\end{eqnarray*}
and  $\mathcal{T}\subset \mathcal{T}_{h_0}$ means $\mathcal{T}$ is a refinement of $\mathcal{T}_{h_0}$. It is seen from the definition that, for all
$\gamma>0$, $\mathcal{A}_{\gamma}^s = \mathcal{A}_{1}^s$. For simplicity, here and hereafter, we use $\mathcal{A}^s$ to stand for  $\mathcal{A}_{1}^s$,
and use $|v|_{s}$ to denote
 $|v|_{s, \gamma}$. So $\mathcal{A}^s$ is the class of
functions that can be approximated within a given tolerance $\varepsilon$ by continuous piecewise polynomial functions over a partition $\mathcal{T}$
with number of degrees of freedom $\#\mathcal{T}-\# \mathcal{T}_{h_0} \leq \varepsilon^{-1/s} |v|_{s}^{1/s}$.

We know that in each mesh $\mathcal{T}_{h_k}$, $w^{h_k, l} = \sum_{i=1}^{q} \alpha_{h_k,i}(u_l) \lambda_{h_k, i} K u_{h_k, i}$ is the solution of the following boundary value problem
\begin{eqnarray}\label{auxiliary-boundary-eq}
 a(w^{h_k, l}, v) = (\sum_{i=1}^{q} \alpha_{h_k,i}(u_l) \lambda_{h_k, i} u_{h_k, i}, v) ~~ \forall v\in
 V_{h_k}, ~~l = 1, \cdots, q.
\end{eqnarray}
Thanks to Theorem \ref{thm-eigen-boundary} and Lemma \ref{lemma-bound-eigen} and their proofs, we are able to analyze the complexity of adaptive finite element method for
multiple eigenvalue problems by using the complexity result for boundary value problems, which is similar to what was demonstrated in the convergence analysis.

Using the similar procedure as in the proof of Theorem \ref{error-reduction}, we have
\begin{lemma}\label{complexity-eigen-boundary}
Let $\lambda \in \mathbb{R}$  be some eigenvalue of (\ref{eigen}) with multiplicity $q$, and $\{u_l\}_{l=1}^q$ be an orthonormal basis of $M(\lambda)$.
Let $(\lambda_{h_k, l}, u_{h_k, l}) \in \mathbb{R}\times V_{h_k}$ and $(\lambda_{h_{k+1}, l},u_{h_{k+1}, l}) \in \mathbb{R} \times V_{h_{k+1}}$ $(l=1,
\cdots, q)$ be discrete
 solutions of (\ref{fe-eigen}) over a
conforming mesh $\mathcal{T}_{h_k}$ and its refinement $\mathcal{T}_{h_{k+1}}$ with marked
 element $\mathcal{M}_{h_{k}}$. Suppose they
 satisfy the following property
 \begin{eqnarray*}\label{lemma-complexity-eigen-bound-cond2}
  &&\|U - E_{h_{k+1}} U \|_{a,\Omega}^2 + \gamma_{\ast} osc^2_{h_{k+1}}(E_{h_{k+1}} U, \Omega)\nonumber\\
  & \leq&   \beta_{\ast}^2 \Big( \|U-E_{h_k} U\|_{a,\Omega}^2 + \gamma_{\ast} osc^2_{h_{k}}(E_{h_{k}}U, \Omega)\Big),
 \end{eqnarray*}
 where $ \gamma_{\ast} >0, \beta_{\ast}>0$ are some constants.
Then for the associated boundary value problem (\ref{auxiliary-boundary-eq}), we have
\begin{eqnarray*}\label{lemma-complexity-eigen-bound-conc2}
&&  \|W^{h_k} - R_{h_{k+1}} W^{h_k} \|_{a,\Omega}^2 + \tilde{\gamma}_{\ast} \widetilde{osc}^2_{h_{k+1}}(R_{h_{k+1}} W^{h_k},
\Omega)\nonumber\\
   &\leq& \tilde{\beta}_{\ast}^2 \Big( \|W^{h_k} -R_{h_k} W^{h_k}\|^2_{a,\Omega}
+ \tilde{\gamma}_{\ast} \widetilde{osc}^2_{h_k}(E_{h_k} U, \Omega)\Big)
\end{eqnarray*}
with
\begin{eqnarray}\label{complexity-eigen-boundary-beta-gamma}
\tilde{\beta}_{\ast} = \left(\frac{\beta_{\ast}^2(1+\delta_1)+\hat{C}_4
\tilde{r}(h_0)}{1-\hat{C}_4 \delta_1^{-1} \tilde{r}^2(h_0)}\right)^{1/2}, \quad
\tilde{\gamma}_{\ast} =
 \frac{\gamma_{\ast}}{1 -\hat{C}_4 \delta_1^{-1} \tilde{r}^2(h_0)},
\end{eqnarray}
where $\hat{C}_4 $ is some positive constant depending on   $A$ , $\tilde{C}$ and $C_{\ast}$, $\delta_1 \in (0, 1)$ is some constant as shown in the proof Theorem
\ref{error-reduction}.
\end{lemma}
\begin{proof}
We observe from Lemma \ref{lemma-bound-eigen} that
\begin{eqnarray*}
\|U-E_h U  \|_{1,\Omega}&=& \|W^{H}-R_h W^{H} \|_{1,\Omega}
+ \mathcal{O} (\tilde{r}(h_0)) \left( \|W^{H}-R_H W^{H} \|_{1,\Omega}+ \|W^{H}-R_h W^{H}\|_{1,\Omega}\right),
\end{eqnarray*}
\vskip -0.6cm
\begin{eqnarray*}
{\rm osc}_h(E_h U,\Omega)&=& \widetilde{{\rm osc}}_h(R_h W^{H}, \Omega) +
\mathcal{O} (\tilde{r}(h_0)) \left( \|W^{H}-R_H W^{H}  \|_{1,\Omega} + \|W^{H}-R_h W^{H} \|_{1,\Omega}\right),
\end{eqnarray*}

Proceeding the similar procedure as in the proof of Theorem  \ref{error-reduction}, we have
\begin{eqnarray}\label{lemma-complexity-eigen-bound-conc2}
\|W^H-R_h W^H\|_{a,\Omega}^2 + \tilde{\gamma}_{\ast} \widetilde{{\rm osc}}_h^2 (R_h W^H, \Omega)
\leq \tilde{\beta}_{\ast}^2 \big(\|W^H-R_H W^H\|_{a,\Omega}^2 +
\tilde{\gamma}_{\ast} \widetilde{{\rm osc}}_H^2 (R_H W^H, \Omega)\big)
\end{eqnarray}
with
\begin{eqnarray}\label{complexity-eigen-boundary-beta-gamma}
\tilde{\beta}_{\ast} = \left(\frac{\beta_{\ast}^2(1+\delta_1)+\hat{C}_4
\tilde{r}(h_0)}{1-\hat{C}_4 \delta_1^{-1} \tilde{r}^2(h_0)}\right)^{1/2}, \quad
\tilde{\gamma}_{\ast} =
 \frac{\gamma_{\ast}}{1 -\hat{C}_4 \delta_1^{-1} \tilde{r}^2(h_0)},
\end{eqnarray}
where $\hat{C}_4 $ is some positive constant and $\delta_1 \in (0, 1)$ is some constant as shown
in the proof of Theorem \ref{error-reduction}.
 This completes the
proof.
\end{proof}

The following statement is a direct consequence of $E_{h_k} u_l=R_{h_k}w^{h_k, l}$, Proposition \ref{complexity-boundary-optimal-marking} and Lemma
\ref{complexity-eigen-boundary}.
\begin{corollary}\label{complexity-eigen-optimal-marking}
Let $\lambda \in \mathbb{R}$  be some eigenpair of (\ref{eigen}) with multiplicity $q$, and $\{u_l\}_{l=1}^q$ be an orthonormal basis of $M(\lambda)$.
 Suppose that they satisfy the decrease property
\begin{eqnarray*}
&& \|U - E_{h_{k+1}} U \|_{a,\Omega}^2 + \gamma_{\ast} osc^2_{h_{k+1}}(E_{h_{k+1}} U, \Omega) \nonumber\\&&\leq
   \beta_{\ast}^2 \Big(\|U-E_{h_k} U\|_{a,\Omega}^2 + \gamma_{\ast} osc^2_{h_{k}}(E_{h_{k}} U, \Omega)\Big)
\end{eqnarray*}
with constants $\gamma_{\ast}>0$  and $\beta_{\ast}\in (0,\sqrt{\frac{1}{2}})$. Then the set $\mathcal{R} :=\mathcal{R}_{\mathcal{T}_{h_k}\to \mathcal{T}_{h_{k+1}}}$
satisfies the following inequality
\begin{eqnarray*}\label{complexity-eigen-optimal-marking-neq1}
\sum_{T \in \mathcal{R}
  } \eta^2_{h_k}(E_{h_k} U, T) \geq \hat{\theta} \sum_{T \in \mathcal{T}_{h_k}}
  \eta^2_{h_k}(E_{h_k} U, T)
\end{eqnarray*}
with $\hat{\theta} = \frac{\tilde{C}_2^2(1-2\tilde{\beta}_{\ast}^2)}{\tilde{C}_0 ( \tilde{C}_1^2 + (1 +  2 C_{\ast}^2 \tilde{C}_1^2)
\tilde{\gamma}_{\ast})}$ and $\tilde{C}_0 = \max(1, \frac{\tilde{C}_3^2}{\tilde{\gamma}_{\ast}})$,  where $\tilde{\beta_{\ast}}$ and
$\tilde{\gamma_{\ast}}$ are defined in (\ref{complexity-eigen-boundary-beta-gamma})  with $\delta_2$ being chosen such that $\tilde{\beta_{\ast}}^2\in
(0,\frac{1}{2})$.
\end{corollary}

To analyze the complexity of {\bf Algorithm 3.1}, we need more requirements than for the convergence rate.

\begin{assumption}\label{assump-5.1}
\begin{enumerate}
\item  The marking parameter $\theta$ satisfy $\theta \in (0, \theta_{\ast})$, with
\begin{eqnarray*}
\theta_{\ast} = \frac{1}{q^2} \frac{C_2^2 \gamma}{C_3^2 ( C_1^2 + (1 + 2C_{\ast}^2 C_1^2)\gamma )}). 
\end{eqnarray*}
\item The marked $\mathcal{M}_k$ satisfy (\ref{dorfler-prop}) with minimal cardinality.
\item The distribution of refinement edges on $\mathcal{T}_0$ satisfies condition (b) of section 4
in \cite{stevenson08}.
\end{enumerate}
\end{assumption}
\begin{lemma}\label{complexity-upper-bound-of-dof} Let $\lambda \in \mathbb{R}$  be some eigenvalue of (\ref{eigen}) with multiplicity $q$,
$\{u_l\}_{l=1}^q$ be an orthonormal basis of $M(\lambda)$,  $u_l \in \mathcal{A}^s (l=1, \cdots, q )$ and $\mathcal{T}_{h_k}$ be a conforming partition
obtained from $\mathcal{T}_{h_0}$.   Let $\mathcal{T}_{h_{k+1}}$ be a mesh created from $\mathcal{T}_{h_k}$ upon making the set $\mathcal{M}_{h_{k}}$
which satisfies D\"{o}rfler property (\ref{dorfler-prop})  with $\theta\in (0, \frac{1}{q^2} \frac{C_2^2 \gamma}{C_3^2 ( C_1^2 + (1 + 2C_{\ast}^2 C_1^2)\gamma )})$(that is, 1 and 2 of Assumption \ref{assump-5.1} are satisfied).
Let $\{(\lambda_{h_k, l}, u_{h_k, l}), l=1,\cdots,q\}_{k\in \mathbb{N}_0}$ be discrete
 solutions of (\ref{fe-eigen}) over a
conforming mesh $\mathcal{T}_{h_k}$ and $M_{h_{k}}(\lambda) = span \{u_{h_k, 1}, \cdots, u_{h_k, q}\}$.
  Then
\begin{eqnarray}\label{complexity-optimal-neq1}
 \# \mathcal{M}_{h_{k}}   \leq  C
  \Big(  \|U - E_{h_k} U\|_{a,\Omega}^2 + \gamma osc^2_{h_{k}}(E_{h_{k}} U, \Omega) \Big)^{-1/2s}
   \Big(q^{\frac{1}{2s} -1} \sum_{l=1}^{q} |u_l|_{s}^{1/s}\Big),
\end{eqnarray}
where the constant $C$ depends on the discrepancy between $\theta$ and $\frac{C_2^2 \gamma}{C_3^2 ( C_1^2 + (1 + 2C_{\ast}^2 C_1^2)\gamma )}$.
\end{lemma}
\begin{proof}
Let $\beta, \beta_1 \in (0,1)$ satisfy $\beta_1\in (0,\beta) $ and
 $$\theta < \frac{1}{q^2}\frac{C_2^2 \gamma}{C_3^2 (
C_1^2 + (1 + 2C_{\ast}^2 C_1^2)\gamma )}(1-\beta^2).$$
Choose
$$\varepsilon = \frac{1}{\sqrt{2}} \beta_1 \Big(\|U-E_{h_k} U\|_a^2 +
\gamma osc^2_{h_{k}}(E_{h_{k}} U, \Omega)\Big) ^{1/2}$$
and let $\mathcal{T}_{h_{\varepsilon}}$ be a refinement of
 $\mathcal{T}_{h_0}$ with minimal degrees of freedom satisfying
 \begin{eqnarray*}
 \|u_l - E_{h_{\varepsilon}} u_l\|_{a,\Omega}^2 + (\gamma + 1) osc^2_{h_{\varepsilon}}(E_{h_{\varepsilon}} u_l, \Omega) \leq
 \frac{\varepsilon^2}{q}, ~~l=1, \cdots, q,
 \end{eqnarray*}
 which means
\begin{eqnarray}\label{complexity-optimal-neq0}
  \|U - E_{h_{\varepsilon}} U\|_{a,\Omega}^2 + (\gamma + 1) osc^2_{h_{\varepsilon}}(E_{h_{\varepsilon}} U, \Omega) \leq
 \varepsilon^2.
 \end{eqnarray}

We get from the definition of $\mathcal{A}^s$ that
  \begin{eqnarray*}
  ~~&&\#\mathcal{T}_{h_{\varepsilon}} - \# \mathcal{T}_{h_0} \nonumber\\
  &\leq& ( \frac{1}{\sqrt{2}}\beta_1)^{-1/s}
 \Big(\|U - E_{h_k} U\|_{a,\Omega}^2 + \gamma osc^2_{h_{k}}(E_{h_{k}} U, \Omega)\Big)^{-1/2s} q^{1/2s} |u_l|_{s}^{1/s}, ~~l=1, \cdots, q,
 \end{eqnarray*}
 which implies
  \begin{eqnarray*}\label{upper-bound-dof-neq1}
  ~~&&\#\mathcal{T}_{h_{\varepsilon}} - \# \mathcal{T}_{h_0} \nonumber\\
  &\leq& ( \frac{1}{\sqrt{2}}\beta_1)^{-1/s}
   \Big(\|U - E_{h_k} U\|_{a,\Omega}^2 + \gamma osc^2_{h_{k}}(E_{h_{k}} U, \Omega)\Big)^{-1/2s}
\Big(q^{\frac{1}{2s} -1} \sum_{l=1}^q |u_l|_{s}^{1/s}\Big).
 \end{eqnarray*}
 Let $\mathcal{T}_{h_{k,+}}$ be the smallest common refinement of
 $\mathcal{T}_{h_k}$ and $\mathcal{T}_{h_{\varepsilon}}$. Note that both $\mathcal{T}_{h_k}$ and
 $\mathcal{T}_{h_{\varepsilon}}$ are refinements of $\mathcal{T}_{h_0}$,
 we have that the number of elements in $\mathcal{T}_{h_{k,+}}$ that are not in $\mathcal{T}_{h_k}$
is less than the number of elements that must be added to go from $\mathcal{T}_{h_0}$ to $\mathcal{T}_{h_{\varepsilon}}$, namely,
\begin{eqnarray*}
 \#\mathcal{T}_{h_{k,+}} - \#\mathcal{T}_{h_k} \leq
 \#\mathcal{T}_{h_{\varepsilon}} - \# \mathcal{T}_{h_0}.
\end{eqnarray*}

Let $w^{h_\varepsilon,l} = \sum_{i=1}^{q} \alpha_{h_\varepsilon,i}(u_l) \lambda_{h_\varepsilon, i}  K u_{h_\varepsilon, i} = K\big(\sum_{i=1}^{q} \alpha_{h_\varepsilon,i}(u_l) \lambda_{h_\varepsilon, i} u_{h_\varepsilon, i}\big)$ $(l=1, \cdots q)$, namely
\begin{eqnarray*}\label{complexity-boundary-problem-2}
 L w^{h_\varepsilon,l} = \sum_{i=1}^{q} \alpha_{h_\varepsilon,i}(u_l) \lambda_{h_\varepsilon, i} u_{h_\varepsilon, i}.
\end{eqnarray*}
We obtain from Lemma \ref{lemma-osc-L} and  the Young inequality that
\begin{eqnarray*}
 \widetilde{osc}^2_{h_{k,+}}(R_{h_{k,+}} W^{h_{\varepsilon}}, \Omega)
  &\leq&
  2 \widetilde{osc}^2_{h_{\varepsilon}}(R_{h_{\varepsilon}}W^{h_{\varepsilon}},\Omega)
   + 2 C_{\ast}^2 \|R_{h_{k,
  +}}W^{h_{\varepsilon}} -
  R_{h_{\varepsilon}}W^{h_{\varepsilon}}\|_{a, \Omega}^2.
\end{eqnarray*}

Note that  $\mathcal{T}_{h_{k,+}}$  is a
 refinement of $\mathcal{T}_{h_{\varepsilon}}$,
 $L^2$-projection error are monotone and the following orthogonality
\begin{eqnarray*}\label{ortho-relation}
\|W^{h_\varepsilon}-R_{h_{k,
  +}}W^{h_\varepsilon}\|_{a, \Omega}^2
   = \|W^{h_\varepsilon} - R_{h_{\varepsilon}}W^{h_\varepsilon}\|_{a, \Omega}^2
  - \|R_{h_{k,
  +}}W^{h_\varepsilon} -
  R_{h_{\varepsilon}}W^{h_\varepsilon}\|_{a, \Omega}^2
\end{eqnarray*}
is valid, we arrive at
\begin{eqnarray*}
  &&\| W^{h_\varepsilon} - R_{h_{k, +}}
  W^{h_\varepsilon}\|_{a,\Omega}^2 + \frac{1}{2 C_{\ast}^2} \widetilde{osc}^2_{h_{k,+}}(R_{h_{k,+}} W^{h_{\varepsilon}}, \Omega)\nonumber\\
  &\leq & \|W^{h_\varepsilon} - R_{h_{\varepsilon}}
  W^{h_\varepsilon}\|_{a,\Omega}^2+ \frac{1}{ C_{\ast}^2}
  osc^2_{h_{\varepsilon}}(R_{h_{\varepsilon}}W^{h_{\varepsilon}},\Omega).
 \end{eqnarray*}
Since  (\ref{gamma-boundary}) implies $\tilde{\gamma} \leq \frac{1}{2 C_{\ast}^2}$,   we obtain that $\sigma \equiv \frac{1}{C_{\ast}^2} - \tilde{\gamma} \in (0, 1) $ and
\begin{eqnarray*}
  &&\| W^{h_\varepsilon} - R_{h_{k, +}}
  W^{h_\varepsilon}\|_{a,\Omega}^2 + \tilde{\gamma} \widetilde{osc}^2_{h_{k,+}}(R_{h_{k,+}} W^{h_{\varepsilon}}, \Omega)\nonumber\\
  &\leq &  \|W^{h_\varepsilon} - R_{h_{\varepsilon}} W^{h_\varepsilon}\|_{a,\Omega}^2
   + \frac{1}{ C_{\ast}^2}
   osc^2_{h_{\varepsilon}}(R_{h_{\varepsilon}}W^{h_{\varepsilon}},\Omega)\nonumber\\
  &\leq & \|W^{h_\varepsilon} - R_{h_{\varepsilon}}
  W^{h_\varepsilon}\|_{a,\Omega}^2+ (\tilde{\gamma} + \sigma)
  osc^2_{h_{\varepsilon}}(R_{h_{\varepsilon}}W^{h_{\varepsilon}},\Omega).
 \end{eqnarray*}
Applying the similar argument in the proof of Theorem \ref{error-reduction}, we may conclude that
 \begin{eqnarray}\label{complexity-optimal-neq2}
 && \|U - E_{h_{k, +}} U\|_{a,\Omega}^2 + \gamma osc^2_{h_{k,+}}(E_{h_{k, +}} U,
 \Omega)\nonumber\\
 & \leq& \beta_0^2  \Big(
 \|U - E_{h_{\varepsilon}} U\|_{a,\Omega}^2 + (\gamma + \sigma) osc^2_{h_{\varepsilon}}(E_{h_{\varepsilon}} U,
 \Omega)\Big) \nonumber\\
&\leq &  \beta_0^2 \Big(\|U - E_{h_{\varepsilon}} U\|_{a,\Omega}^2
     + (\gamma + 1)
osc^2_{h_{\varepsilon}}(E_{h_{\varepsilon}} U,
 \Omega)\Big),
 \end{eqnarray}
 where
 \begin{eqnarray*}
\beta_0 = \left(\frac{ 1+ \delta_1+\hat{C}_2\tilde{\gamma}^2(h_0)}{1 -\hat{C}_2\tilde{\gamma}^2(h_0)}\right)^{1/2},
\end{eqnarray*}
and $\delta_1$ is the constant appearing in the proof of Theorem \ref{error-reduction}. Combining (\ref{complexity-optimal-neq0}) and
(\ref{complexity-optimal-neq2}), we then arrive at
\begin{eqnarray*}
\|U - E_{h_{k, +}} U\|_{a,\Omega}^2 + \gamma osc^2_{h_{k,+}}(E_{h_{k, +}} U,
 \Omega) \leq \check{\beta}^2  \Big(\|U-E_{h_k} U\|_a^2 + \gamma osc^2_{h_{k}}(E_{h_{k}} U,
 \Omega)\Big)
\end{eqnarray*}
with  $\check{\beta} = \frac{1}{\sqrt{2}} \beta_0 \beta_1$.

 Let $\delta_1\in (0,1)$ be some constant satisfying
\begin{eqnarray}\label{thm-complexity-delta-cond-1}
(1+\delta_1)^2 \beta_1^2 \leq \beta^2,
\end{eqnarray}
which implies
\begin{eqnarray}\label{thm-complexity-delta-cond-3}
(1+\delta_1) \beta_1^2 <1.
\end{eqnarray}
We see from $h_0 \ll1 $ and (\ref{thm-complexity-delta-cond-3})
 that $\check{\beta}^2\in (0,\frac{1}{2})$. Thus we get from  Corollary
\ref{complexity-eigen-optimal-marking}  that
$\mathcal{T}_{h_{k, +}}$ satisfies 
\begin{eqnarray*}
  \sum_{T \in \mathcal{R}_{\mathcal{T}_{h_k}\rightarrow \mathcal{T}_{h_{k, +}}}
  } \eta^2_{h_k}(E_{h_k} U, T) \geq \check{\theta}  \sum_{T \in \mathcal{T}_{h_k}}
  \eta^2_{h_k}(E_{h_k} U, T),
\end{eqnarray*}
where $\check{\theta} = \frac{\tilde{C}_2^2(1-2\hat{\beta}^2)}{\tilde{C}_0 ( \tilde{C}_1^2 + (1 + 2 C_{\ast}^2 \tilde{C}_1^2) \hat{\gamma})}, \quad
\hat{\gamma}= \frac{\gamma}{1 -\hat{C}_3\tilde{\gamma}^2(h_0)}$, $\tilde{C}_0 = max(1, \frac{\tilde{C}_3^2}{\hat{r}})$ and
\begin{eqnarray*}
 \hat{\beta}= \left(\frac{\check{\beta}^2( 1+ \delta_1)+\hat{C}_3\tilde{\gamma}^2(h_0)}{1 -\hat{C}_3\tilde{\gamma}^2(h_0)}\right)^{1/2}.
 \end{eqnarray*}
 From the
 definition of $\gamma$ (see (\ref{gamma})) and $\tilde{\gamma}$ (see (\ref{gamma-boundary})), we
obtain that $\hat{\gamma}<1$. On the other hand, we have $\tilde{C}_3>1$ and hence ${\tilde C}_0 = \frac{\tilde{C}_3^2}{\hat{\gamma}}.$  Consequently, we
can write $\check{\theta}$ as $\check{\theta} = \frac{\tilde{C}_2^2(1-2\hat{\beta}^2)}{\tilde{C}_3^2 ( \frac{\tilde{C}_1^2}{\hat{\gamma}} + (1 + 2
C_{\ast}^2 \tilde{C}_1^2) )}.$

We then obtain from Lemma \ref{mark-etah-eta} that there exists a constant $\check{\theta}' \in (0, 1)$, such that
\begin{eqnarray}
  \sum_{T \in \mathcal{R}_{\mathcal{T}_{h_k}\rightarrow \mathcal{T}_{h_{k, +}}}
  } \eta^2_{h_k}(U_{h_k}, T) \geq \check{\theta}'  \sum_{T \in \mathcal{T}_{h_k}}
  \eta^2_{h_k}(U_{h_k}, T),
\end{eqnarray}
 where $\check{\theta}' = \frac{1 - C \rho_{\Omega}(h_0) \delta_{h_0}(\lambda)}{q^2} \check{\theta}$.

 Since $h_0 \ll 1$,
  we obtain that  $\hat{\gamma}>\gamma$ and $\hat{\beta}\in
  (0,\frac{1}{\sqrt{2}}\beta)$ from (\ref{thm-complexity-delta-cond-1}).
Using (\ref{coef-eigen-bound}), we get that
\begin{eqnarray*}
 \check{\theta}'&=&\frac{1 - C \rho_{\Omega}(h_0) \delta_{h_0}(\lambda)}{q^2} \frac{\tilde{C}_2^2(1-2\hat{\beta}^2)}{\tilde{C}_3^2 (
\frac{\tilde{C}_1^2}{\hat{\gamma}} + (1 + 2 C_{\ast}^2 \tilde{C}_1^2) )} \nonumber\\
& \geq &\frac{1 - C \rho_{\Omega}(h_0) \delta_{h_0}(\lambda)}{q^2} \frac{\tilde{C}_2^2}{\tilde{C}_3^2 (
\frac{\tilde{C}_1^2}{\hat{\gamma}} + (1 + 2 C_{\ast}^2 \tilde{C}_1^2)
)}(1-\beta^2)\nonumber\\
&=& \frac{1 - C \rho_{\Omega}(h_0) \delta_{h_0}(\lambda)}{q^2} \frac{\frac{C_2^2}{(1 - \tilde{C} \tilde{r}(h_0))^2} }{\frac{C_3^2}{(1 - \tilde{C} \tilde{r}(h_0))^2}  \big( \frac{ C_1^2}{(1 + \tilde{C}
\tilde{r}(h_0))^2\hat{\gamma}} + (1 + 2 C_{\ast}^2 \frac{C_1^2}{(1+ \tilde{C}
\tilde{r}(h_0))^2 } )\big)}(1-\beta^2),
\end{eqnarray*}
which together with the fact that $h_0 \ll 1$ and $\hat{\gamma}>\gamma$ yields
\begin{eqnarray*}
 \check{\theta}
\geq \frac{1 }{q^2} \frac{C_2^2 }{C_3^2 ( \frac{C_1^2}{\gamma} + (1 + 2 C_{\ast}^2 C_1^2) )}(1-\beta^2) =\frac{1}{q^2} \frac{C_2^2 \gamma }{C_3^2 ( C_1^2 + (1 + 2 C_{\ast}^2
C_1^2)\gamma )}(1-\beta^2)
> \theta.
\end{eqnarray*}
Therefore,
\begin{eqnarray}
  \sum_{T \in \mathcal{R}_{\mathcal{T}_{h_k}\rightarrow \mathcal{T}_{h_{k, +}}}
  } \eta^2_{h_k}(U_{h_k}, T) \geq \theta  \sum_{T \in \mathcal{T}_{h_k}}
  \eta^2_{h_k}(U_{h_k}, T).
\end{eqnarray}
Since $\mathcal{M}_{h_k}$ satisfies (\ref{dorfler-prop}) with minimal cardinality,   we have
\begin{eqnarray*}
 \#\mathcal{M}_{h_{k}}  &\leq& \# \mathcal{R}_{\mathcal{T}_{h_k}\rightarrow \mathcal{T}_{h_{k, +}}} \leq
 \#\mathcal{T}_{h_{k, +}} - \#\mathcal{T}_{h_k}
  \leq \#\mathcal{T}_{h_{\varepsilon}} - \#\mathcal{T}_{h_0}\nonumber\\
 &\leq& (\frac{1}{\sqrt{2}}\beta_1)^{-1/s}
  \Big(\|U - E_{h_k} U\|_{a,\Omega}^2 + \gamma osc^2_{h_{k}}(E_{h_{k}} U,
 \Omega)\Big)^{-1/2s}
 \Big(q^{\frac{1}{2s} -1} \sum_{l=1}^q |u_l|_{s}^{1/s}\Big).
\end{eqnarray*}
This is the desired estimate (\ref{complexity-optimal-neq1}) with an explicit dependence on the discrepancy between $\theta$ and $\frac{1}{q^2}\frac{C_2^2
\gamma}{C_3^2 ( C_1^2 + (1 + 2 C_{\ast}^2 C_1^2)\gamma )}$ via $\beta_1$. This completes the proof.
\end{proof}

We are now ready to show that {\bf Algorithm \ref{algorithm-AFEM-eigen}} possesses quasi-optimal complexity.
\begin{theorem}\label{thm-optimal-complexity}
Let $\lambda \in \mathbb{R}$  be some eigenvalue of (\ref{eigen}) with multiplicity $q$, $\{u_l\}_{l=1}^q$ be an orthonormal basis of $M(\lambda)$,
and $u_l \in \mathcal{A}^s (l=1, \cdots, q)$. Let $\{(\lambda_{h_k, l}, u_{h_k, l}), l=1,\cdots,q\}_{k\in \mathbb{N}_0}$ be a sequence of finite element
 solutions
 produced by {\bf Algorithm \ref{algorithm-AFEM-eigen}} of
 Section \ref{adaptive-algorithm} and $M_{h_{k}}(\lambda) = span \{u_{h_k, 1}, \cdots, u_{h_k, q}\}$.
 If  Assumption \ref{assump-5.1}  are satisfied for {\bf Algorithm \ref{algorithm-AFEM-eigen}}, then the $n$-th iterate solution space $M_{h_n}(\lambda)$ of
 {\bf Algorithm \ref{algorithm-AFEM-eigen}} satisfies the quasi-optimal bound
 \begin{eqnarray*}
  \|U-E_{h_n} U\|_{a,\Omega}^2 + \gamma osc^2_{h_n}(E_{h_n} U, \Omega) \lc (\#\mathcal{T}_{h_n}
  -\#\mathcal{T}_{h_0})^{-2s},\\
 \lambda^{h_n,l}-\lambda\lc (\#\mathcal{T}_{h_n}
  -\#\mathcal{T}_{h_0})^{-2s},
 \end{eqnarray*}
 where the hidden constant depends on the exact solution $(\lambda, M(\lambda))$ and
 the discrepancy between
 $\theta$ and $\frac{C_2^2 \gamma}{C_3^2 ( C_1^2 + (1 + 2 C_{\ast}^2 C_1^2)\gamma
)}$.
\end{theorem}
\begin{proof}
We see from (\ref{lower-bound}) that
\begin{eqnarray*}
\|U - E_{h_k} U\|^2_{a, \Omega} + \gamma \eta^2_{h_k}(E_{h_k} U, \Omega) \leq \check{C} \big( \|U - E_{h_k} U\|^2_{a, \Omega} + \gamma
osc^2_{h_k}(E_{h_k} U, \Omega)\big),
\end{eqnarray*}
where $ \check{C} = \max(1 + \frac{\gamma}{C_2^2}, \frac{C_3^2}{C_2^2}).$
 Hence, we get from (\ref{complexity-optimal-neq1}) that
\begin{eqnarray*}
\# \mathcal{M}_{h_{k}}  &\leq&
  (\frac{1}{\sqrt{2}}\beta_1)^{-1/s}  \check{C}^{\frac{1}{2s}}
  \Big(  \|U - E_{h_k} U\|_{a,\Omega}^2 + \gamma \eta^2_{h_k}(E_{h_k} U)\Big)^{-1/2s}
 \Big(q^{\frac{1}{2s} -1}  \sum_{l=1}^q|u_l|_{s}^{1/s}\Big).
\end{eqnarray*}

Note that  Theorem \ref{error-reduction} implies
\begin{eqnarray*}
\|U-E_{h_{k+1}} U\|_{a,\Omega}^2 + \gamma \eta^2_{h_{k+1}}(E_{h_{k+1}} U, \Omega)\leq
   \alpha^2 \Big( \|U-E_{h_k} U\|_{a,\Omega}^2 + \gamma \eta^2_{h_{k}}(E_{h_{k}} U,
   \Omega)\Big).
\end{eqnarray*}
Thus for $0\leq k < n$, we arrive at
\begin{eqnarray*}
&& \Big( \|U - E_{h_k} U\|_{a,\Omega}^2 + \gamma \eta^2_{h_k}(E_{h_k} U, \Omega)\Big)^{-1/2s}\nonumber\\
 & \leq&
 \alpha^{(n-k)/s}\Big( \|U-E_{h_{n}} U\|_{a,\Omega}^2
 + \gamma \eta^2_{h_{n}}(E_{h_{n}} u, \Omega)\Big)^{-1/2s}.
\end{eqnarray*}
We next employ Lemma \ref{complexity-refine}
 to deduce that
\begin{eqnarray*}
 &&\#\mathcal{T}_{h_n}-\#\mathcal{T}_{h_0} \lc \sum_{k=0}^{n-1}
 \# \mathcal{M} _{h_{k}} \nonumber\\
 &\lc& \Big( \sum_{l=1}^q|u_l|_{s}^{1/s} \Big) \sum_{k=0}^{n-1}
 \Big(\|U-E_{h_k} U\|_{a,\Omega}^2 + \gamma \eta^2_{h_k}(E_{h_k} U, \Omega)\Big)^{-1/2s}\nonumber\\
 &\lc& \Big(  \|U-E_{h_n} U\|_{a,\Omega}^2 + \gamma \eta^2_{h_n}(E_{h_n} U, \Omega)\Big)^{-1/2s}  \Big( \sum_{l=1}^q|u_l|_{s}^{1/s}\Big)
 \sum_{k=1}^{n} \alpha^{\frac{k}{s}},
\end{eqnarray*}
which together with the fact $\alpha<1$ leads to
\begin{eqnarray*}
 \#\mathcal{T}_{h_n}-\#\mathcal{T}_{h_0} \lc  \Big( \|U-E_{h_n} U\|_{a,\Omega}^2 +
 \gamma \eta^2_{h_n}(E_{h_n} U, \Omega)\Big)^{-1/2s} \Big(\sum_{l=1}^q|u_l|_{s}^{1/s} \Big).
\end{eqnarray*}
 Since $osc(E_{h_n} U, \Omega) \leq  \eta_{h_n}(E_{h_n} U,
 \Omega)$, we conclude
\begin{eqnarray*}
&& \#\mathcal{T}_{h_n}-\#\mathcal{T}_{h_0} \lc  \Big( \|U-E_{h_n} U\|_{a,\Omega}^2 +
 \gamma osc^2_{h_n}(E_{h_n} U, \Omega)\Big)^{-1/2s} \Big( \sum_{l=1}^q|u_l|_{s}^{1/s} \Big).
\end{eqnarray*}
 This completes the proof.
\end{proof}

Then, similar to Theorem \ref{thm-convergence-rate-eigenspace}, we obtain from Theorem \ref{thm-optimal-complexity} that
\begin{theorem}\label{thm-optimal-complexity-eigenspace}
Let $\lambda \in \mathbb{R}$  be some eigenvalue of (\ref{eigen}) with multiplicity $q$, $\{u_l\}_{l=1}^q$ be an orthonormal basis of $M(\lambda)$,
and $u_l \in \mathcal{A}^s (l=1, \cdots, q)$. Let $\{(\lambda_{h_k, l}, u_{h_k, l}), l=1,\cdots,q\}_{k\in \mathbb{N}_0}$ be a sequence of finite element
 solutions
 produced by {\bf Algorithm \ref{algorithm-AFEM-eigen}} and $M_{h_{k}}(\lambda) = span \{u_{h_k, 1}, \cdots, u_{h_k, q}\}$.
 If   Assumption \ref{assump-5.1} are satisfied for {\bf Algorithm \ref{algorithm-AFEM-eigen}}, then the $n$-th iterate solution space $M_{h_n}(\lambda)$ of
 {\bf Algorithm \ref{algorithm-AFEM-eigen}} satisfies the quasi-optimal bound
 \begin{eqnarray*}
  \delta^2_{H_0^1(\Omega)}(M(\lambda), M_{h_n}(\lambda)) \lc (\#\mathcal{T}_{h_n}
  -\#\mathcal{T}_{h_0})^{-2s},
  \end{eqnarray*}
 where the hidden constant depends on the exact solution $(\lambda, M(\lambda))$ and
 the discrepancy between
 $\theta$ and $\frac{1}{q^2}\frac{C_2^2 \gamma}{C_3^2 ( C_1^2 + (1 + 2 C_{\ast}^2 C_1^2)\gamma
)}$.
\end{theorem}

\section{Numerical examples} \label{numerical-experiments}\setcounter{equation}{0}
In this section,  we  show some numerical examples for both linear finite elements and quadratic finite elements in three dimensions to illustrate
the theoretical results obtained in this paper.

 Our numerical examples were carried out on LSSC-III in
the State Key Laboratory of Scientific and Engineering Computing, Chinese Academy of Sciences, and our codes were based on the toolbox PHG of the State
Key Laboratory of Scientific and Engineering Computing, Chinese Academy of Sciences.


For the convenience of present for our numerical results below, we denote $\eta_h(U_h, \Omega) $ as $\eta_h(M_h(\lambda), \Omega)$, where $U_h = (u_{h, 1}, \cdots, u_{h, q})$, $u_{h, 1}, \cdots, u_{h, q}$ are the $q$ discrete eigenfunctions corresponding to $\lambda$.

{\bf Example 1} Consider the following harmonic oscillator equation, which is a simple model in quantum mechanics \cite{greiner94}:
\begin{eqnarray}\Label{example-harmonic}
-\frac{1}{2} \Delta u +\frac{1}{2} |x|^2 u = \lambda u ~~ \mbox{in} ~\mathbb{R}^3,
\end{eqnarray}
where $|x|=\sqrt{|x_1|^2+|x_2|^2+|x_3|^2}$. The  eigenvalues
 of (\ref{example-harmonic}) are $\lambda_n  = n + \frac{1}{2}$ with multiplicity $\frac{ n(n+1)}{2}(n = 1, 2, \cdots)$, and its associated
eigenfunction is $u_n = \gamma e^{-|x|^2/2} H_n(x)$ with any nonzero constant $\gamma$ and $H_n(x) = (-1)^n e^{x^2} \frac{d^n}{d x^n}(e^{-x^2})$.

Since the solution of (\ref{example-harmonic})  exponentially
decays, we may solve it over some bounded domain $\Omega$. In the computation, we solve the following eigenvalue problem: find $(\lambda, u)\in \mathbb{R}\times H^1_0(\Omega)$ such that
$\displaystyle\int_{\Omega}|u|^2dx=1$ and
\begin{eqnarray*}\Label{example-harmonic-1}
\left\{\begin{array}{rcl}
-\frac{1}{2} \Delta u +\frac{1}{2} |x|^2 u &=& \lambda u  ~~ \mbox{in} ~\Omega, \\[0.2cm]
u&=& 0 ~~ \mbox{on}~ \partial\Omega,
\end{array}
\right.
\end{eqnarray*}
where $\Omega = (-5.5,5.5)^3.$ We calculate the approximation of the first two smallest eigenvalues $\lambda_1$ and $\lambda_2$ with multiplicity $1$ and $3$, respectively, and
their corresponding eigenfunction spaces $M(\lambda_1)$ and $M(\lambda_2)$ with dimension $1$ and $3$, respectively.

 Some cross-sections of the adaptively refined mesh constructed by the {\bf D\"{o}rfler marking strategy}
are displayed in Figure \ref{harmonic-grid}, from which we observe that the mesh is denser in the center of the domain where the solutions oscillate
quickly than in the domain far away from the center where the solution is smoother.  This shows that our adaptively refined mesh can catch the
oscillation of the solution efficiently and the a posteriori error estimators we designed are efficient. Our numerical results are presented in
Figure \ref{figure-error-eigenvalue-harmonic} and Figure \ref{figure-error-harmonic}. Since the multiplicity of the first two smallest eigenvalues are $1$
and $3$, respectively, for the discrete problem, we calculate the first $4$ eigenpairs.  We see from the left figure of Figure
\ref{figure-error-eigenvalue-harmonic} that   the convergence curves of error for all eigenvalues by using linear finite elements are
parallel to the line with slope $-\frac{2}{3}$.  Besides, we also observe that the convergence curves for the second, the third and the forth eigenvalues
overlap together, this coincide with the fact that the multiplicity of the second eigenvalue  is $3$. Meanwhile, from the left figure of  Figure
\ref{figure-error-harmonic} we  see that by using linear finite elements,  the convergence curves of the a posteriori error estimators for
eigenfunction space $\eta_h(M_h(\lambda_1))$ and $\eta_h(M_h(\lambda_1))$ are parallel to the line with slope $-\frac{1}{3}$. From Theorem \ref{thm-error-estimator-space},  $\eta_h(M_h(\lambda),
\Omega) \approx \delta_{H_0^1(\Omega)}(M(\lambda), M_h(\lambda))$, we obtain that  the convergence curves of error for the gap between space $M(\lambda_1)$ and its finite approximation $M_h(\lambda_1)$, the gap between space $M(\lambda_2)$ and its finite approximation $M_h(\lambda_2)$  are also parallel to the line with slope $-\frac{1}{3}$. This  means that the approximation of eigenvalues as well as the eigenfunction space have reached the optimal convergence rate, which coincides with our theory in
Section \ref{convergence-sec} and Section \ref{complexity}.   We have the similar conclusion for the
quadratic  finite elements from the right figures of Figure \ref{figure-error-eigenvalue-harmonic} and Figure \ref{figure-error-harmonic}.


\begin{figure}[htb]
  \centering
  \subfloat[Linear finite elements]{
    \includegraphics[width=6cm]{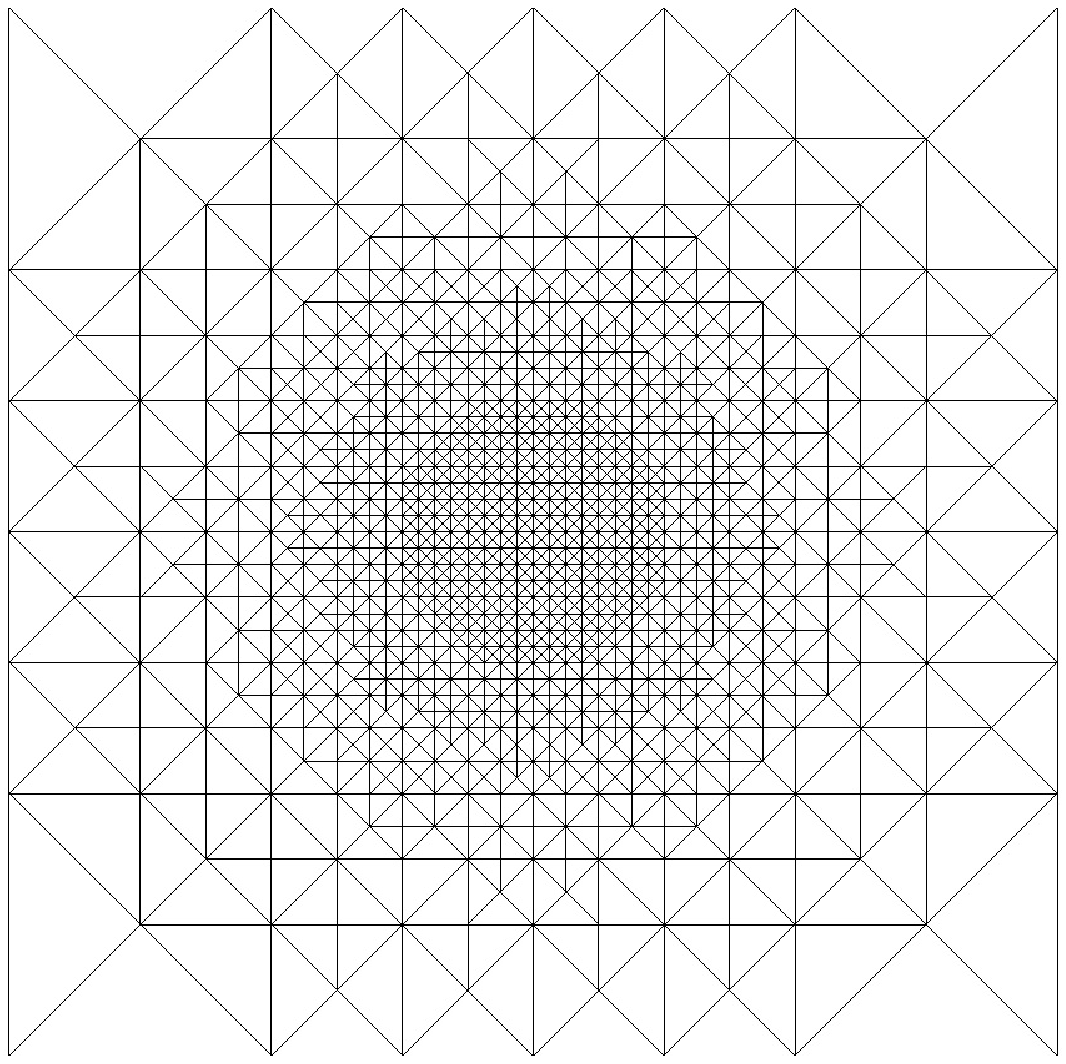}
  }
  \,
  \subfloat[Quadratic finite elements]{
    \includegraphics[width=6.0cm]{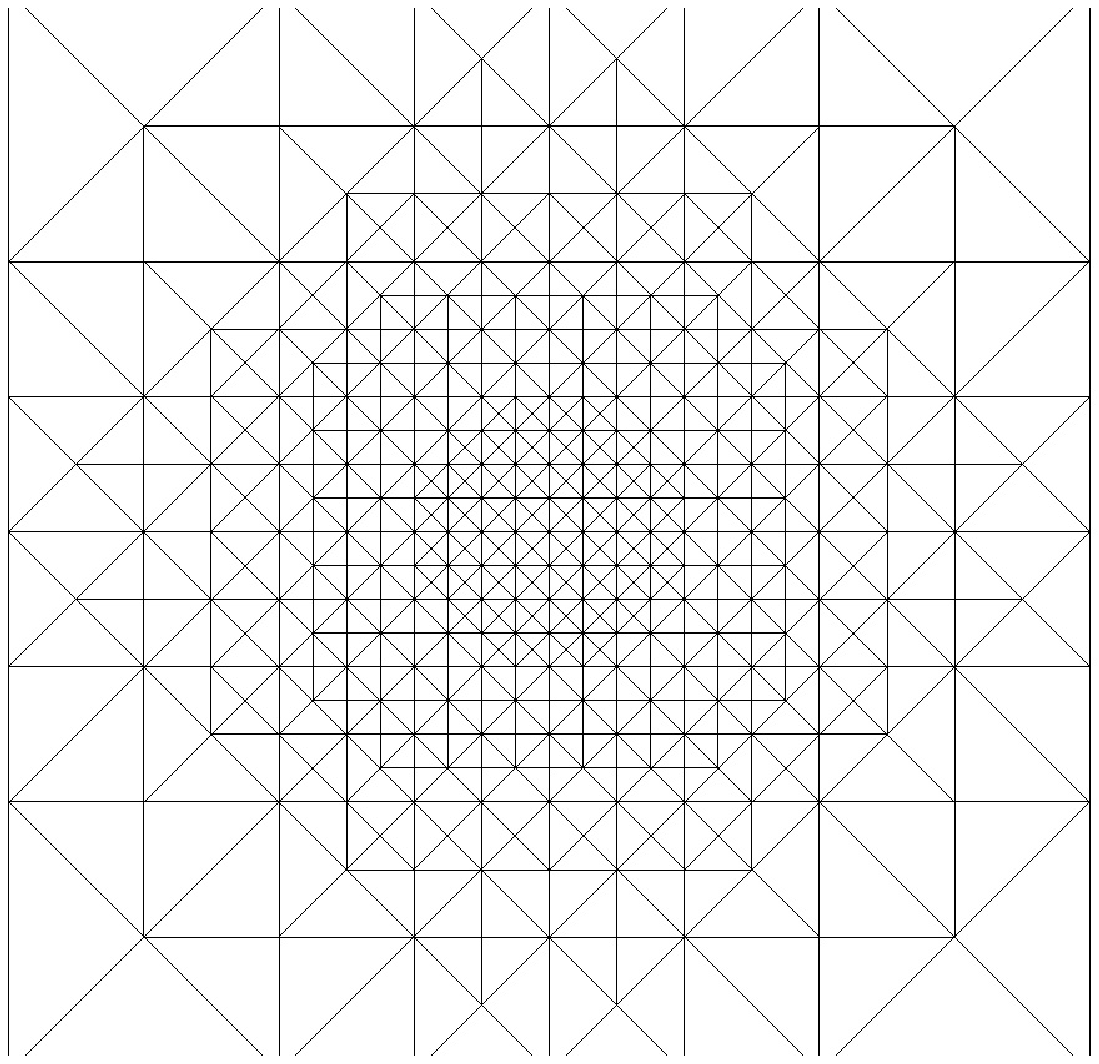}
  }
  \caption{The cross-sections of an
adaptive mesh of {\bf Example 1}}
  \label{harmonic-grid}
\end{figure}

%
%
%

\begin{figure}[htb]
  \centering
  \subfloat[Linear finite elements]{
    \includegraphics[width=5cm,angle=-90]{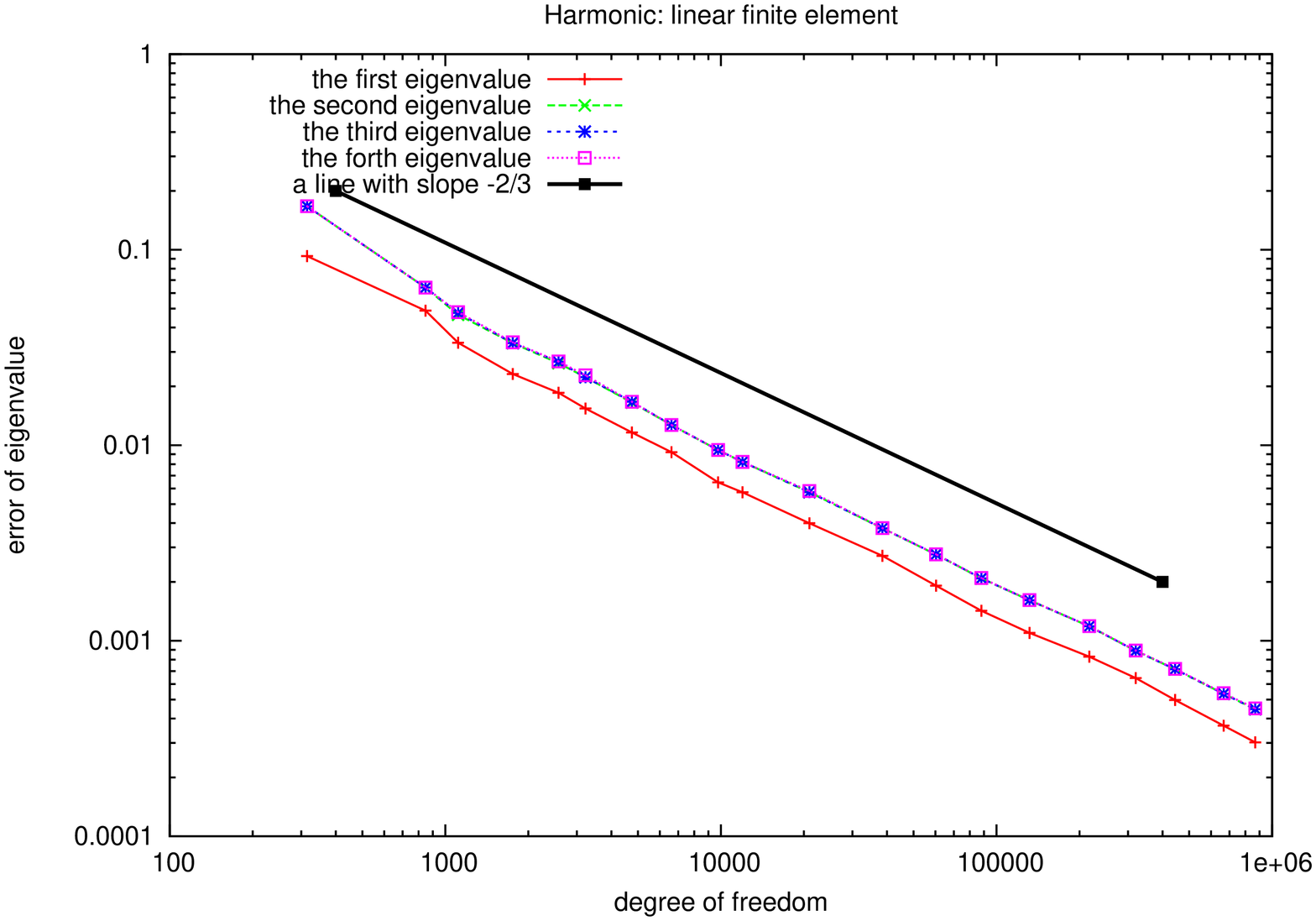}
  }
  \,
  \subfloat[Quadratic finite elements]{
    \includegraphics[width=5.0cm,angle=-90]{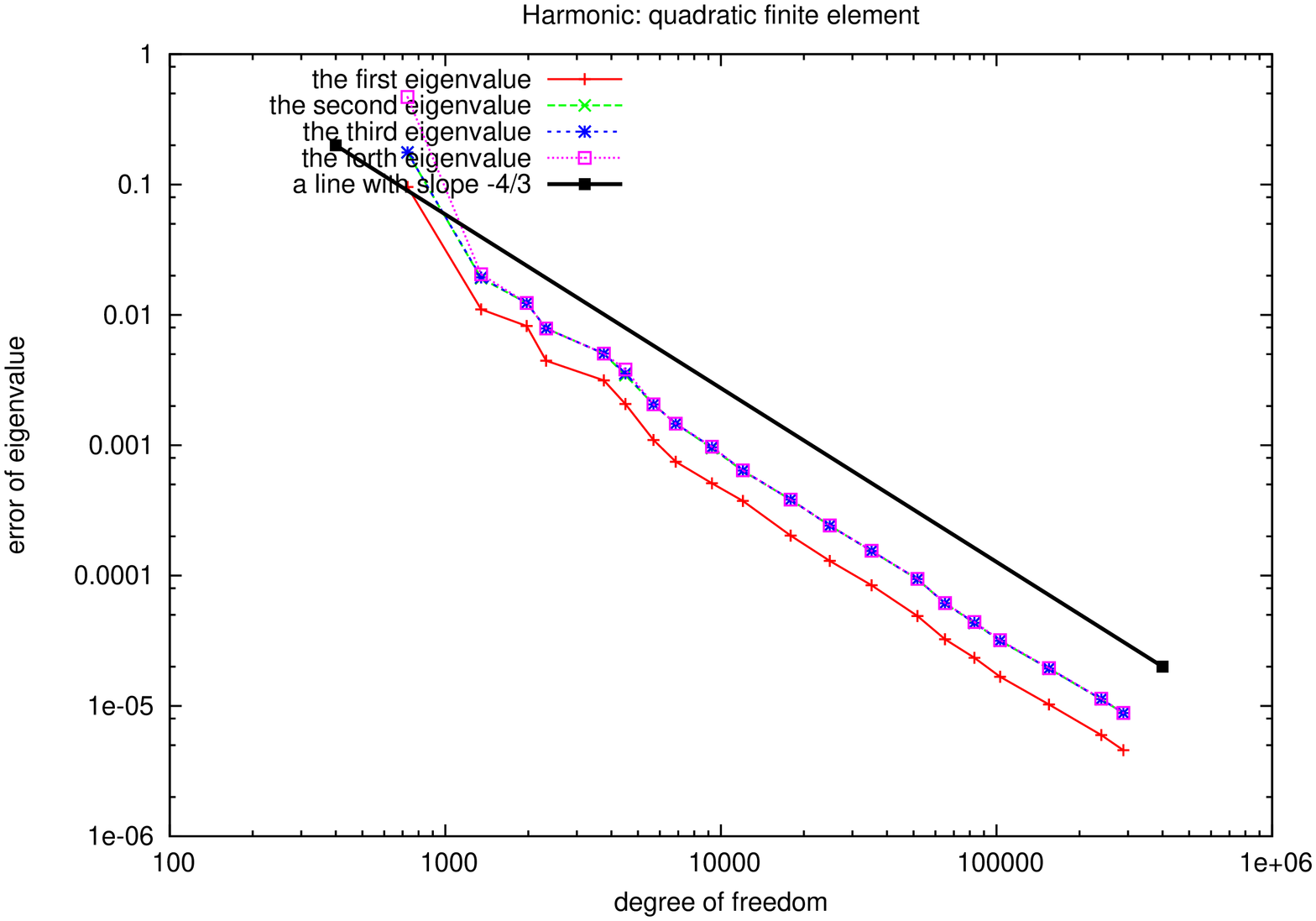}
  }
  \caption{The convergence curve of relative error for the eigenvalues of {\bf Example
1}}
  \label{figure-error-eigenvalue-harmonic}
\end{figure}


\begin{figure}[htb]
  \centering
  \subfloat[Linear finite elements]{
    \includegraphics[width=5cm,angle=-90]{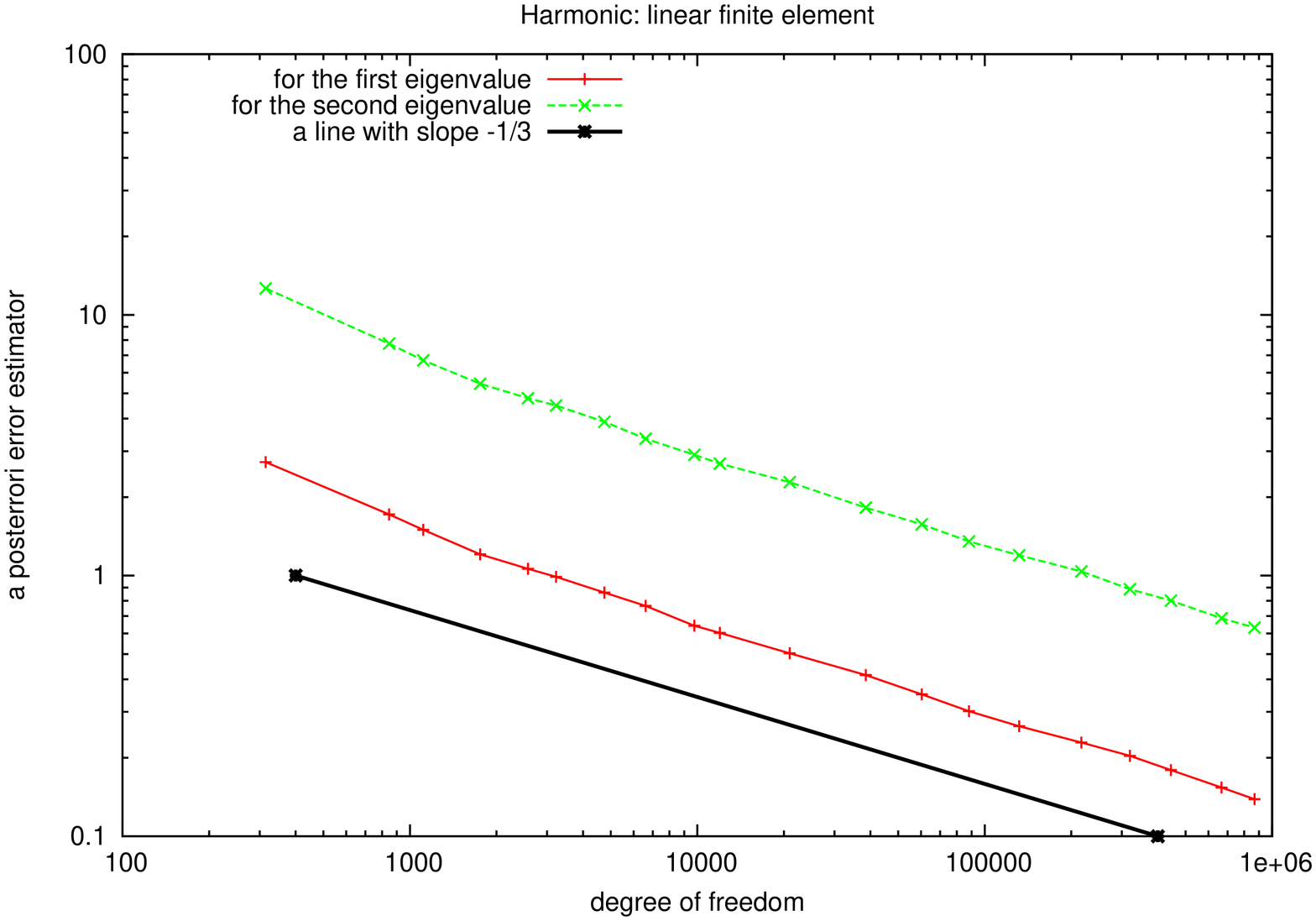}
  }
  \,
  \subfloat[Quadratic finite elements]{
    \includegraphics[width=5.0cm,angle=-90]{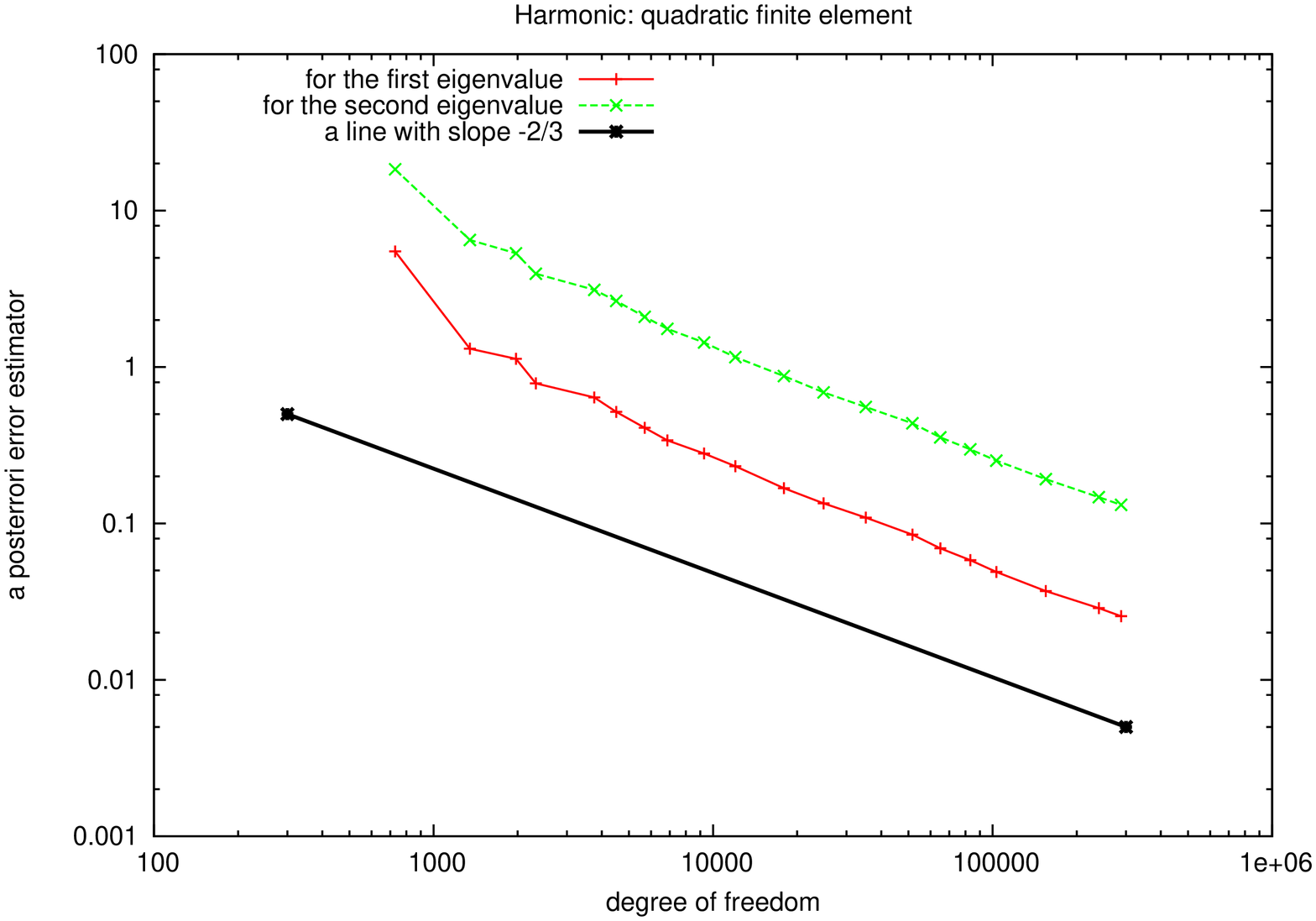}
  }
\caption{The convergence curves of $\eta_h(M_h(\lambda_1), \Omega)$ and $\eta_h(M_h(\lambda_2), \Omega)$ for {\bf Example 1}}
  \label{figure-error-harmonic}
\end{figure}





\vskip 0.2cm

{\bf Example 2} Consider the Schr{\" o}dinger equation for hydrogen atoms:
\begin{eqnarray}\Label{atom-H}
\left(-\frac {1}{2}\Delta  -\frac {1}{|x|}\right) u = \lambda u ~~\mbox{in}~ \mathbb{R}^3
\end{eqnarray}
with $\displaystyle\int_{\mathbb{R}^3}|u|^2dx=1$. The eigenvalues of (\ref{atom-H}) are $\lambda_n = -\frac{1}{2 n^2} (n= 1, 2, \cdots) $ and the
multiplicity of $\lambda_n$ is $n^2$ (see, e.g., \cite{greiner94}).

Since the eigenfunctions of (\ref{atom-H})  decay exponentially, instead of (\ref{atom-H}), we may solve the following eigenvalue problem:
find $(\lambda,u)\in \mathbb{R}\times H^1_0(\Omega)$ such that $\displaystyle\int_{\Omega}|u|^2dx=1$ and
\begin{equation}\label{atomH-numer}
\left\{\begin{array}{rcl} \displaystyle \left(-\frac {1}{2}\Delta -
\frac{1}{|x|}\right) u &=& \lambda u ~~\mbox{in} ~\Omega, \\
u &=& 0 ~~ \mbox{on}~\partial\Omega,
\end{array}\right.
\end{equation}
where $\Omega$ is some bounded domain in $\mathbb{R}^3$. In our computation, we choose $\Omega=(-20.0, 20.0)^3$ and find the first $2$ smallest
eigenvalue approximations and their associated eigenfunction space approximations. Since the multiplicity of the $n$-th smallest eigenvalue is $n^2$,
for the discrete problem of (\ref{atomH-numer}), we  calculate the first $5$ smallest eigenvalues and their associated eigenfunctions.

Figure \ref{atomH-grid} is the
cross-sections of the adaptively refined mesh constructed by {\bf D\"{o}rfler marking strategy}.
Similarly, we  see that for both the linear finite elements and quadratic finite elements, the mesh is much denser in the center of the domain where the solution oscillates quickly  than in the domain far away from the center where the solution is smooth. This means that the a posteriori error estimators we used are efficient.
\begin{figure}[htb]
  \centering
  \subfloat[Linear finite elements]{
    \includegraphics[width=5cm]{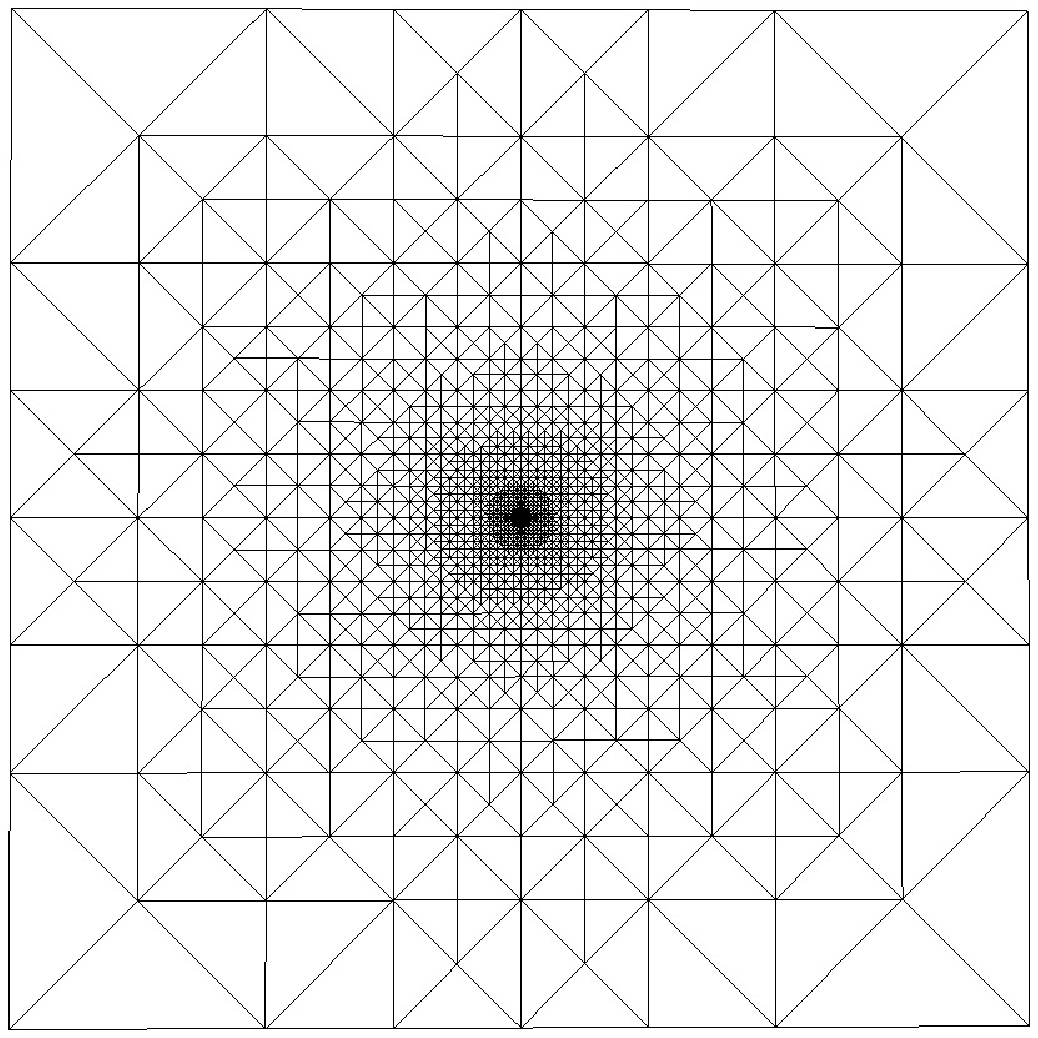}
  }
  \,
  \subfloat[Quadratic finite elements]{
    \includegraphics[width=5.0cm]{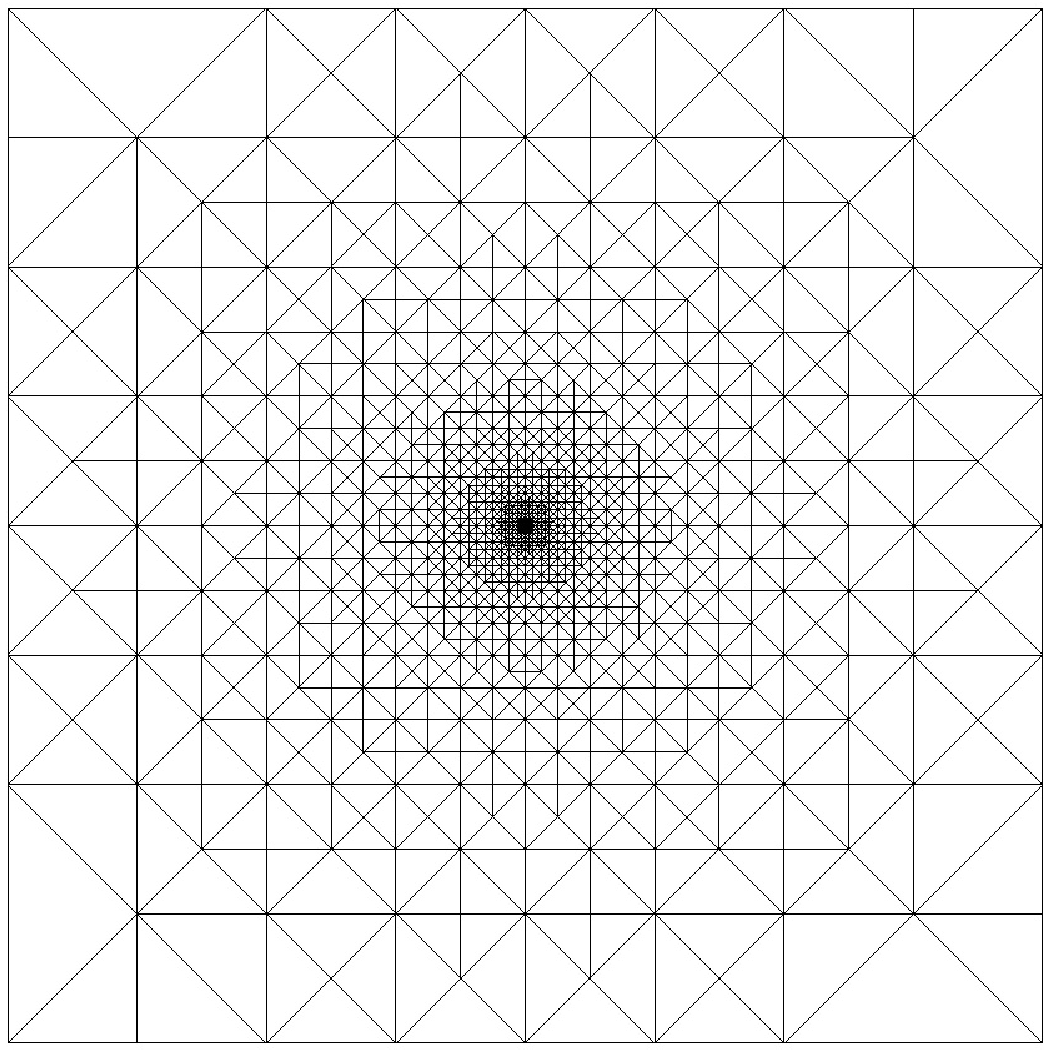}
  }
  \caption{The cross-section of an adaptive mesh of {\bf Example 2}}
  \label{atomH-grid}
\end{figure}

\begin{figure}[htb]
  \centering
  \subfloat[Linear finite elements]{
    \includegraphics[width=5cm,angle=-90]{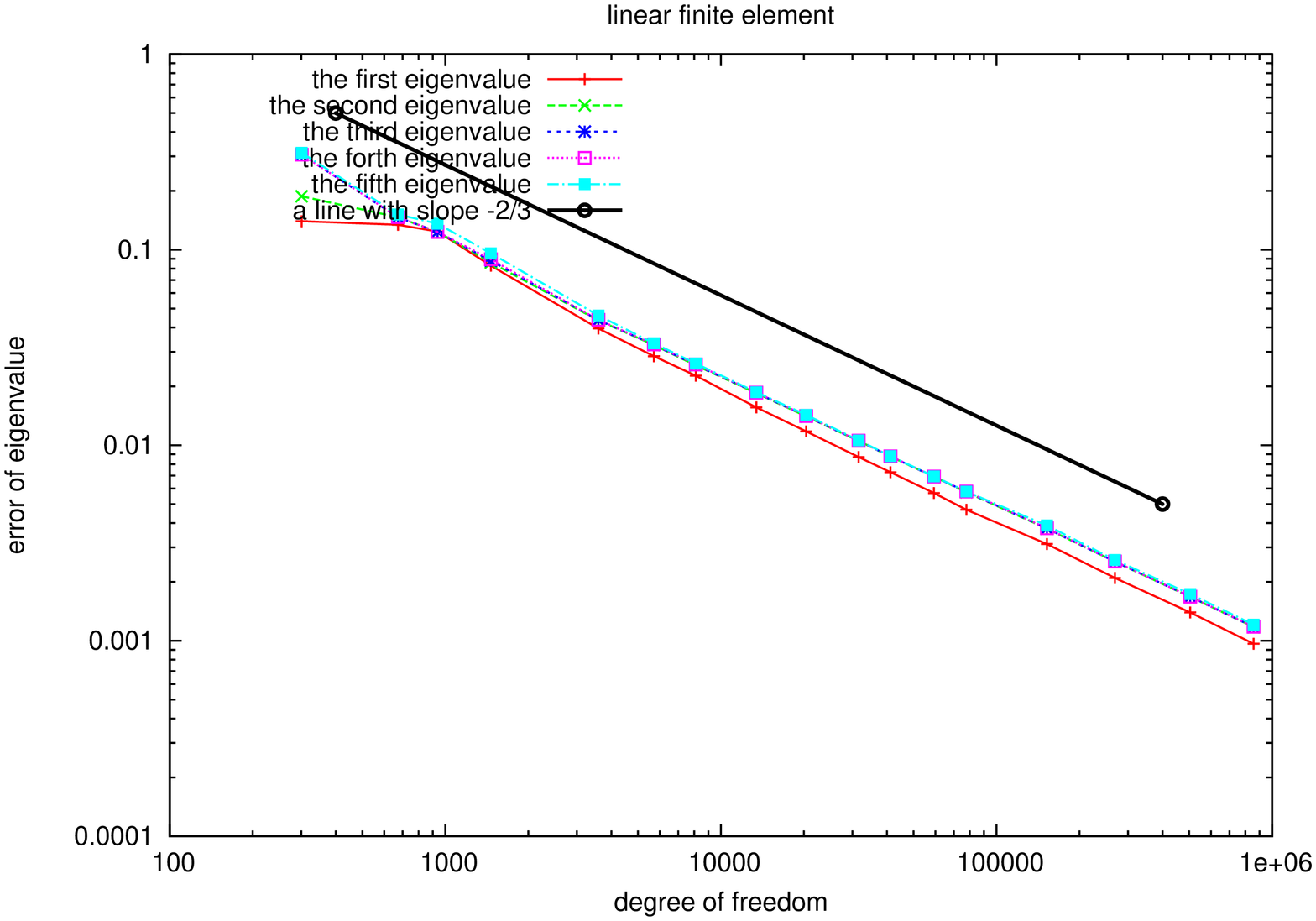}
  }
  \,
  \subfloat[Quadratic finite elements]{
    \includegraphics[width=5.0cm,angle=-90]{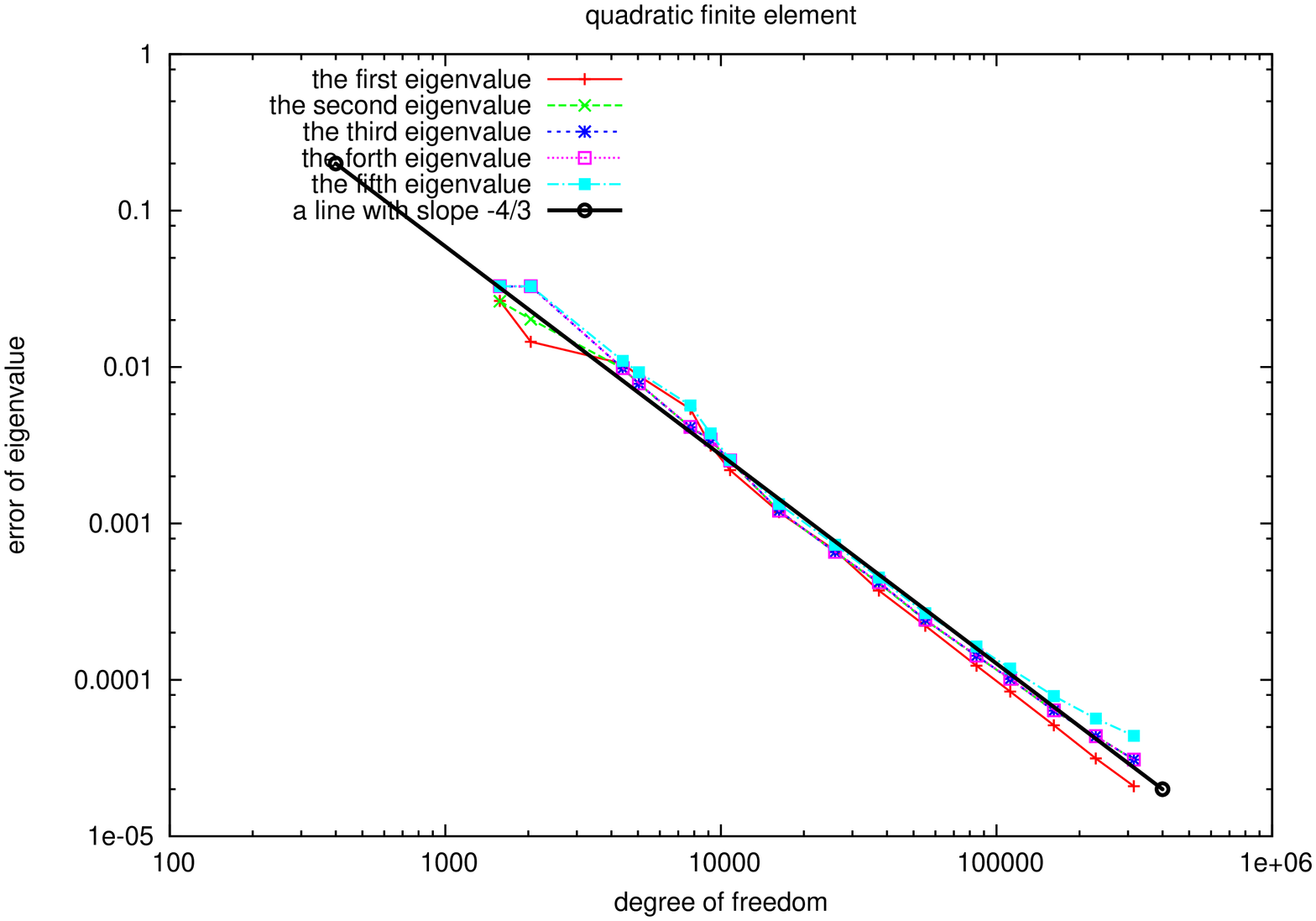}
  }
  \caption{The convergence curves of relative error for  eigenvalues of {\bf Example 2}}
  \label{figure-error-eigenvalue-atomH}
\end{figure}

\begin{figure}[htb]
  \centering
  \subfloat[Linear finite elements]{
    \includegraphics[width=5cm,angle=-90]{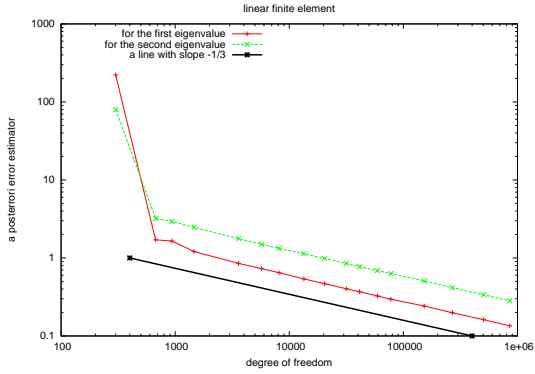}
  }
  \,
  \subfloat[Quadratic finite elements]{
    \includegraphics[width=5.0cm,angle=-90]{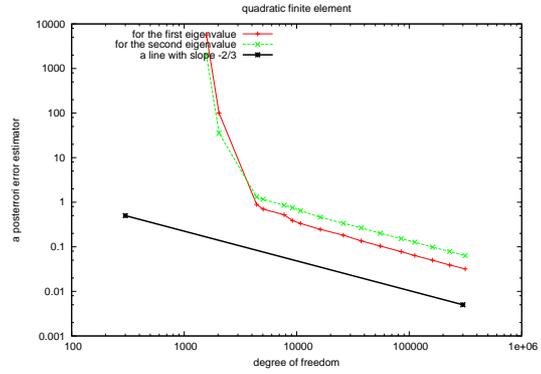}
  }
\caption{The convergence curves of $\eta_h(M_h(\lambda_1), \Omega)$ and $\eta_h(M_h(\lambda_2), \Omega)$ for {\bf Example 2}}
  \label{figure-error-atomH}
\end{figure}

The numerical results are presented in Figure \ref{figure-error-eigenvalue-atomH} and Figure \ref{figure-error-atomH}.
Similar to {\bf Example 1}, Figure \ref{figure-error-eigenvalue-atomH} shows that the convergence
curve for all the eigenvalues obtained by linear finite elements and quadratic finite elements
are parallel to the line with slope $-\frac{2}{3}$ and $-\frac{4}{3}$, respectively,
which means  all the eigenvalue approximations reach the optimal convergence rate for
 both linear finite element and quadratic finite element.
 Meanwhile,  we see from Figure \ref{figure-error-atomH} that convergence curve for  the
  a posteriori error estimators for eigenfunction space $\eta_h(M_h(\lambda_1))$
  and $\eta_h(M_h(\lambda_2))$ obtained by linear finite element are parallel to the line with slope $-\frac{1}{3}$,
  and those obtained by quadratic finite element are parallel to the line with slope  $-\frac{2}{3}$.
We observe that the approximation of eigenfunction space has
  also reached optimal convergent rate. These results validate our theoretical results.

{\bf Example 3} Consider the following eigenvalue problem which is defined in a non-convex domain: find $(\lambda,u)\in \mathbb{R}\times H^1_0(\Omega)$
such that $\displaystyle\int_{\Omega}|u|^2dx=1$ and
\begin{eqnarray}\Label{L-3d-problem}
\left\{\begin{array}{rcl} \displaystyle -\frac {1}{2}\Delta u &=& \lambda u ~~\mbox{in} ~\Omega, \\
u &=& 0 ~~ \mbox{on}~\partial\Omega,
\end{array}\right.
\end{eqnarray}
where $\Omega = (-5.0, 5.0)^3 \setminus (0, 5.0)^3$, see the left figure of Figure \ref{laplace-L-grid} below. We observe
from the numerical calculation that $\lambda_1 = 0.210651$ with multiplicity $1$ and $\lambda_2 = 0.331779$ with multiplicity $2$.

%
%

The surface of the adaptively refined meshes constructed by {\bf D\"{o}rfler marking strategy}   is shown in Figure  \ref{laplace-L-grid}.  Besides, some
cross-sections are displayed in Figure \ref{laplace-L-grid}. We see from these figures that for both linear finite elements and quadratic finite elements, the mesh is much denser along the lines where the solution is singular than in the domain far away from the singular lines. It indicates that our error estimator and marking strategy are efficient.

Our numerical results are listed in Figure \ref{figure-error-eigenval-laplace-L} and Figure \ref{figure-error-laplace-L}. Similar to {\bf Example 1} and {\bf Example 2}, we can also see that the approximations of eigenvalue as well as eigenfunction have reached optimal convergence rate, which coincides with our theory in Section \ref{convergence-sec} and Section \ref{complexity}.

\begin{figure}[htb]
  \centering
  \subfloat[nonconvex domain $\Omega$]{
    \includegraphics[width=4.5cm]{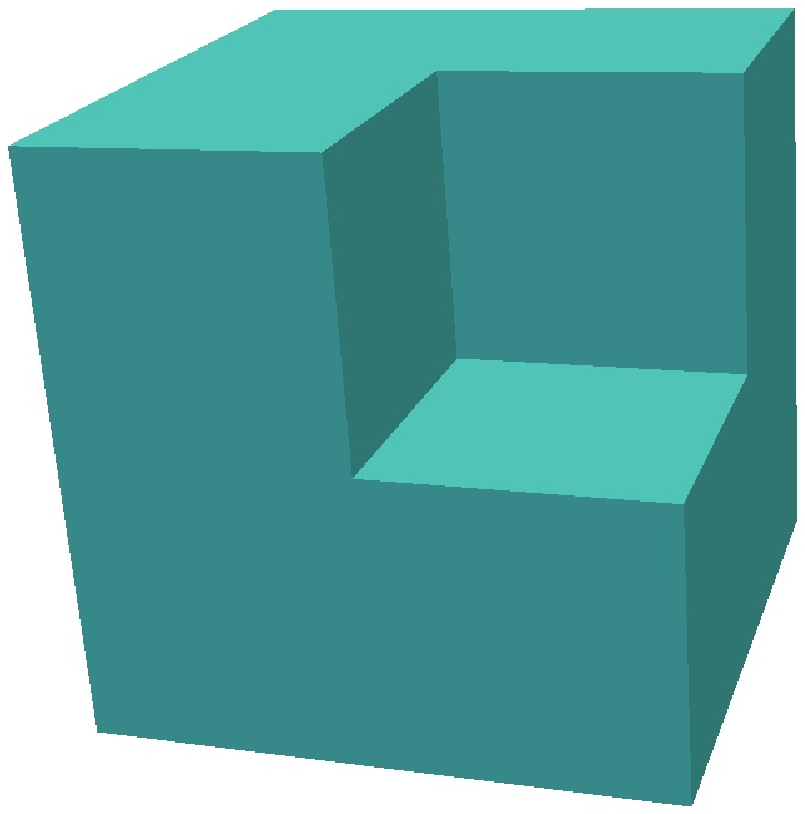}
  }
  \,
  \subfloat[Linear finite elements]{
    \includegraphics[width=4.5cm]{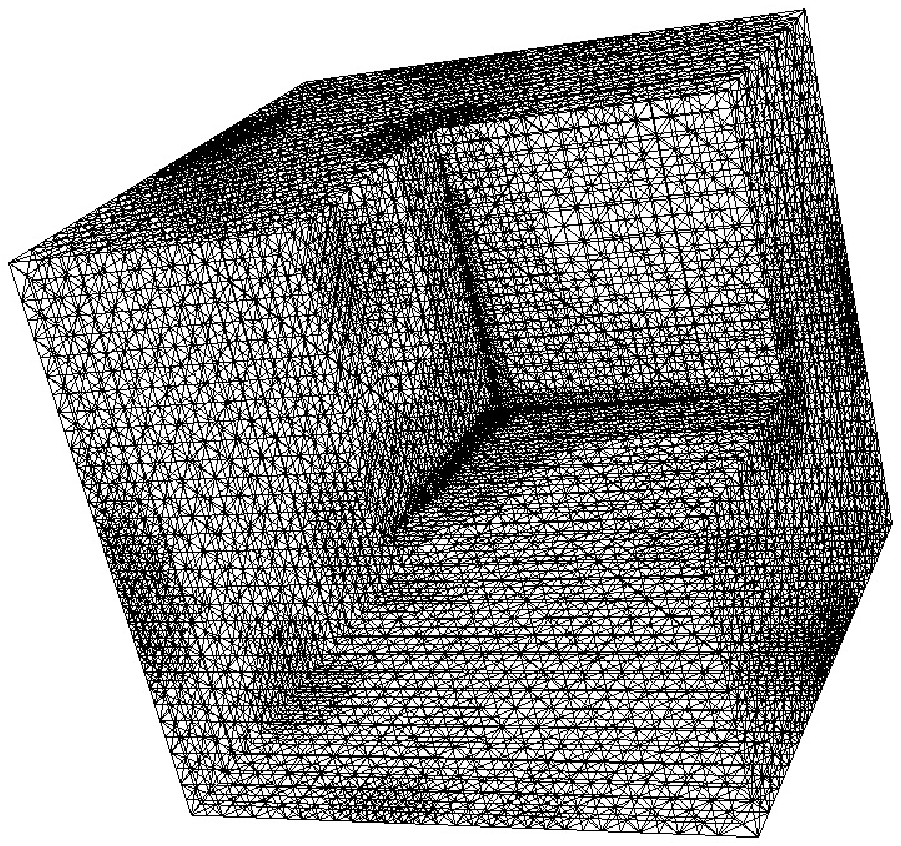}
  }
  \,
  \subfloat[Quadratic finite elements]{
    \includegraphics[width=4.5cm]{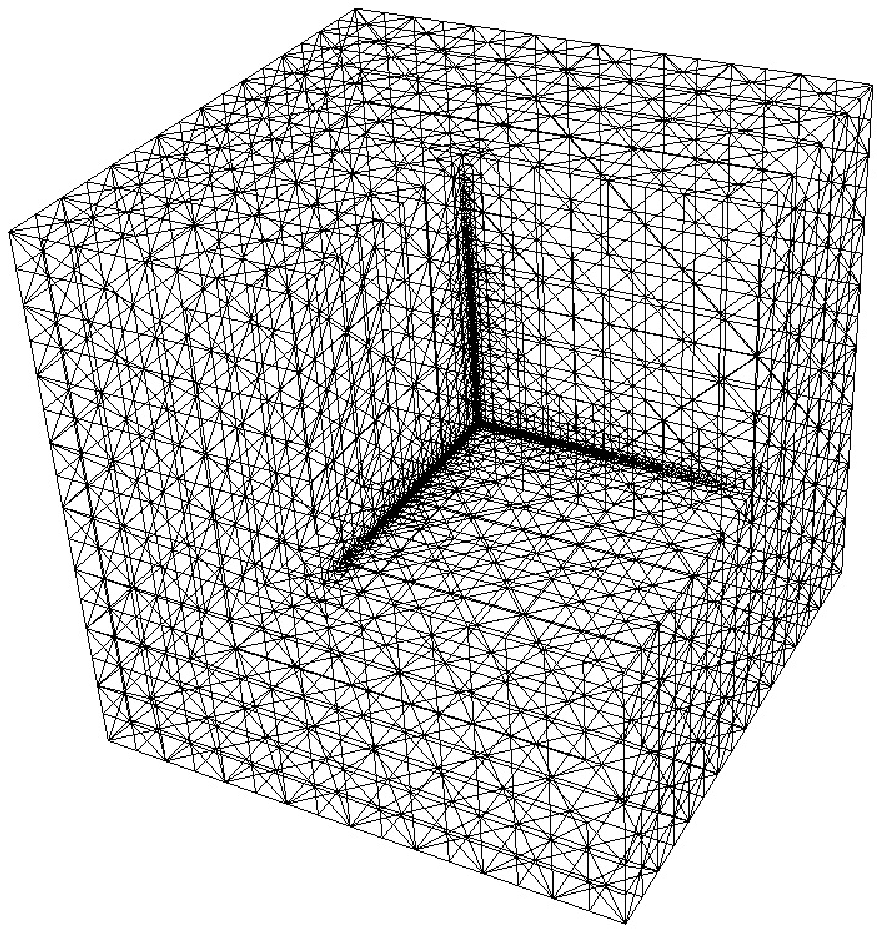}
  }
  \caption{The surface of an adaptive mesh of {\bf Example 3}}
  \label{laplace-L-grid}
\end{figure}

%

\begin{figure}[htb]
  \centering
  \subfloat[Linear finite elements]{
    \includegraphics[width=5cm,angle=-90]{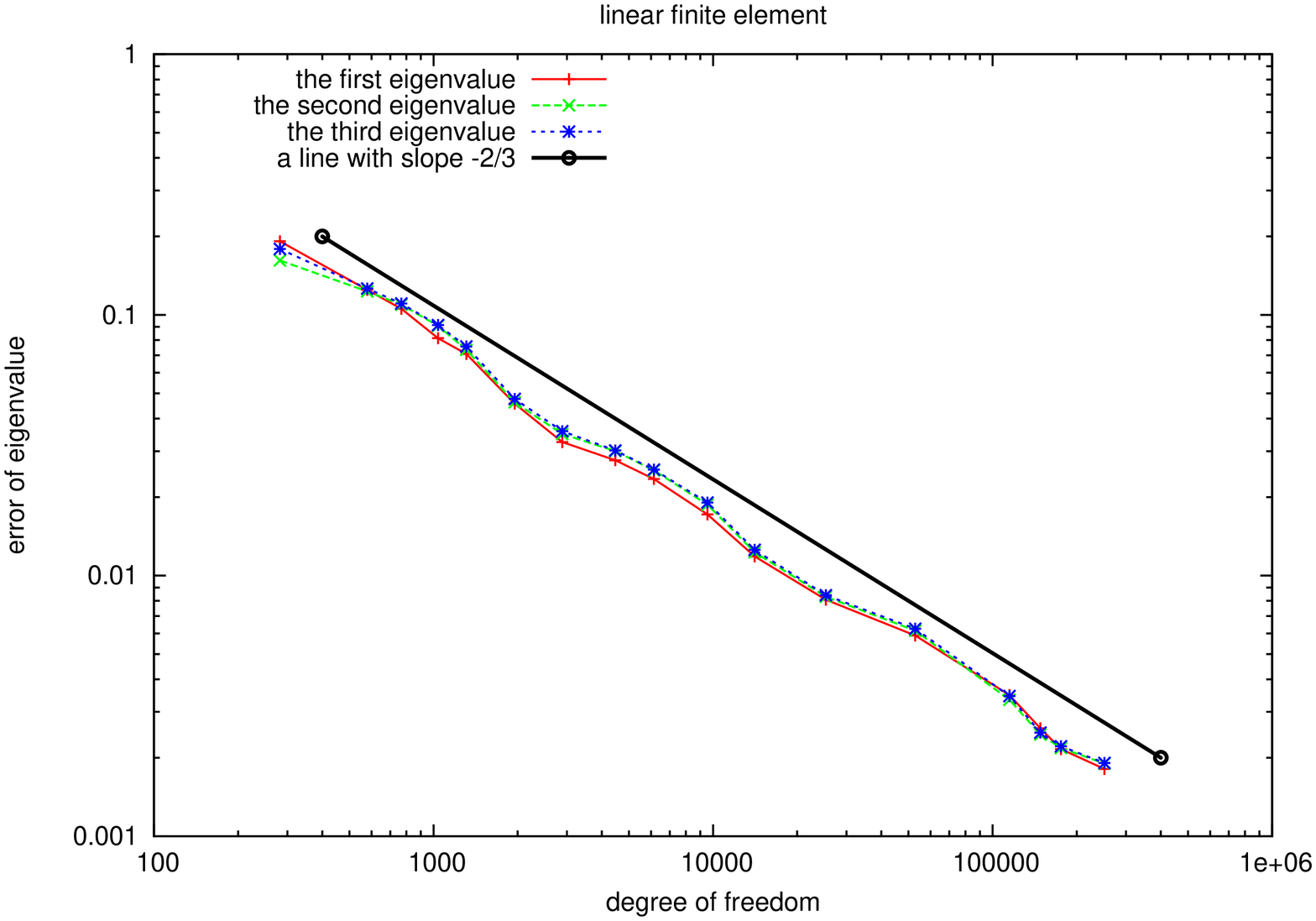}
  }
  \,
  \subfloat[Quadratic finite elements]{
    \includegraphics[width=5.0cm,angle=-90]{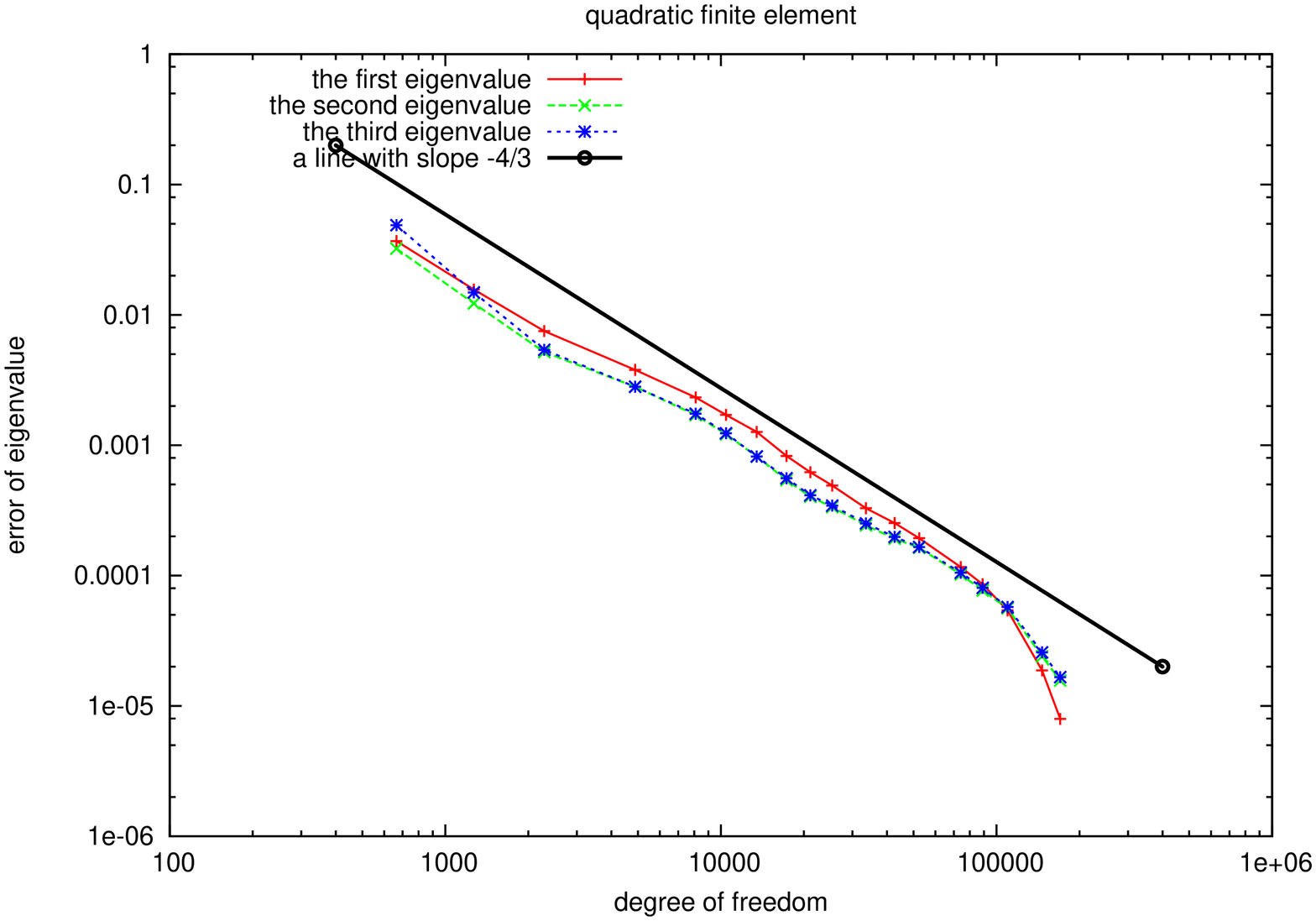}
  }
  \caption{The convergence curve of relative error for  eigenvalues of {\bf Example 3}}
  \label{figure-error-eigenval-laplace-L}
\end{figure}

\begin{figure}[htb]
  \centering
  \subfloat[Linear finite elements]{
    \includegraphics[width=5cm,angle=-90]{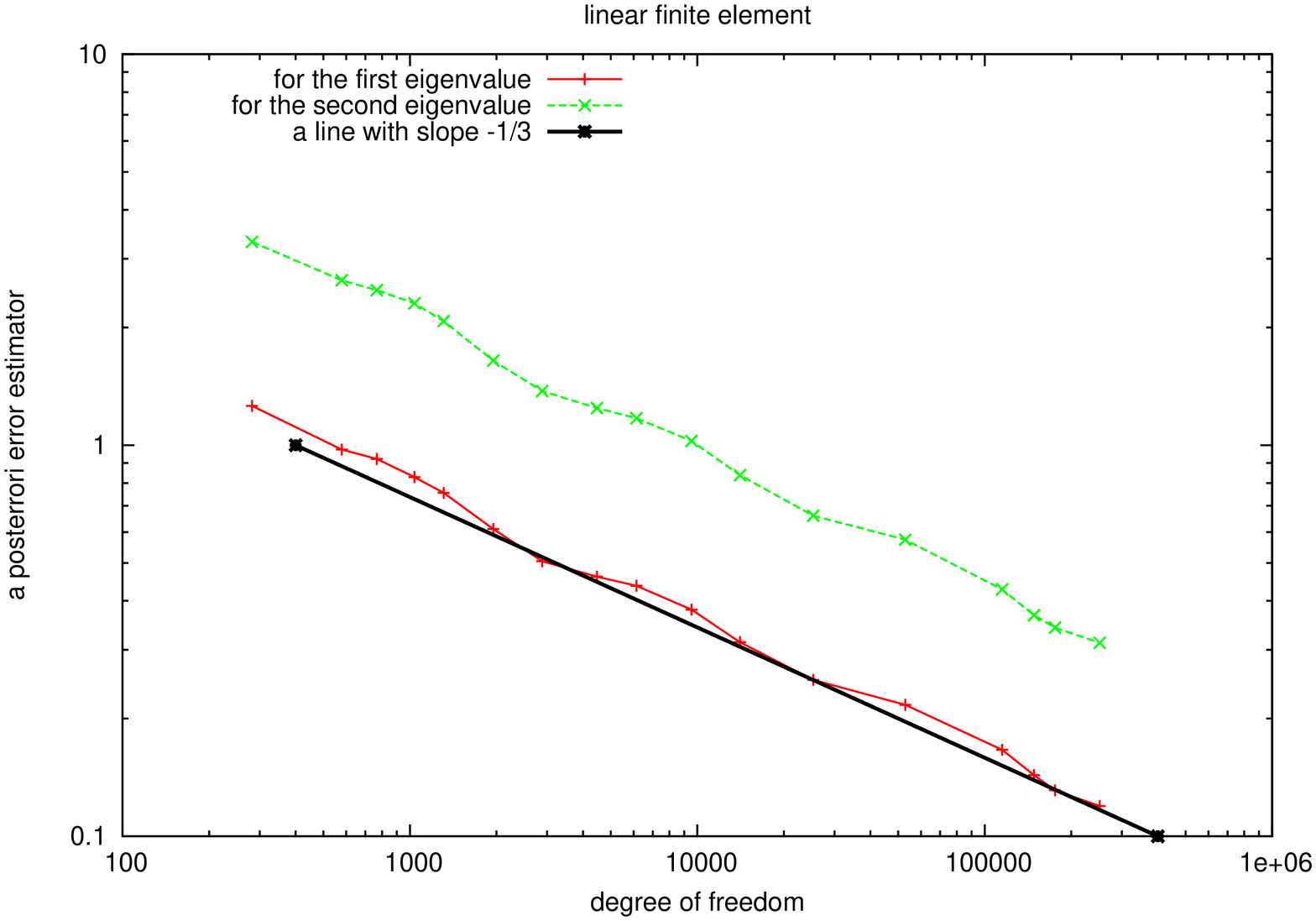}
  }
  \,
  \subfloat[Quadratic finite elements]{
    \includegraphics[width=5.0cm,angle=-90]{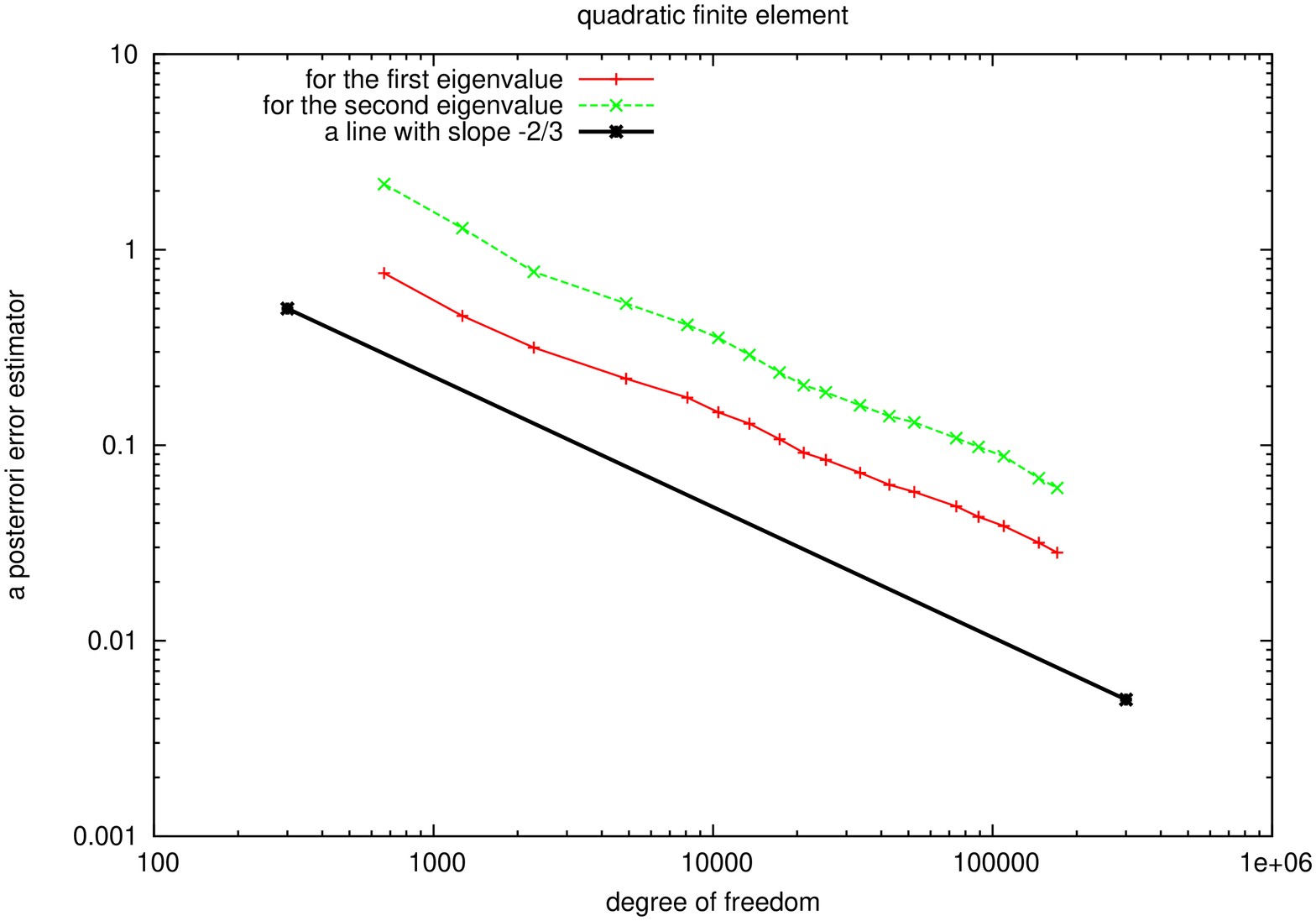}
  }
  \caption{The convergence curves of $\eta_h(M_h(\lambda_1), \Omega)$ and $\eta_h(M_h(\lambda_2), \Omega)$ for {\bf Example 3}}
  \label{figure-error-laplace-L}
\end{figure}

\section{Concluding remarks}\label{concluding-remark}\setcounter{equation}{0}
We have studied the convergence rate of an adaptive finite element algorithm for elliptic multiple eigenvalue problems.

\subsection{The first $N$ eigenvalues}
The adaptive algorithm for computing the multiple eigenvalue and eigenfunctions uses explicitly the
multiplicity of the eigenvalue, as well as its position in the ordered list of eigenvalues. More precisely, the index and the multiplicity of the eigenvalue must be known and provided to the
adaptive algorithm. However, for most of applications, we do not know these information, which makes the directly computation of a multiple eigenvalue and its corresponding eigenfunctions not realistic.

Fortunately, in most realistic computation, people usually need to solve the smallest $N$ eigenvalues,  or the largest $N$ eigenvalues, or the $N$ eigenvalues which is close to some special values. By some simple deduction, we can easily extend our results for some multiple eigenvalue $\lambda$ to this case.  We take the case of solving the smallest $N$ eigenvalues as an example to illustrate it, the other two cases are similar.

We consider the approximation for the  smallest $N$ eigenvalues of  (\ref{eigen}) and its corresponding eigenfunctions. If not accounting the multiplicity,  we assume the $N$ smallest eigenvalues belongs to the first $m$ eigenvalues with multiplicity of each eigenvalue being $q_i$, $\sum_{i=1}^{m-1} q_i  < N \leq \sum_{i=1}^{m} q_i$.
Let $(\lambda_{i, h}, u_{i, h}), i=1, \cdots, N$ be the first $N$ eigenpairs of (\ref{fe-eigen}).

Similar to the case of multiple eigenvalue, for $U_h = (u_{1, h}, u_{2, h}, \cdots, u_{N, h}) \in (V_h)^N$,
we  define
\begin{eqnarray*}
\eta^2_h(U_h, T)=\sum_{i=1}^N\eta^2_h(u_{i, h}, T) \quad \textnormal{and} \quad osc^2_h(U_h, T)=\sum_{i=1}^N osc^2_h(u_{i, h}, T), ~\forall T\in\mathcal{T}_h,
\end{eqnarray*}
and
\begin{eqnarray*}
\eta^2_h(U_h, \Omega)=\sum_{T\in\mathcal{T}_h} \eta^2_h(U_h, T) \quad \textnormal{and} \quad osc^2_h(U_h, \Omega)=\sum_{T\in\mathcal{T}_h} osc^2_h(U_h, T).
\end{eqnarray*}

Then, we design the adaptive finite element algorithm for solving the first smallest $N$ eigenpairs of (\ref{eigen}) as follows:

\vskip 0.1cm
\begin{algorithm}\label{algorithm-AFEM-eigen-2}~
 Choose a parameter $0 < \theta <1.$
\begin{enumerate}
\item Pick a given mesh $\mathcal{T}_0$, and let $k=0$.
\item Solve the system (\ref{fe-eigen}) on $\mathcal{T}_k$ to get the discrete solution  $(\lambda_{l, h_k}, u_{l, h_k})(l=1, \cdots, N)$.
\item Compute local error indictors $\eta_{h_k}(u_{h_{l, k}}, T)(l=1, \cdots, N)$ for all $T \in \mathcal{T}_{h_k}$.
\item Construct $\mathcal{M}_{h_k} \subset \mathcal{T}_{h_k}$ by {\bf D\"{o}rfler marking strategy} and parameter
 $\theta$.
\item Refine $\mathcal{T}_{h_k}$ to get a new conforming mesh $\mathcal{T}_{h_{k+1}}$ by Procedure {\bf REFINE}.
\item Let $k=k+1$ and go to 2.
\end{enumerate}
\end{algorithm}
\vskip 0.1cm
The {\bf D\"{o}rfler marking strategy} in the algorithm above is defined as follows.
\begin{center}
\begin{tabular}{|p{120mm}|}\hline
\begin{center}
{\bf D\"{o}rfler marking strategy}
 \end{center}\\
Given a parameter $0<\theta < 1$.
\begin{enumerate}
 \item Construct a   subset $\mathcal{M}_{h_k}$ of
 $\mathcal{T}_{h_k}$ by selecting some elements
 in $\mathcal{T}_{h_k}$ such that
\begin{eqnarray*}
 \sum_{T\in \mathcal{M}_{h_k}} \eta_{h_k}^2(U_{h_k}, T)  \geq \theta
  \eta_{h_k}^2(U_{h_k},\Omega).
 \end{eqnarray*}
 \item Mark all the elements in $\mathcal{M}_{h_k}$.
 \end{enumerate}\\
 \hline
\end{tabular}
\end{center}

We obtain the following results  from the similar arguments
in Section \ref{convergence-sec} and Section \ref{complexity}.
\begin{theorem}\label{thm-convergence-rate-eigenspace-2}
Let  $\lambda_1 < \lambda_2 \leq \cdots \leq \lambda_N $  be the first   $N$  eigenvalues of (\ref{eigen}),
   which belong to the first $m$ eigenvalues if not accounting the multiplicity, with multiplicity of each eigenvalue being $q_i$ and the corresponding eigenfunction space being $M(\lambda_i)$,  $\sum_{i=1}^{m-1} q_i  < N \leq \sum_{i=1}^{m} q_i$. Let
$\{(\lambda_{l, h_k}, u_{l, h_k}), l=1,\cdots,N\}_{k\in \mathbb{N}_0}$ be a sequence of finite element
 solutions
 produced by {\bf Algorithm \ref{algorithm-AFEM-eigen-2}}. Set $M_{h_k}(\lambda_i) = span\{u_{k_i + 1, h_k}, \cdots, u_{k_i + q_i, h_{k}}\}$, with $k_i = \sum_{j=1}^{i-1} q_j$.  Assume $N = \sum_{i=1}^{m} q_i$.   If $h_0\ll 1$,
 then there exists constant  $\alpha \in (0,1)$,
  depending only on the shape regularity of meshes, $ C_a $ and $c_a$, the parameter $\theta$ used by
 {\bf Algorithm \ref{algorithm-AFEM-eigen-2}}, such that
 the $k$-th iterate solution $(\lambda_{h_k, l}, u_{h_k, l}) (l=1, \cdots, q)$ of
 {\bf Algorithm \ref{algorithm-AFEM-eigen-2}} satisfies
\begin{eqnarray}\label{convergence-neq-2}
\delta^2_{H_0^1(\Omega)}(\mathcal{M}, \mathcal{M}_{h_k}) \lc \alpha^{2k},
  \end{eqnarray}
  where
  \begin{eqnarray*}
  \mathcal{M} = \left(  \begin{array}{c} M(\lambda_1) \\ M(\lambda_2) \\ \cdots \\ M(\lambda_m) \end{array} \right), ~~~~~
 \mathcal{M}_{h_k} = \left(  \begin{array}{c} M_{h_k}(\lambda_1) \\ M_{h_k}(\lambda_2) \\ \cdots \\ M_{h_k}(\lambda_m) \end{array} \right),
  \end{eqnarray*}
  and
\begin{eqnarray}
\delta^2_{H_0^1(\Omega)}(\mathcal{M}, \mathcal{M}_{h_k}) = \sum_{i=1}^{m} \delta^2_{H_0^1(\Omega)}(M(\lambda_i), M_{h_k}(\lambda_i)).
  \end{eqnarray}

\end{theorem}

\begin{theorem}\label{thm-optimal-complexity-eigenspace-2}
Let  $\lambda_1 < \lambda_2 \leq \cdots \leq \lambda_N $  be the first   $N$  eigenvalues of (\ref{eigen}),
   which belong to the first $m$ eigenvalues if not accounting the multiplicity, with multiplicity of each eigenvalue being $q_i$ and the corresponding eigenfunction space being $M(\lambda_i)$,  $\sum_{i=1}^{m-1} q_i  < N \leq \sum_{i=1}^{m} q_i$.  Let
$\{(\lambda_{l, h_k}, u_{l, h_k}), l=1,\cdots,N\}_{k\in \mathbb{N}_0}$ be a sequence of finite element
 solutions
 produced by {\bf Algorithm \ref{algorithm-AFEM-eigen-2}}. Set $M_{h_k}(\lambda_i) = span\{u_{k_i + 1, h_k}, \cdots, u_{k_i + q_i, h_{k}}\}$, with $k_i = \sum_{j=1}^{i-1} q_j$.
 If  Assumption \ref{assump-5.1} are satisfied for {\bf Algorithm \ref{algorithm-AFEM-eigen-2}} and $N = \sum_{i=1}^{m} q_i$, then the $n$-th iterate solution space $M_{h_n}(\lambda)$ of
 {\bf Algorithm \ref{algorithm-AFEM-eigen-2}} satisfies the quasi-optimal bound
 \begin{eqnarray*}
  \delta^2_{H_0^1(\Omega)}(\mathcal{M}(\lambda), \mathcal{M}_{h_n}(\lambda)) \lc (\#\mathcal{T}_{h_n}
  -\#\mathcal{T}_{h_0})^{-2s},
  \end{eqnarray*}
  provided $h_0\ll 1$,
 where the hidden constant depends on the exact solution $(\lambda, M(\lambda))$ and
 the discrepancy between
 $\theta$ and $\frac{1}{q^2}\frac{C_2^2 \gamma}{C_3^2 ( C_1^2 + (1 + 2 C_{\ast}^2 C_1^2)\gamma
)}$.
\end{theorem}

%

\subsection{Steklov eigenvalue problem}

Now we turn to address how to apply the same arguments to the Steklov problem
that consists in finding $\lambda\in \mathbb{R}$ and $u\neq 0$ such that
\begin{eqnarray*}\Label{steklov-problem}
\left\{\begin{array}{rl}
-\nabla\cdot (A\nabla u) +cu &= 0 \quad \mbox{in} \quad \Omega,\\[0.2cm]
(A\nabla u)\cdot \overrightarrow{n} &= \lambda uv\quad\mbox{on}~~\partial\Omega,\end{array}\right.
\end{eqnarray*}
where
 $\overrightarrow{n}$ is the outward unit normal vector of $\Omega$ on $\partial \Omega$.

 We set $a(\cdot, \cdot) =(A \nabla \cdot, \nabla \cdot)_{\Omega} + (c \cdot, \cdot)_{\Omega}$,
 $b(\cdot,\cdot)=(\cdot,\cdot)_{\partial \Omega}$
 and consider the
non-homogeneous Neumann problem as a model problem as follows:
\begin{equation}\Label{problem-neumann}
\left\{\begin{array}{rl}
L u_i &= 0 \,\,\, \mbox{in} \quad \Omega, ~~i=1, \cdots, N,\\[0.2cm]
(A\nabla u_i)\cdot \overrightarrow{n} &= f_i\,\,\,\mbox{on}~~\partial\Omega. \end{array}\right.
\end{equation}
Define the element residual $\tilde{\mathcal{R}}_T(u_{i,h})$ and the jump residual $\tilde{J}_E(u_{i,h})$ for \eqref{problem-neumann} as follows:
 \begin{eqnarray*}\label{residual-neumann}
  \tilde{\mathcal{R}}_T(u_{i,h}) &=& \nabla\cdot(A\nabla u_{i,h})-c u_{i,h}\quad \mbox{in}~ T\in
  \mathcal{T}_h,\\
  \tilde{J}_E(u_{i,h}) &=& \left\{ \begin{array}{rcl} [[A\nabla u_{i,h}]]_E \cdot \nu_E  \quad \mbox{on}~ E\in
  \mathcal{E}_h\\[1ex]
\displaystyle A\nabla u_{i,h} \cdot \overrightarrow{n} -f_i  \quad \mbox{on}~E\in  \mathcal{S}_{h},\end{array} \right.
\end{eqnarray*}
where $\mathcal{S}_h$ denote the set of boundary faces. For $T\in \mathcal{T}_h$, we denote the local error indicator
  $\tilde{\eta}_h(u_{i,h}, T)$ by
  \begin{eqnarray*}\label{error-indicator}
   \tilde{\eta}^2_h(u_{i,h}, T) =h_T^2\|\tilde{\mathcal{R}}_T(u_{i,h})\|_{0,T}^2
   + \sum_{E\in \mathcal{E}_h\cup \mathcal{S}_h,E\subset\partial T
   } h_E \|\tilde{J}_E(u_{i,h})\|_{0,E}^2
  \end{eqnarray*}
and the oscillation $\widetilde{osc}_h(u_{i,h},T)$ by
\begin{eqnarray}\label{local-oscillation}
\widetilde{osc}^2_h(u_{i,h},T) = h_T^2\|\tilde{\mathcal{R}}_{T}(u_{i,h})-\overline{\tilde{\mathcal{R}}_{T}(u_{i,h})}\|_{0,T}^2
   + \sum_{E\in \mathcal{E}_h\cup \mathcal{S}_h,E\subset\partial T
   } h_E \|\tilde{J}_E(u_{i,h})-\overline{\tilde{J}_E(u_{i,h})}\|_{0,E}^2.
\end{eqnarray}
We see that Lemma \ref{lemma-osc-L-neq} is also valid for \eqref{local-oscillation}.

In context of Steklov eigenvalue problems, we define
 \begin{eqnarray*}
  \mathcal{R}_T(E_h u) &=& \nabla\cdot(A\nabla E_h u)-c E_h u~~ \mbox{in}~ T\in
  \mathcal{T}_h,\label{residual-eigen-multi-steklov}\\
  J_E(E_h u) &=&
\left\{ \begin{array}{rcl} [[A\nabla E_h u]]_E \cdot \nu_E  \quad \mbox{on}~ E\in
  \mathcal{E}_h\\[1ex]
\displaystyle A\nabla E_h u \cdot \overrightarrow{n} -\lambda^h E_hu \quad \mbox{on}~E\in  \mathcal{S}_{h},\end{array} \right.
\end{eqnarray*}
where $\overrightarrow{n}$ denotes the outward unit normal vector on $E\in \mathcal{S}_h$. For $T\in \mathcal{T}_h$, we define the local error indicator
  $\eta_h(E_h u, T)$ by
  \begin{eqnarray*}\label{error-local-eigen-multi}
   \eta^2_h(E_h u,T) = h_T^2\|\mathcal{R}_T(E_h u)\|_{0,T}^2 + \sum_{E\in \mathcal{E}_h\cup \mathcal{S}_h,E\subset\partial T
   } \|J_E(E_h u)\|_{0,E}^2
  \end{eqnarray*}
and the oscillation $osc_h(E_h u,T)$ by
\begin{eqnarray*}\label{osc-local-eigen-multi}
osc^2_h(E_h u,T) = h_T^2\|\mathcal{R}_{T}(E_h u)-\overline{\mathcal{R}_{T}(E_h u)}\|_{0,T}^2
   + \sum_{E\in \mathcal{E}_h\cup \mathcal{S}_h,E\subset\partial T
   } h_E \|J_E(E_h u)-\overline{J_E(E_h u)}\|_{0,E}^2.
\end{eqnarray*}
We obtain by  using the same argument  that Theorem \ref{thm-error-estimator-space},  Theorem \ref{error-reduction},  Theorem \ref{thm-convergence-rate-eigenspace-2}, and  Theorem \ref{thm-optimal-complexity-eigenspace-2} are valid for the Steklov problem with
multiple eigenvalues.\vskip 0.2cm

\subsection{The inexact numerical solutions}
In our numerical analysis above, for convenience, we assume that the algebraic eigenvalue problem is exactly solved and the numerical integration is exact. Indeed, the same conclusion can be expected if all the numerical errors are taken into account, including both the error resulting from the inexact solving of the algebraic eigenvalue problem and the error coming from the inexact numerical integration. Suppose $(\lambda, u)$ is an eigenpair with the multiplicity of $\lambda$ being $q$, the exact solution on mesh $\mathcal{T}_k$
 are $\{(\lambda_{h, i}, u_{h, i})\}_{i=1}^{q}$, and the   the solution considering the numerical error are $\{(\hat{\lambda}_{h, i}, \hat{u}_{h, i})\}_{i=1}^{q}$. If the numerical errors resulting from the solution of algebraic system and the numerical integration are small enough, say, satisfy
\begin{eqnarray*}
\sum_{i = 1}^{q}\big(\|u_{h, i} - \hat{u}_{h, i}\|_{a}^2 + |\lambda_{h, i} - \hat{\lambda}_{h, i}|\big) \lc r(h_0) \sum_{i = 1}^{q} \eta_{h}^2(\hat{u}_{h, i}, \Omega)
\end{eqnarray*}
with $r(h_0) \ll 1$ for $h_0 \ll 1$, then we have from the following triangle inequality
\begin{eqnarray*}
 \| u_i  - \hat{u}_{h, i}\|_a \leq   \| u_i  - u_{h, i}\|_a +  \| u_{h, i}  - \hat{u}_{h, i}\|_a,
 \end{eqnarray*}
 that  our main results obtained in this paper hold true for inexact algebraic solution and inexact numerical integration, too.

\end{document}